\documentclass[11pt]{amsart} 
\usepackage{amsmath}
\usepackage{amsthm}
\usepackage{amssymb}
\usepackage[mathscr]{euscript}
\usepackage{amsmath}
\usepackage{amsthm}
\usepackage{amssymb}
\usepackage{bm}
\usepackage{fancyhdr}
\usepackage{enumerate}
\usepackage{amscd}
\usepackage[all]{xy}
\usepackage{graphicx}
\usepackage{color}
\usepackage{hyperref}
\hypersetup{
    colorlinks,
    citecolor=black,
    filecolor=black,
    linkcolor=black,
    urlcolor=black
}
\usepackage{tikz}
\usepackage[all]{xy}
\usepackage{graphicx}
\usetikzlibrary{backgrounds}
\usepackage{subfiles} 
\usepackage{soul} 

\theoremstyle{plain}
\newtheorem{lemma}{Lemma}[section]

\newtheorem{proposition}[lemma]{Proposition}

\newtheorem{theorem}[lemma]{Theorem}
\newtheorem{Theorem}{Theorem}

\newtheorem{Example}[Theorem]{Example}

\newtheorem{corollary}[lemma]{Corollary}
\newtheorem*{vacuum-conjecture}{Vacuum stationary conjecture}
\newtheorem*{conjecture}{Bartnik's stationary conjecture}

\numberwithin{equation}{section}

\theoremstyle{remark}

\theoremstyle{definition}
\newtheorem{remark}[lemma]{Remark}
\newtheorem{example}[lemma]{Example}

\newtheorem{definition}[lemma]{Definition}
\newtheorem{Definition}[Theorem]{Definition}

\DeclareMathOperator{\tr}{tr} 
\DeclareMathOperator{\Div}{div}
\DeclareRobustCommand\jksub{jk}

 \DeclareMathOperator{\Det}{det}
 \DeclareMathOperator{\Int}{Int}

 \def\Sc{\sigma}

\newenvironment{enumeratei}
{\begin{enumerate}[\upshape (i)]}{\end{enumerate}}

\usepackage[tmargin=1in,bmargin=1in,lmargin=1.5in,rmargin=1.5in]{geometry}
\begin{document}
\title[Bartnik minimizers and improvability] 
{Bartnik mass minimizing initial data sets and improvability of the dominant energy scalar}

\author{Lan-Hsuan Huang}
\address{Department of Mathematics, University of Connecticut, Storrs, CT 06269, USA}
\email{lan-hsuan.huang@uconn.edu}
\author{Dan A. Lee}
\address{65-30 Kissena Boulevard, Department of Mathematics, Queens College, Queens, NY 11367, USA}
\email{dan.lee@qc.cuny.edu}
\maketitle

\begin{abstract}
We introduce the concept of improvability of the dominant energy scalar, and we derive strong consequences of non-improvability. In particular, we prove that a non-improvable initial data set without local symmetries must sit inside a null perfect fluid spacetime carrying a global Killing vector field. We also show that the dominant energy scalar is always \emph{almost} improvable in a precise sense. Using these main results, we provide a characterization of Bartnik mass minimizing initial data sets which makes substantial progress toward Bartnik's stationary conjecture.

Along the way we observe that in dimensions greater than eight there exist pp-wave counterexamples (without the optimal decay rate for asymptotically flatness) to the equality case of the spacetime positive mass theorem. As a consequence, we find counterexamples to Bartnik's stationary and strict positivity conjectures in those dimensions.
\end{abstract}

\tableofcontents

\section{Introduction}\label{section:introduction}
Let $n\ge 3$ and let $U$ be a connected $n$-dimensional smooth manifold, $g$ a Riemannian metric, and $\pi$ a symmetric $(2,0)$-tensor.  
We refer to such a triple $(U, g, \pi)$ as an \emph{initial data set} and $\pi$ as the \emph{conjugate momentum tensor}. It is related to the usual $(0,2)$-tensor $k$ via the equation
\begin{align}\label{equation:k}
 k_{ij}:=g_{i\ell}g_{jm} \pi^{\ell m} -\tfrac{1}{n-1}(\tr_g \pi)g_{ij}.
 \end{align}
We say that $(U, g, \pi)$ \emph{sits inside} a Lorentzian spacetime $(\mathbf{N}, \mathbf{g})$ of one higher dimension if $(U, g)$ isometrically embeds into $(\mathbf{N}, \mathbf{g})$ with $k$ as its second fundamental form.

One can define the energy and current densities $\mu$ and $J$ in terms of $g$ and $\pi$, see \eqref{equation:mass-current}, and then we say that $(g,\pi)$ satisfies the \emph{dominant energy condition, or DEC,} if $\mu \ge |J|_g$. This condition can be recast as nonnegativity of the quantity
\[ \Sc(g,\pi) := 2(\mu-|J|_g),\]
which we will call the \emph{dominant energy scalar}. Note that in the \emph{time-symmetric case} where $\pi \equiv 0$, $\Sc(g,\pi)$ is just the scalar curvature of $g$, denoted $R_g$. In this paper we will study the question of when one can use compactly supported perturbations of $(g,\pi)$ to increase $\Sc$. This flexibility is important for situations in which one wants to find perturbations that preserve (or reinstate) the DEC.

Consider the following theorem about prescribing compactly supported perturbations of scalar curvature, which is a slight variant of Theorem 1 of~\cite{Corvino:2000}. 
\begin{Theorem}[Corvino]\label{theorem:scalar}
Let $(\Omega, g)$ be a compact Riemannian manifold with nonempty smooth boundary, such that $g\in C^{4,\alpha}(\Omega)$. Assume that $DR|_g^*$ is injective on $\Int\Omega$, where 
$DR|_g^*$ denotes the adjoint of the linearization of the scalar curvature operator at $g$. 
Then there exists a $C^{4,\alpha}({\Omega})$ neighborhood $\mathcal{U}$ of $g$ such that for every $V\subset\subset \Int\Omega$, there exist constants $\epsilon, C>0$ such that the following statement holds: For $\gamma\in \mathcal{U}$ and  $u\in C^{0,\alpha}_c(V)$ with $\| u \|_{C^{0,\alpha}(\Omega)}< \epsilon$, there exists $h\in C^{2,\alpha}_c(\Int\Omega)$ with $\|h\|_{C^{2,\alpha}(\Omega)}\le C\| u\|_{C^{0,\alpha}(\Omega)}$ such that the metric  $\gamma+h$ satisfies
\[
	R_{\gamma+h} =R_\gamma + u.
\]
\end{Theorem}
Since $DR|_g^*$ is heavily overdetermined, the existence of a nontrivial kernel element $f$ is a highly special circumstance; for example, if $f$ is positive on $\Int\Omega$, then $g$ has constant scalar curvature and $(\Int\Omega, g, 0)$ sits inside a static spacetime.


The \emph{constraint operator} on initial data,
\begin{equation}\label{equation:constraint}
 \Phi(g,\pi):=(2\mu, J),
 \end{equation}
may be thought of as a generalization of the scalar curvature operator on metrics. Extending Theorem~\ref{theorem:scalar}, Corvino and Schoen
proved that injectivity of $D\Phi|_{(g,\pi)}^*$ allows one to prescribe small, compactly supported perturbations of $(\mu, J)$ using compactly supported perturbations of $(g,\pi)$~\cite[Theorem 2]{Corvino-Schoen:2006}. While this theorem is useful when dealing with \emph{vacuum} initial data sets (those with $\mu=|J|_g=0$),
it is less useful for dealing with the DEC. The reason for this is that even if you can prescribe $\mu$ and $J$ exactly, that is not enough to know whether  $\mu\ge |J|_g$ holds, because you lose control over $g$. 

The following definition captures the concept of being able to increase~$\Sc$ using compact perturbations.
\begin{Definition}
Let $(\Omega, g, \pi)$ be a compact initial data set with nonempty smooth boundary, such that $(g,\pi)\in C^{4,\alpha}(\Omega)\times C^{3,\alpha}(\Omega)$.
We say that \emph{the dominant energy scalar is improvable in $\Int\Omega$}  if there is a $C^{4,\alpha}({\Omega})\times C^{3,\alpha}({\Omega})$ neighborhood $\mathcal{U}$ of $(g, \pi)$, such that for any $V\subset\subset \Int\Omega$, there exist constants $\epsilon, C>0$ such that the following statement holds: For $(\gamma, \tau)\in \mathcal{U}$ and  $u\in C^{0,\alpha}_c(V)$ with $\| u \|_{C^{0,\alpha}(\Omega)}< \epsilon$, there exists $(h, w)\in C^{2,\alpha}_c(\Int\Omega)$ with $\|(h,w)\|_{C^{2,\alpha}(\Omega)}\le C\| u\|_{C^{0,\alpha}(\Omega)}$ such that the initial data  $(\gamma+h, \tau+w)$ satisfies
\[
	\Sc (\gamma+h, \tau+w) \ge \Sc(\gamma,\tau) + u.
\]
\end{Definition}

Corvino and the first author~\cite{Corvino-Huang:2020}   introduced the \emph{modified constraint operator at $(g,\pi)$}, denoted~$\overline{\Phi}_{(g,\pi)}$, to deal with improvability of the dominant energy scalar, by expanding upon conformal-type perturbations discovered in \cite[Section 6]{Eichmair-Huang-Lee-Schoen:2016}.  On initial data $(\gamma, \tau)$,
\begin{align*}
	\overline{\Phi}_{(g,\pi)}(\gamma,\tau) :=\Phi(\gamma, \tau) +(0, \tfrac{1}{2} \gamma\cdot J).
\end{align*}
The following is a slight variant of Theorem 1.1 of~\cite{Corvino-Huang:2020}. (See also~\cite{Corvino-Schoen:2006, Chrusciel-Delay:2003} for the vacuum case.)

 \begin{Theorem}[Corvino-Huang]\label{theorem:no-kernel}
Let $(\Omega, g, \pi)$ be a compact initial data set with nonempty smooth boundary, such that $(g,\pi)\in C^{4,\alpha}(\Omega)\times C^{3,\alpha}(\Omega)$.
Let $D\overline{\Phi}_{(g,\pi)}^*$ denote adjoint of the linearized operator $\left.D\overline{\Phi}_{(g,\pi)}\right|_{(g,\pi)}$, and assume that it is injective on $\Int\Omega$. Then the dominant energy scalar of $(g, \pi)$ is improvable  in $\Int\Omega$.
\end{Theorem}
The domain of $D\overline{\Phi}_{(g,\pi)}^*$ consists of pairs $(f, X)$ such that $f$ is a $C^2_{\mathrm{loc}}$ scalar function and $X$ is a $C^1_{\mathrm{loc}}$ vector field. 
 We will refer to these pairs $(f, X)$ as \emph{lapse-shift pairs}. In the case where $(g, \pi)$ happens to be vacuum, there is a nice interpretation of the kernel. (Note that in this case, $\overline{\Phi}_{(g,\pi)}=\Phi$.) 
\begin{Theorem}[Moncrief~\cite{Moncrief:1975}, cf.~\cite{Fischer-Marsden-Moncrief:1980}]\label{theorem:Moncrief}
Let $(U, g, \pi)$ be a vacuum initial data set such that $(g,\pi)\in C^{3}_{\mathrm{loc}}(U)\times C^{2}_{\mathrm{loc}}(U)$, and suppose that there exists a nontrivial lapse-shift pair $(f, X)$ on $U$ solving 
\[ D{\Phi}|_{(g, \pi)}^*(f, X) =0.\]
Then $(U, g, \pi)$ sits inside a vacuum spacetime admitting a unique global Killing vector field $\mathbf{Y}$ such that $\mathbf{Y}= 2f \mathbf{n} + X$ along~$U$, where $\mathbf{n}$ is the future unit normal to $U$.\footnote{The factor of $2$ in front of $f$ is due to the factor of $2$ occurring in the definition of $\Phi$ in~\eqref{equation:constraint} and will appear in many places throughout the paper for this reason.}

Conversely, given a vacuum spacetime equipped with a global Killing vector field $\mathbf{Y}$ and a spacelike hypersurface $U$ with induced initial data $(g, \pi)$,  if we decompose $\mathbf{Y}=2f \mathbf{n} + X$ along~$U$, then the lapse-shift pair $(f, X)$ must lie in the kernel of $D{\Phi}|_{(g, \pi)}^*$.
\end{Theorem}

In a general non-vacuum setting, the existence of a nontrivial lapse-shift pair in the kernel of either the adjoint $D\Phi_{(g,\pi)}^*$ or the modified adjoint $D\overline{\Phi}_{(g,\pi)}^*$ has less obvious geometric or physical significance. Our first main result establishes a significant consequence of non-improvability. 
\begin{Theorem}\label{theorem:improvable}
Let $(\Omega, g, \pi)$ be a compact initial data set with nonempty smooth boundary, such that $(g,\pi)\in C^{4,\alpha}(\Omega)\times C^{3,\alpha}(\Omega)$ and also $(g,\pi) \in C^5_{\mathrm{loc}}(\Int\Omega)\times C^4_{\mathrm{loc}}(\Int\Omega)$. 
 Then either the dominant energy scalar is improvable in $\Int\Omega$, or else there exists a nontrivial lapse-shift pair $(f, X)$ on $\Int\Omega$ satisfying the system
 \begin{equation}\label{equation:pair}\tag{$\star$}
\begin{aligned}
D\overline{\Phi}_{(g,\pi)}^*(f,X)&=0\\ 
2fJ + |J|_g X&=0.
\end{aligned}
\end{equation}
\end{Theorem}
The first equation follows directly from Theorem~\ref{theorem:no-kernel}, so the new content is the second equation, which we will refer to as the \emph{$J$-null-vector equation} for $(f,X)$.


We are able to show that \eqref{equation:pair} has a meaningful physical consequence along the lines of Theorem~\ref{theorem:Moncrief}, and we also generalize the fact that a nontrivial kernel of $DR|_g$ implies that $R_g$ is constant. To state the result, we introduce some terms. We say that a spacetime $(\mathbf{N}, \mathbf{g})$ satisfies the \emph{spacetime dominant energy condition}\footnote{In general, the spacetime DEC along an initial data set $(U, g, \pi)$ is much stronger than the DEC of the initial data set, $\sigma(g, \pi)\ge 0$, which is equivalent to  $G(\mathbf{n}, \mathbf{w})\ge 0$ for any future causal~$\mathbf{w}$.}  (or \emph{spacetime DEC} for short) if $G(\mathbf{u}, \mathbf{w})\ge 0$ for all future causal  vectors $\mathbf{u}, \mathbf{w}$, where $G$ denotes the Einstein tensor of $\mathbf{g}$. 
A spacetime
$(\mathbf{N}, \mathbf{g})$ is said to be a \emph{null perfect fluid spacetime} with velocity $\mathbf{v}$ and pressure $p$ if $\mathbf{v}$ is either future null or zero at each point and the Einstein tensor takes the form:
\[
	G_{\alpha\beta} = p \mathbf{g}_{\alpha\beta} + v_\alpha v_\beta.
\]
We define \emph{null dust} to be null perfect fluid with $p\equiv 0$.

\begin{Theorem}\label{theorem:DEC}
Let $(U, g, \pi)$ be an initial data set such that $(g,\pi)\in C^{3}_{\mathrm{loc}}(U)\times C^{2}_{\mathrm{loc}}(U)$.
Assume there exists a nontrivial lapse-shift pair $(f,X)$ on $U$ solving the system~\eqref{equation:pair}, and assume that $f$ is  nonvanishing in $U$. Then the following holds:
\begin{enumerate}
\item The dominant energy scalar $\sigma(g, \pi)$ is constant on $U$.
\item  $(U, g, \pi)$ sits inside a spacetime $(\mathbf{N}, \mathbf{g})$ that admits a global Killing vector field  $\mathbf{Y}$ equal to $2f\mathbf{n} +X$ along~$U$, where $\mathbf{n}$ is the future unit normal to $U$, and $(\mathbf{N}, \mathbf{g})$ is a null perfect fluid spacetime with velocity $\mathbf{v} = \frac{\sqrt{|J|}}{2f} \mathbf{Y}$  and  pressure $p= -\tfrac{1}{2}\sigma(g, \pi)$. 
\label{item:Bartnik-spacetime}
\item  If $(g,\pi)$ satisfies the dominant energy condition, then $\mathbf{g}$ satisfies the spacetime  dominant energy condition. 
\end{enumerate}

Conversely, let $(\mathbf{N}, \mathbf{g})$ be a null perfect fluid spacetime such that $\mathbf{g}$ is $C^3_{\mathrm{loc}}$, with velocity $\mathbf{v}$ and pressure $p$, admitting a global $C^2_{\mathrm{loc}}$ Killing vector field $\mathbf{Y}$. Assume that $\mathbf{v} = \eta \mathbf{Y}$ for some scalar function $\eta$. Then $p$ is constant, and for any smooth spacelike hypersurface $U$ with induced initial data $(g, \pi)$ and future unit normal $\mathbf{n}$, if we decompose $\mathbf{Y}=2f \mathbf{n} + X$ along~$U$, then the lapse-shift pair $(f, X)$ satisfies the system~\eqref{equation:pair}. 
 \end{Theorem} 

Note the nonvanishing assumption of $f$ in the first half of Theorem~\ref{theorem:DEC}.  If $f$ vanishes on an open subset of $U$, then the $J$-null-vector equation implies that $J$ must vanish there as well.\footnote{Technically, this argument only works where $X\ne0$, but 
 the first sentence of the proof of Corollary~\ref{corollary:one-dimensional} guarantees that $X\ne0$ on a dense subset of the zero set of $f$.}
Thus the equation $D\overline{\Phi}_{(g,\pi)}^*(f,X)=0$ on such an open subset reduces to saying that $X$ is an infinitesimal symmetry of $(g, \pi)$, that is, $L_X g=0$ and $L_X\pi=0$.\footnote{This justifies the second sentence of the abstract. More precisely, the correct statement is: If $(U, g, \pi)$ is a non-improvable initial data set such that no open subset carries an infinitesimal symmetry, then an open dense subset of $U$ must sit inside a null perfect fluid spacetime carrying a global Killing vector field.}
 For the important case where it is already known that $\sigma(g, \pi)\equiv 0$ in $U$, we are able to remove the nonvanishing assumption on $f$ from Theorem~\ref{theorem:DEC}, and we can conclude that $(\mathbf{N}, \mathbf{g})$ is a null dust spacetime. (See Theorem~\ref{theorem:f_not_zero}.) 

The ``converse'' part of Theorem~\ref{theorem:DEC} can be used to construct explicit examples of initial data sets that  admit a nontrivial lapse-shift pair solving the system~\eqref{equation:pair}. In particular, pp-wave spacetimes give rise to the following examples. Refer to Section~\ref{section:asymp_flat} for the definition of \emph{asymptotically flat of type $(q, \alpha)$.} 
\begin{Example}\label{example:pp}
For each $n>8$, there exist 
complete, asymptotically flat initial data sets $(\mathbb{R}^n, g, \pi)$ that satisfy $\sigma(g, \pi)\equiv 0$, admit a nontrivial lapse-shift pair solving the system~\eqref{equation:pair}, and have $E=|P|>0$ where $(E, P)$ is the ADM energy-momentum. These examples have asymptotic decay rate $(q, \alpha)$ with $q>\frac{n-2}{2}$ but not with any $q\ge n-5$, and $(g, \pi) = (g_{\mathbb{E}}, 0)$ outside a slab.
\end{Example}

By \emph{slab}, we just mean any region lying between parallel coordinate planes in $\mathbb{R}^n$.
Despite appearances, Example~\ref{example:pp} does not contradict the existing theorems on the equality case of the spacetime positive mass theorem in~\cite{Chrusciel-Maerten:2006, Huang-Lee:2020} because those results demand a decay rate of $q>n-3$ rather than the general $q>\frac{n-2}{2}$ used in this paper and many others.


As a companion theorem to Theorem~\ref{theorem:improvable}, we show that the dominant energy scalar is always \emph{almost} improvable in the following sense:
\begin{Theorem}\label{theorem:improvability-error}
Let $(\Omega, g, \pi)$ be a compact initial data set with nonempty smooth boundary, such that $(g,\pi)\in C^{4,\alpha}(\Omega)\times C^{3,\alpha}(\Omega)$ and also $(g,\pi) \in C^5_{\mathrm{loc}}(\Int\Omega)\times C^4_{\mathrm{loc}}(\Int\Omega)$.  Let $B$ be an open ball in $\Int\Omega$, and let $\delta>0$.

Then there is a $C^{4,\alpha}(\Omega)\times C^{3,\alpha}({\Omega})$ neighborhood $\mathcal{U}$ of $(g, \pi)$ such that for any $V\subset\subset\Int\Omega$, there exist constants $\epsilon, C>0$ such that the following statement holds: For $(\gamma, \tau)\in \mathcal{U}$ and   $u\in C^{0,\alpha}_c(V)$ with $\| u \|_{C^{0,\alpha}(\Omega)}< \epsilon$, there exists $(h, w)\in C^{2,\alpha}_c(\Int\Omega)$ with $\|(h,w)\|_{C^{2,\alpha}(\Omega)}\le C\| u\|_{C^{0,\alpha}(\Omega)}$ such that the initial data set  $(\gamma+h, \tau+w)$ satisfies
\[
	\Sc(\gamma+h, \tau+w) \ge \Sc(\gamma,\tau)+u -\delta {\bf 1}_B.
\]
where ${\bf 1}_B$ is the indicator function of $B$, which equals $1$ on $B$ and $0$ outside $B$.
\end{Theorem}
In other words, we can specify an arbitrarily small function with arbitrarily small support, namely $\delta {\bf 1}_B$, and then achieve improvability up to that small error. 


Our main application of these results is to Bartnik's stationary conjecture. Let $(\Omega_0, g_0, \pi_0)$ be a compact initial data set with nonempty smooth boundary, satisfying the DEC. The \emph{Bartnik mass} $m_B(\Omega_0, g_0, \pi_0)$ is defined to be the infimum of the ADM masses of all ``admissible'' asymptotically flat extensions $(M, g, \pi)$ of $(\Omega_0, g_0, \pi_0)$. The definition of an admissible extension is given in Definition~\ref{definition:admissible}, including a no-horizon condition in terms of marginally outer trapped  hypersurfaces.
For now we emphasize that in this paper, the extension $M$ refers to an asymptotically flat manifold with boundary $\partial M$ that we think of as being glued to $\partial\Omega_0$. 

An admissible extension $(M, g, \pi)$ whose ADM mass realizes this Bartnik mass is called a \emph{Bartnik mass minimizer} for $(\Omega_0, g_0, \pi_0)$. 
Bartnik conjectured that minimizers should have special properties  (see \cite[p. 2348]{Bartnik:1989}, \cite[Conjecture~2]{Bartnik:1997}, and \cite[p. 236]{Bartnik:2002}).

\begin{conjecture}
Let $(\Omega_0, g_0, \pi_0)$ be a compact initial data set with nonempty smooth boundary, satisfying the dominant energy condition, and suppose that $(M, g, \pi)$ is a Bartnik mass minimizer for $(\Omega_0, g_0, \pi_0)$. Then $(\Int M, g, \pi)$ sits inside a vacuum stationary spacetime. 
\end{conjecture}

For our purposes, we define a \emph{stationary} spacetime $(\mathbf{N}, \mathbf{g})$ containing $(\Int M, g, \pi)$ to be one that admits a global Killing vector field $\mathbf{Y}$ that is ``uniformly'' timelike outside some bounded subset of $\Int M$, meaning that away from that subset, $\mathbf{g}(\mathbf{Y}, \mathbf{Y}) < -\varepsilon$ for some $\varepsilon>0$. See~\cite{Chrusciel-Wald:1994}. (Note that this definition of \emph{stationary} is different from that used in Bartnik's original formulation of the conjecture, in which the Killing vector field must be globally timelike.) One might also expect that if $(\Omega_0, g_0, \pi_0)$ has Bartnik mass zero, then  $(\Int \Omega_0, g_0, \pi_0)$ should sit inside the Minkowski spacetime. We will refer to this as \emph{Bartnik's strict positivity conjecture} because of the analogous conjecture made by Bartnik in the time-symmetric case. Anderson and Jauregui showed that this is false if one demands that all of $(\Omega_0, g_0, \pi_0)$, including the boundary, sits inside the Minkowski spacetime~\cite{Anderson-Jauregui:2019}.


 There is also a time-symmetric version of Bartnik's conjecture which has been affirmed (see references cited for the precise regularity and dimension assumptions):
\begin{Theorem}[Corvino~\cite{Corvino:2000}, Anderson-Jauregui~\cite{ Anderson-Jauregui:2019}, cf.~\cite{Huang-Martin-Miao:2018}]
Let $(\Omega_0, g_0, 0)$ be a compact initial data set with nonempty smooth boundary, satisfying $R_{g_0}\ge0$, and suppose that $(M, g, 0)$ minimizes ADM mass among all admissible extensions with $\pi\equiv0$. Then $(\Int M, g, 0)$ sits inside a vacuum static spacetime. 
\end{Theorem}

On the other hand, relatively little progress has been made toward the general, non-time-symmetric case of Bartnik's stationary conjecture until recently.  Corvino~\cite{Corvino:2019} used Theorem~\ref{theorem:no-kernel} combined with a conformal argument to say that if $(g,\pi)$ is a Bartnik mass minimizer, then it must admit a nontrivial kernel of $D\overline{\Phi}_{(g,\pi)}^*$. In three dimensions, Zhongshan An~\cite{An:2020} dealt with the case of vacuum minimizers by carrying out the variational approach proposed by Bartnik~\cite{Bartnik:2005} and used by Anderson-Jauregui~\cite{Anderson-Jauregui:2019} in the time-symmetric case.


Using results established in this paper, we are able to prove part of Bartnik's stationary conjecture. 

\begin{Theorem}\label{theorem:Bartnik}
Let $3\le n \le 7$, and let $(\Omega_0, g_0, \pi_0)$ be an $n$-dimensional compact smooth initial data set with nonempty smooth boundary, satisfying the dominant energy condition. Suppose that $(M, g, \pi)$ is a Bartnik mass minimizer for $(\Omega_0, g_0, \pi_0)$ such that $(g, \pi)\in C^5_{\mathrm{loc}}(\Int M)\times C^4_{\mathrm{loc}}(\Int M)$, and it has nonnegative ADM mass (that is, $E\ge |P|$). Then $(M, g, \pi)$ satisfies the following properties:
\begin{enumerate}
\item  $\sigma(g, \pi)\equiv 0$  in $\Int M$.\label{item:sigma_vanish}
\item  $(\Int M, g, \pi)$ sits inside a null dust spacetime $(\mathbf{N}, \mathbf{g})$ which satisfies the spacetime dominant energy condition and also admits a global Killing vector field~$\mathbf{Y}$. 
\item The metric $\mathbf{g}$ is vacuum on the domain of dependence of the subset of $\Int M$ where $(g, \pi)$ is vacuum, and $\mathbf{Y}$ is null on the region of $\mathbf{N}$ where $\mathbf{g}$ is not vacuum.\label{item:spacetime_vac}
\item If we further assume $E>|P|$, then $(g, \pi)$ is vacuum outside a compact subset of $M$, and  thus $(\mathbf{N}, \mathbf{g})$  is vacuum near spatial infinity. \label{item:vanishing-J}
\end{enumerate}
\end{Theorem}

Under a spin assumption, an admissible extension must have $E\ge|P|$.
 See Remark~\ref{remark:pmt_corners}. 
For $n>8$, Example~\ref{example:pp} leads to the existence of Bartnik mass minimizers that are non-vacuum and do not admit timelike Killing vectors, thereby contradicting Bartnik's stationary and strict positivity conjectures, though it should be noted that Bartnik only considered the case $n=3$.
   On the other hand, these examples are consistent with Theorem~\ref{theorem:Bartnik}.  Although Theorem~\ref{theorem:Bartnik} assumes $n\le7$ for technical reasons (see Proposition~\ref{proposition:openness}),  
   we have no reason to think the theorem fails in higher dimensions.

Finally, we remark that in the presence of a cosmological constant, the corresponding DEC takes the form of a constant lower bound on the object that we have named the ``dominant energy scalar'' (or equivalently, one can re-define ``dominant energy scalar'' it by shifting it by a constant). Because of this, one can see that although our paper has been written in the context of zero cosmological constant, the central findings of Theorems~\ref{theorem:improvable}, \ref{theorem:DEC}, and \ref{theorem:improvability-error} can easily be applied in the context of nonzero cosmological constant.

\medskip
\noindent{\bf Outline of the paper.}
In Section~\ref{section:examples} we provide examples of solutions to the system~\eqref{equation:pair} and present Example \ref{example:pp}.
In Section \ref{section:new-operators}, we introduce a new infinite-dimensional family of deformations of the modified constraint operator and present some useful properties, including a generalization of Theorem~\ref{theorem:no-kernel}. In Section~\ref{section:deformation}, we prove a key result (Proposition~\ref{theorem:generic}) that says, generically, the adjoint linearizations of those modified operators are either  injective, or else kernel elements satisfy a null-vector equation. In Section~\ref{section:improvability}, we apply Proposition~\ref{theorem:generic} to prove Theorems~\ref{theorem:improvable} and~\ref{theorem:improvability-error}.  In Section~\ref{section:null}, we obtain several strong consequences of the $J$-null-vector equation including Theorem~\ref{theorem:DEC} (proved in Section~\ref{section:null_perfect}). 
 Section~\ref{section:Bartnik} deals with application to Bartnik mass minimizers. After constructing suitable  
deformations in Section~\ref{section:decrease-energy}, we define Bartnik mass and prove Theorem~\ref{theorem:Bartnik} in Section~\ref{section:admissibility}. We end with a short discussion of the concept of Bartnik energy.

\medskip
\noindent{\bf Acknowledgement.}
The authors would like to thank Professors Richard Schoen  and Shing-Tung Yau for discussions. The work was completed while the first  author was partially supported by the NSF CAREER Award DMS-1452477 and DMS-2005588, Simons Fellowship of the Simons Foundation, and von Neumann Fellowship at the Institute for Advanced Study.  The authors also thank the referees for their helpful comments.

\section{Examples}\label{section:examples}


It is natural to ask whether there exists any \emph{non}-vacuum initial data admitting nontrivial solutions to the system~\eqref{equation:pair}. (The vacuum case is classical in view of Theorem~\ref{theorem:Moncrief}.)  It suffices to find any non-vacuum null perfect fluid spacetime $(\mathbf{N}, \mathbf{g})$  with  velocity $\mathbf{v}$ and pressure $p$ that carries a global Killing vector field $\mathbf{Y}$ with $\mathbf{v} = \eta \mathbf{Y}$ for some scalar function $\eta$. Then by Theorem~\ref{theorem:DEC}, any spacelike hypersurface in such $(\mathbf{N}, \mathbf{g})$ gives rise to an initial data set that carries a nontrivial lapse-shift pair solving~\eqref{equation:pair}.  

The simplest case is when $J=0$ everywhere. By Theorem~\ref{theorem:DEC}, we can see that this corresponds to when $(\mathbf{N}, \mathbf{g})$ has $G_{\alpha\beta} = p \mathbf{g}_{\alpha\beta}$ with constant $p$ and also admits a Killing vector. Of course, this is the same as being vacuum with respect to some \emph{cosmological constant}\footnote{Everywhere else in this paper, when we say ``vacuum'' we are referring to zero cosmological constant.} $\Lambda=-p$ and admitting a Killing vector. Perhaps the most well-known explicit examples of such spacetimes would be the de Sitter and anti-de Sitter analogs of the Kerr spacetime, discovered by Carter~\cite{Carter:1973}. Many more examples with negative cosmological constant are constructed by Chru\'{s}ciel and Delay~\cite{Chrusciel-Delay:2007}. 

More interesting is what happens where $J$ is nonzero. In this case the $J$-null-vector equation of~\eqref{equation:pair} imposes a stringent condition. In particular, the Killing vector $\mathbf{Y}$ must be null. If we make the simplifying assumption that $\mathbf{Y}$ is actually covariantly constant, then there is a coordinate chart 
 $u, x^1, \dots, x^{n}$  such that $\mathbf{Y} = \frac{\partial}{\partial u}$ and the metric $\mathbf{g}$  can be locally expressed as
\[
	\mathbf{g} = 2 du\, dx^n + S\, (dx^n)^2 +\sum_{a,b=1}^{n-1} h_{ab}dx^a dx^b,
\]
where $S$ and  $h_{ab}$ are functions independent of $u$. (See~\cite{Blau:2020}, for example.) If we  further  assume that $h_{ab}=\delta_{ab}$,  then $(\mathbf{N}, \mathbf{g})$ is called a \emph{\hbox{pp-wave} spacetime}. It is standard to verify that the Einstein tensor is $G_{\alpha\beta}=-\tfrac{1}{2} (\Delta' S) Y_\alpha Y_\beta$, where  $\Delta'$ represents the Euclidean Laplacian in the $x':=(x^1, . . . , x^{n-1})$ variables. It easily follows that the dominant energy condition holds if and only if  $\Delta' S\le0$.

In the special case that $S$ is positive, the constant $u$-slices give rise to examples of initial data sets, which are described in the following lemma, whose proof is given in Appendix~\ref{se:pp}.  
\begin{lemma}\label{lemma:pp-initial} 
Let $U$ be an open subset of $\mathbb{R}^n$ equipped with the Cartesian coordinates $x^1, \dots, x^n$. Let $S$ be a positive function on $U$ satisfying $\Delta' S\le 0$. Define the initial data set $(g, \pi)$ on $U$ by 
\[ g = S\, (dx^n)^2  + \sum_{a=1}^{n-1} (dx^a)^2,\]
and 
\begin{align*}
\pi^{nn} &= 0 \\
\pi^{na}&=\pi^{an} = \tfrac{1}{2}S^{-\frac{3}{2}}  \tfrac{\partial S}{\partial x^a} \\
\pi^{ab}&=- \tfrac{1}{2}S^{-\frac{3}{2}}\tfrac{\partial S}{\partial x^n}\delta^{ab}, 
\end{align*}
where the $a$ and $b$ indices run from $1$ to $n-1$. Then
\begin{align*}
\mu &=  -\tfrac{1}{2} S^{-1} \Delta' S\\
J &= \tfrac{1}{2} S^{-\frac{3}{2}} (\Delta' S)  \tfrac{\partial}{\partial x^n},
\end{align*}
and in particular, $\mu=|J|_g\ge 0$.

Furthermore, if we define 
\begin{align*}
f = \tfrac{1}{2} S^{-\frac{1}{2}}\quad \mbox{ and }\quad  X = S^{-1} \tfrac{\partial}{\partial x^n},
\end{align*}
then $(f,X)$ is a solution to the system~\eqref{equation:pair} on $(U, g, \pi)$. 
\end{lemma}

In the next lemma, we summarize the properties that $S$ must have in order  for the data from Lemma~\ref{lemma:pp-initial} to lead to Example~\ref{example:pp}. Below, we will adopt the definitions of weighted H\"{o}lder spaces and asymptotic flatness  in Section~\ref{section:asymp_flat}, which the reader may want to review. 
\begin{lemma}\label{lemma:S}
Let $n>3$. There exist nonconstant positive smooth functions $S$ on $\mathbb{R}^{n}$ with the following properties:
\begin{enumerate}
\item $\Delta' S\le 0$ everywhere, strictly negative somewhere, and $\Delta'S$ is integrable on $\mathbb{R}^n$. \label{item:superharmonic}
\item $S \equiv 1$ in  $\{ |x^n|\ge C\}$ for some constant $C>0$.\label{item:slab} 
\item $\displaystyle \lim_{\rho\to \infty} \int_{|x'|=\rho} - \sum_{a=1}^{n-1} \frac{\partial S}{\partial x^a} \frac{x^a}{|x'|} \, d\mu $ exists and is positive. \label{item:mass}
\item For each nonnegative integer $k$ and each $\alpha\in (0, 1)$, we have $S-1\in C^{k,\alpha}_{-q} (\mathbb{R}^{n})$ with $q=n-3-(k+\alpha)$.\label{item:decay}
\end{enumerate}
\end{lemma}
\begin{remark}
Such a function $S$ cannot exist for $n=3$ because Liouville's theorem says that any superharmonic function on $\mathbb{R}^2$ that is bounded below must be constant.
\end{remark}
\begin{proof}
Let $F$ be a smooth nonnegative function on $\mathbb{R}^{n-1}$ with coordinates $x'=(x^1, \ldots, x^{n-1})$, such that $F= O(|x'|^{-s})$ for some $s>n-1$. We can solve $\Delta'\psi = -F$ on $\mathbb{R}^{n-1}$ via convolution with the fundamental solution of the Laplacian on Euclidean $\mathbb{R}^{n-1}$. As long as $F$ is not identically zero, $\psi(x')$ will be a  positive, globally superharmonic function on  $\mathbb{R}^{n-1}$. For $n>3$, it must have the expansion
$\psi(x')= A|x'|^{3-n}+O_1(|x'|^{2-n})$, and since $\psi$ is positive, the constant $A$ must also be positive. Now define $S(x', x^n) = 1+\phi(x^n) \psi(x')$, where $\phi$ is chosen to be any nontrivial compactly supported smooth nonnegative function on $\mathbb{R}$. Note that $\Delta' S = -\phi(x^n)F(x')\le0$ everywhere and strictly negative somewhere. It is straightforward to verify that $S $ satisfies Items \eqref{item:superharmonic}, \eqref{item:slab}, and \eqref{item:mass}.

Since the derivatives of $S$ in the $x^n$ direction do not decay any faster than $|x'|^{3-n}$,  we can only conclude that $S-1$ and its derivatives  of any order are $O(|x|^{3-n})$ and thus $S-1\in C^{k,\alpha}_{-q} (\mathbb{R}^{n})$ with $q=n-3-k-\alpha$  by the definition of weighted H\"older spaces.

\end{proof}

\begin{proof}[Proof of Example~\ref{example:pp}]
Choose any $S$ as in Lemma~\ref{lemma:S} and use this choice in Lemma~\ref{lemma:pp-initial} to construct initial data $(\mathbb{R}^n, g, \pi)$. We claim that for $n>8$, this is the desired example. Note that by construction, $(g,\pi)$ is clearly complete, $(g, \pi) = (g_{\mathbb{E}}, 0)$ outside of a slab, satisfies $\sigma(g,\pi)=0$, and admits a nontrivial solution to~\eqref{equation:pair}. The main task is to show that $(g,\pi)$ is asymptotically flat.

Recall that our asymptotic flatness condition requires $g_{ij}-\delta_{ij}=C^{2,\alpha}_{-q}(\mathbb{R}^n)$ and $\pi_{ij}=C^{1,\alpha}_{-q}(\mathbb{R}^n)$ for some $\frac{n-2}{2} < q < n-2$, and $(\mu, J)\in L^1(\mathbb{R}^n)$. For our $(g, \pi)$ from Lemma~\ref{lemma:pp-initial}, this is equivalent to requiring that $S-1\in C^{2,\alpha}_{-q}$ for some $\frac{n-2}{2} < q < n-2$ and $\Delta' S$ is integrable. By Item~\eqref{item:decay} of Lemma~\ref{lemma:S}, it imposes the condition on $n$:
\[
 (n-3)-(2+\alpha) > \tfrac{n-2}{2} \text{ for some $\alpha\in(0,1)$},\quad \mbox{or equivalently, } \quad n>8.
\]

To see that $E=|P|>0$, we evaluate the ADM energy-momentum by integrating over large capped cylinders. The caps do not contribute, and we can see that
\begin{align*}
	E = -P_n =  \frac{1}{2(n-1)\omega_{n-1}} \lim_{\rho \to \infty} \int_{|x'|=\rho} -\sum_{a=1}^{n-1}\frac{\partial S}{\partial x^a} \frac{x^a}{|x'|} \, d\mu >0,
\end{align*}
and $P_1=\dots=P_{n-1}=0$. 

\end{proof}


\begin{remark}While the above argument needs $n>8$, it can be easily shown that initial data from  Lemma~\ref{lemma:pp-initial} cannot be asymptotically flat and satisfy the DEC if  $3\le n \le 6$. Previously, Beig and Chru\'{s}ciel observed that \emph{vacuum} pp-waves give rise to initial data sets with the appropriate asymptotic decay rate and  $E=|P|>0$ for $n=3$, but these examples are not complete~\cite[p. 1951]{Beig-Chrusciel:1996}. 
\end{remark} 
In dimensions $n>8$, Example~\ref{example:pp} provides numerous counterexamples to what one might naively expect to be the optimal statement of the equality case of the spacetime positive mass theorem. The ``expected'' statement is that any complete, asymptotically flat initial date set satisfying the DEC and having $E=|P|$ must have $E=|P|=0$ and sit inside Minkowski space. Although the strong decay rate $q=n-2$ might be regarded as the most natural and physically relevant asymptotically flat decay rate, Example~\ref{example:pp} is still surprising in light of the fact that the spacetime positive mass inequality $E\ge|P|$ does indeed hold for all decay rates $q>\frac{n-2}{2}$. Specifically, a density theorem \cite[Theorem 18]{Eichmair-Huang-Lee-Schoen:2016} shows that Example~\ref{example:pp} lies in the \emph{limit} of complete, asymptotically flat initial data sets satisfying the DEC with strong decay rate $q=n-2$. 

It would be interesting to know the lowest decay rate for which the equality case of the spacetime positive mass theorem holds. In dimensions $n=3$ and $4$,  by~\cite{Beig-Chrusciel:1996, Chrusciel-Maerten:2006, Huang-Lee:2020}, the decay rate assumption of $q>\frac{n-2}{2}$ was already known to be sufficient for the equality case to hold. But in dimensions $n=5$ through $8$, there is a gap between $q>n-3$ (for which the equality case of the positive mass theorem is known to hold) and $q>\frac{n-2}{2}$, where counterexamples might or might not exist. In dimensions $n>8$, there is a gap between $q>n-3$ and the counterexamples at $q=n-5-\alpha$ obtained in Example~\ref{example:pp} .

The fact that Example~\ref{example:pp} is Euclidean outside of a slab is reminiscent of the work of Carlotto and Schoen~\cite{Carlotto-Schoen:2016}, who constructed examples of complete, asymptotically flat vacuum initial data sets that are trivial outside of a conical region, but do not sit inside Minkowski space. Those examples also have decay rates less than the strong decay rate $q=n-2$. 

\begin{example}[Bartnik mass minimizers]\label{example:Bartnik}
These examples also give counterexamples to  Bartnik's stationary and strict positivity conjectures for $n>8$. Let $\Omega_0$ be a closed ball centered at the origin in an initial data set $(\mathbb{R}^n, g, \pi)$ from Example~\ref{example:pp}. For large enough $\Omega_0$, it is clear that $(\mathbb{R}^n\smallsetminus \Int\Omega_0, g, \pi)$ is an admissible extension of $(\Omega_0, g, \pi)$. See Definition~\ref{definition:admissible} for the definition of admissible extensions.

 Since the extension $(\mathbb{R}^n\smallsetminus \Int\Omega_0, g, \pi)$ has ADM mass equal to zero, the Bartnik mass of $(\Omega_0, g, \pi)$ is zero and the extension has to be a Bartnik mass minimizer\footnote{ This is assuming that the ADM mass of admissible extensions are always nonnegative. This would be the case if we only consider admissible extensions that are spin (for example, diffeomorphic to $\mathbb{R}^n\smallsetminus \Int\Omega_0$). See Remark~\ref{remark:pmt_corners}.}.  This contradicts Bartnik's stationary and strict positivity conjectures because $(\mathbb{R}^n\smallsetminus \Int\Omega_0, g, \pi)$ and $(\Omega_0, g, \pi)$ are not vacuum, so long as the function $F$ from the proof of Lemma~\ref{lemma:S} was chosen to be positive on $\mathbb{R}^{n-1}$.

On the other hand, if we select $F$, $\phi$, and $\Omega_0$ such that $\phi(x^n)F(x')$ is supported in the ball $\Omega_0$, then the Bartnik minimizing extension $(\mathbb{R}^n\smallsetminus \Int\Omega_0, g, \pi)$ is vacuum, but it cannot sit inside a stationary spacetime, as defined in the introduction. This can be seen as follows: Given a global Killing field that is uniformly timelike outside a bounded subset of $\mathbb{R}^n\smallsetminus \Omega_0$, it induces a lapse-shift pair $(f, X)$ with  $4f^2 > |X|^2_g +\varepsilon$ near infinity, for some $\varepsilon>0$, and by Theorem~\ref{theorem:Moncrief}, it must satisfy $D\Phi|^*_{(g,\pi)}(f, X)=0$ in the vacuum region near infinity.  Applying Lemma~\ref{lemma:asymptotics},
it is not hard to see that the inequality $4f^2 > |X|^2_g \ge0$ is only possible if the constants $c_i$ and $d_{ij}$ appearing in~\eqref{equation:asymptotics-linear} are all zero, and thus  $(f, X)$ must be asymptotic to $(aE, -2aP)$ for some constant $a$. Therefore  $4(aE)^2 > |2aP|^2 +\varepsilon$, but this is  a contradiction since we already know that $E=|P|$ in this example.
\end{example}

Example~\ref{example:Bartnik} may be surprising because it shows that  \emph{asymptotically flat} Cauchy hypersurfaces can exist in pp-wave spacetimes in higher dimensions and satisfy admissibility for the Bartnik mass. Szabados~\cite[p. 125]{Szabados:2009} opined that ``pure radiation'' initial data $(\Omega_0, g_0,  \pi_0)$, as is found in a pp-wave, should have zero ``quasi-local mass'' and suggested that strict positivity of the Bartnik mass is undesirable. While we have computed the Bartnik mass to be zero for the cases in~Example~\ref{example:Bartnik}. the general question remains open. 
In light of Example~\ref{example:Bartnik}, one might want to require a strong decay rate of $q=n-2$ in the definition of admissible extensions, but even with such a modified definition, it is still possible that $(\Omega_0, g, \pi)$ could have Bartnik mass equal to zero, in view of the density theorem mentioned above. 

\section{A new family of modified constraint operators}\label{section:new-operators}
We set some conventions for our usage of variables: When we refer to an initial data set $(U, g, \pi)$,  we assume $U$ is a connected manifold unless otherwise specified, and we do not require $(U, g)$ to be complete since we will often want to think of it as an open subset of some larger manifold. We will typically use $\Omega$ to denote a compact manifold with nonempty smooth boundary and $M$ to denote a complete asymptotically flat manifold which may or may not have boundary. We will \emph{always} assume $(g, \pi)$ is locally $C^3\times C^2$ and be explicit when we assume more regularity.

Define the \emph{energy density} $\mu$ and the \emph{current density} $J$ by
\begin{align} \label{equation:mass-current}
\begin{split}
	\mu&:=\tfrac{1}{2}\left(R_g  + \tfrac{1}{n-1}(\tr_g \pi )^2 - |\pi|_g^2 \right)\\
	J&:=  \Div_g \pi,
\end{split}
\end{align}
where $n$ is the dimension of $U$. The \emph{constraint map} is defined by $\Phi(g, \pi) := (2\mu, J)$. Denote its linearization at $(g, \pi)$ by $D\Phi|_{(g, \pi)}$. Recall the formula for the $L^2$ formal adjoint operator $D\Phi|_{(g, \pi)}^*$ on a \emph{lapse-shift pair} $(f, X)$, where  $f$ is a $C^2_{\mathrm{loc}}$ function and $X$ is a $C^1_{\mathrm{loc}}$ vector field on~$U$: 
\begin{align}	\label{equation:adjoint}
\begin{split}
 	&D\Phi|_{(g,\pi)} ^*(f, X)   = \left(-(\Delta_g f)g_{ij}+ f_{;ij} +\left[-R_{ij} +\tfrac{2}{n-1} (\mbox{tr}_g \pi) \pi_{ij} - 2 \pi_{i\ell} \pi^\ell_j \right] f\right.\\
	& \quad+ \tfrac{1}{2} \left[ (L_X\pi)_{ij} + (\Div_g X) \pi_{ij}- X_{k;m} \pi^{km} g_{ij} - \langle X, J \rangle_g g_{ij} \right]- (X\odot J)_{ij}, \\
	&\quad \left. -\tfrac{1}{2} (L_Xg)^{ij} + \left(\tfrac{2}{n-1} (\mbox{tr}_g \pi ) g^{ij}- 2 \pi^{ij}  \right) f\right), \end{split}
 \end{align}
 where semicolon denotes covariant differentiation, $R_{ij}$ is the Ricci curvature of $g$,  $X\odot J$ is the symmetric tensor product $(X\odot J)_{ij} = \tfrac{1}{2} (X_i J_j+ J_i X_j)$, and all indices are raised and lowered using~$g$. In particular, note that
  $(L_X\pi)_{ij}:=g_{i\ell} g_{jm} (L_X\pi)^{\ell m}$.
 
The \emph{modified constraint operator} at $(g,\pi)$, introduced by Corvino and the first author~\cite{Corvino-Huang:2020}, is defined on initial data $(\gamma,\tau)$ on $U$ by
\begin{align*} 
	\overline{\Phi}_{(g, \pi)}(\gamma,\tau) :=\Phi(\gamma, \tau) +(0, \tfrac{1}{2} \gamma\cdot J)
\end{align*}
where $J$ is the current density of $(g,\pi)$ and $(\gamma\cdot J)^i := g^{ij} \gamma_{jk} J^k$ denotes the contraction of $\gamma$ and $J$ with respect to the background metric $g$. Denote its linearization at $(g, \pi)$ by  $D\overline{\Phi}_{(g, \pi)}:= \left.D\overline{\Phi}_{(g, \pi)}\right|_{(g,\pi)}$.  Then its $L^2$ formal adjoint $D\overline{\Phi}_{(g, \pi)}^*$ on a lapse-shift pair $(f,X)$ satisfies
\begin{align}\label{equation:adjoint-mod}
	D\overline{\Phi}_{(g, \pi)}^* (f,X) = (D\Phi|_{(g,\pi)})^*(f, X) + (\tfrac{1}{2} X\odot J, 0).
\end{align}

\begin{definition}
Given a function $\varphi$ and a vector field $Z$ on $U$, we introduce  
the \emph{$(\varphi, Z)$-modified constraint operator} at $(g, \pi)$, denoted $\Phi^{(\varphi, Z)}_{(g,\pi)}$.
It is defined on initial data $(\gamma,\tau)$ on $U$ by
\begin{align}
		\Phi^{(\varphi, Z)}_{(g,\pi)}  (\gamma, \tau) 
		&:=\overline{\Phi}_{(g,\pi)}(\gamma, \tau) + \left(2\varphi |Z|_\gamma^2, \varphi |Z|_g \gamma \cdot Z\right)\notag \\
		&= \Phi(\gamma, \tau) +(0, \tfrac{1}{2} \gamma\cdot J) + \left(2\varphi |Z|_\gamma^2, \varphi |Z|_g \gamma \cdot Z\right),	\label{equation:modified2}	
\end{align}
where $(\gamma\cdot Z)^i := g^{ij} \gamma_{jk} Z^k$ denotes the contraction of $\gamma$ and $Z$ with respect to $g$. Note that with our notation, $\Phi^{(0,0)}_{(g,\pi)}=\overline{\Phi}_{(g,\pi)}$. 
\end{definition}

Denote the linearization at $(g, \pi)$ by  $D\Phi^{(\varphi, Z)}_{(g, \pi)}:= \left.D\Phi^{(\varphi, Z)}_{(g, \pi)}\right|_{(g,\pi)}$.  Then it is easy to see that its $L^2$ formal adjoint $\left(D\Phi^{(\varphi, Z)}_{(g, \pi)}\right)^*$  on a lapse-shift pair $(f,X)$ satisfies
\begin{align}
	\left(D\Phi^{(\varphi, Z)}_{(g, \pi)}\right)^*(f, X) =D\overline{\Phi}_{(g, \pi)}^*(f, X) + \left( \varphi Z\odot (2fZ+|Z|_g X), 0\right). \label{equation:adjoint-Z-mod}
\end{align}

 \begin{lemma}\label{lemma:Hessian}
 If $\left(D\Phi^{(\varphi, Z)}_{(g, \pi)}\right)^*(f, X) =0$, then 
\begin{align}
\begin{split}\label{equation:Hamiltonian}
0 &=  - (\Delta_g f)g_{ij} + f_{;ij} -R_{ij} f +  \left[ \tfrac{3}{n-1}(\tr \pi)\pi_{ij} - 2\pi_{i\ell}\pi^\ell_j
  \right]f\\
  &\quad  + \left[   -\tfrac{1}{n-1}(\tr_g \pi)^2 + |\pi|^2_g \right] g_{ij} f\\
&\quad +\tfrac{1}{2}(L_X \pi)_{ij} 
 -\tfrac{1}{2}\langle X, J\rangle_g g_{ij} -\tfrac{1}{2}(X\odot J)_{ij} \\
 &\quad + \varphi \left[Z\odot (2f Z + |Z|_g X)\right]_{ij}\end{split}\\
 	0&=-\tfrac{1}{2}(L_Xg)_{ij} + \left( \tfrac{2}{n-1} (\tr_g \pi) g_{ij} - 2\pi_{ij} \right)f.\label{equation:momentum}
\end{align}
Moreover, if we assume \eqref{equation:momentum}, then $\left(D\Phi^{(\varphi, Z)}_{(g, \pi)}\right)^*(f, X) =0$ is equivalent to \eqref{equation:Hamiltonian}.

Furthermore,  If $\left(D\Phi^{(\varphi, Z)}_{(g, \pi)}\right)^*(f, X) =0$, then $(f, X)$ satisfies the Hessian type equations:
\begin{align}
\begin{split}\label{equation:Hessian}
	0&= f_{;ij}  +\left[ -R_{ij} + \tfrac{2}{n-1} (\tr_g \pi) \pi_{ij} - 2 \pi_{ik} \pi^k_j \right]f\\
	&+ \left[ \tfrac{1}{n-1} \left(R_g- \tfrac{2}{n-1} (\tr_g \pi)^2 + 2|\pi|^2\right)g_{ij}\right] f\\
	&+ \tfrac{1}{2} \left[ (L_X\pi)_{ij} + (\Div_g X) \pi_{ij} \right]-\tfrac{1}{2}(X\odot J)_{ij} \\
	&  +\tfrac{1}{2(n-1)} \left[- \tr_g (L_X \pi) -  (\Div_gX) (\tr_g \pi) +  X_{k;m}\pi^{km} + 2 \langle X, J \rangle_g\right]g_{ij} \\
	&+ \varphi \left[Z\odot (2f Z + |Z|_g X)\right]_{ij}- \tfrac{1}{n-1}\varphi (2 f |Z|_g^2 + |Z|_g \langle X, Z\rangle_g) g_{ij},
	\end{split}
\end{align}
\begin{align}
	\begin{split} \label{equation:second-derivative}
	 0&=   X_{i;jk}+\tfrac{1}{2} \left(R^{\ell}_{kji} + R^{\ell}_{ikj} + R^{\ell}_{ijk}\right) X_{\ell}\\
	&\,\,\,-\left[\left(\tfrac{2}{n-1} (\tr_g \pi) g_{ij} -2 \pi_{ij} \right) f\right]_{;k}-\left[\left(\tfrac{2}{n-1} (\tr_g \pi) g_{ki} - 2\pi_{ki} \right) f\right]_{;j} \\
	&\,\,\, +\left[\left(\tfrac{2}{n-1} (\tr_g \pi) g_{jk} - 2\pi_{jk} \right) f\right]_{;i},
	\end{split}
\end{align}
where our convention for the Riemann tensor is such that the Ricci tensor $R_{jk}:=R^\ell_{\ell j k}$.
 \end{lemma}
\begin{proof}
From equations~\eqref{equation:adjoint}, \eqref{equation:adjoint-mod}, and~\eqref{equation:adjoint-Z-mod}, we see that $(f, X)$ satisfies \eqref{equation:momentum} and 
 \begin{align}\label{equation:Hamiltonian0}
\begin{split}
	0&= -(\Delta_g f )g_{ij} + f_{;ij} + \left[ -R_{ij} + \tfrac{2}{n-1} (\tr_g \pi) \pi_{ij} - 2 \pi_{i\ell} \pi^\ell_j \right]f  \\
	&\quad + \tfrac{1}{2}  \left[ (L_X\pi)_{ij} + (\Div_g X) \pi_{ij}- X_{k;m} \pi^{km} g_{ij} - \langle X, J \rangle_g g_{ij} \right] \\
	&\quad - \tfrac{1}{2}(X\odot J)_{ij}+ \varphi \left[Z\odot (2f Z + |Z|_g X)\right]_{ij}.
\end{split}
\end{align}
We can use~\eqref{equation:momentum} and its trace to swap out all terms with derivatives of $X$ in the above equation in exchange for terms involving $\pi$ and $f$. If we do this, we obtain~\eqref{equation:Hamiltonian} after straightforward manipulations.

 Equation~\eqref{equation:Hessian} comes from taking the trace of~\eqref{equation:Hamiltonian0} to solve for $\Delta_g f$ and then substituting it back in to~\eqref{equation:Hamiltonian0}. Equation~\eqref{equation:second-derivative} follows from~\eqref{equation:momentum} and commuting second covariant derivatives of $X$ to obtain the curvature terms. (See, for example,~\cite[Lemma B.3]{Huang-Lee:2020} for details.)
 
\end{proof}

By Theorem~\ref{theorem:no-kernel}, injectivity of $D\Phi|_{(g,\pi)}^*$ allows one to prescribe perturbations of $\Phi$ using compactly supported perturbations of $(g,\pi)$~\cite[Theorem 2]{Corvino-Schoen:2006}. Essentially the same reasoning can be used to prove a similar statement for the $(\varphi, Z)$-modified operator $\Phi^{(\varphi, Z)}_{(g, \pi)}$. We say that the operator $\left( D\Phi^{(\varphi, Z)}_{(g, \pi)}\right)^*$ is \emph{injective on $\Int\Omega$} if it is injective when thought of as a linear map from the domain
$C^2_{\mathrm{loc}}(\Int \Omega)\times  C^1_{\mathrm{loc}}(\Int \Omega)$. 
  \begin{theorem}\label{theorem:deformationZ}
Let $(\Omega, g, \pi)$ be a compact initial data set with nonempty smooth boundary, such that $(g,\pi)$ is $C^{4,\alpha}(\Omega)\times C^{3,\alpha}(\Omega)$ for some $\alpha\in (0,1)$, and let $(\varphi, Z)\in$ $C^{2,\alpha}(\Omega)$.
Assume that $\left( D\Phi^{(\varphi, Z)}_{(g, \pi)}\right)^*$ is injective on $\Int\Omega$.

Then there is a $C^{4,\alpha}(\Omega)\times C^{3,\alpha}(\Omega)$ neighborhood $\mathcal{U}$ of $(g, \pi)$ and a $C^{2,\alpha}(\Omega)$ neighborhood $\mathcal{W}$ of $(\varphi, Z)$ such that for any $V\subset\subset \Int\Omega$, there exist constants $\epsilon, C>0$ such that the following statement holds: For $(\gamma, \tau)\in \mathcal{U}$, $(\varphi', Z')\in\mathcal{W}$, and  $(u,\Upsilon )\in C^{0,\alpha}_c(V)\times C^{1,\alpha}_c(V) $ with \hbox{$\| (u, \Upsilon) \|_{C^{0,\alpha}(\Omega)\times C^{1,\alpha}(\Omega)}< \epsilon$}, there exists $(h, w)\in C^{2,\alpha}_c(\Int\Omega)$ satisfying $\|(h,w)\|_{C^{2,\alpha}(\Omega)}\le C\| (u, \Upsilon)\|_{C^{0,\alpha}(\Omega)\times C^{1,\alpha}(\Omega)}$ such that
 \[
 	\Phi^{(\varphi', Z')}_{(\gamma,\tau)} (\gamma+h, \tau+w) = \Phi^{(\varphi', Z')}_{(\gamma,\tau)}(\gamma,  \tau) + (u, \Upsilon).
 \]
  \end{theorem}
 \begin{proof}
The proof is essentially the same as the proof of \cite[Theorem 3.1]{Corvino-Huang:2020} (which is itself based on~\cite[Theorem 2]{Corvino-Schoen:2006}), because the operators involved only differ by inconsequential zero order terms. Note that \cite[Theorem 3.1]{Corvino-Huang:2020} is stated more generally in terms of certain weighted H\"{o}lder spaces, but here we state a simpler version sufficient for our applications.

We outline the steps with explicit references to ~\cite{Corvino-Huang:2020}. For given $(u, \Upsilon)\in C^{0,\alpha}_c(V)\times C^{1,\alpha}_{c}(V)$, we first solve the linearized equation by showing that there exists $(h_0, w_0)\in C^{2,\alpha}_c(\Int \Omega)$ solving $D\Phi^{(\varphi, Z)}_{(\gamma,\tau)}(h_0, w_0) = (u, \Upsilon)$. To do this, we define a functional $\mathcal{G}$ on the space of lapse-shift pairs $(f, X)$ as in \cite[Section 5.1]{Corvino-Huang:2020} by substituting the adjoint operator there with our operator $\left(D\Phi^{(\varphi, Z)}_{(g, \pi)}\right)^*$. The injectivity assumption allows us to derive the coercivity estimate for $\left(D\Phi^{(\varphi, Z)}_{(\gamma,\tau)}\right)^*$, just as \cite[Equation (5.5)]{Corvino-Huang:2020}, because the operators involved differ only by  zero order terms. Therefore, we can minimize the functional $\mathcal{G}$ to obtain a solution $(h_0, w_0)$ to the linearized equation.    Moreover, $(h_0, w_0)$  satisfies the same weighted estimates as in  \cite[Theorem 5.6]{Corvino-Huang:2020} because the Schauder interior estimates apply in the same way in our case. Once the weighted estimates are established, the iteration scheme \cite[Theorem 5.10 and Lemma 5.11]{Corvino-Huang:2020} is applicable in our setting to solve the nonlinear equation as desired.

 \end{proof}

The main usefulness of the modified constraint operator $\overline{\Phi}_{(g,\pi)}$ is that controlling the modified constraints of $(\gamma, \tau)$ near $(g,\pi)$ gives good control over the dominant energy scalar $\sigma(\gamma, \tau)$. Looking at the case $Z=J$, the operator $\Phi^{(\varphi, J)}_{(g, \pi)}$ shares a similar property, and this is what motivated the definition of $\Phi^{(\varphi, J)}_{(g, \pi)}$.

\begin{lemma}\label{lemma:sigma_preserve}
Let $(\gamma, \tau)$ and $(\bar\gamma, \bar\tau)$ be initial data on a manifold $U$ such that $|\bar\gamma-\gamma|_\gamma<1$. Let $\varphi$ be a function on $U$ such that $|\varphi J|_\gamma < (\sqrt{2}-1)/2$, where $J$ is the current density of $(\gamma, \tau)$. Assume that
\begin{equation}\label{equation:prescribe}
\Phi^{(\varphi, J)}_{(\gamma, \tau)} (\bar\gamma, \bar\tau  ) = \Phi^{(\varphi,J)}_{(\gamma, \tau)} (\gamma, \tau ) + (u, 0)
\end{equation}
for some function $u$. Then
\[ 
\Sc(\bar\gamma,\bar\tau) \ge 
  \Sc(\gamma, \tau)+u.
\]
\end{lemma}
\begin{proof}
 Let
 $(\bar{\mu}, \bar{J})$ and $(\mu, J)$ denote the energy and current densities of $(\bar{\gamma}, \bar{\tau})$ and $(\gamma, \tau)$, respectively, and let $(h, w):=(\bar{\gamma}-\gamma, \bar{\tau}-\tau)$.
  In what follows, we will compute all lengths, inner products, and ``dot contractions'' using the metric $\gamma$, unless otherwise specified.
  
 If we re-write out hypothesis~\eqref{equation:prescribe} using the definition of the $(\varphi, J)$-modified constraint operator~\eqref{equation:modified2} and write as much as we can in terms of  $h$, we can see that 
\[
	\Phi(\bar{\gamma}, \bar{\tau}) = \Phi(\gamma, \tau) + \left(-2\varphi \langle h\cdot J,  J\rangle, - \left(\tfrac{1}{2} + \varphi |J|\right) h\cdot J\right) + (u, 0),
\]
or in other words,
\begin{align*}
	\bar{\mu} &= \mu + \tfrac{u}{2} -\varphi  \langle h\cdot J,  J\rangle \\
	\bar{J} &=J - \left(\tfrac{1}{2} + \varphi |J| \right) h\cdot J.
\end{align*}

We compute
\begin{align*}
	|\bar{J}|^2_{\bar{\gamma}} &= (\gamma+h)_{ij} \left( J^i   -(\tfrac{1}{2} + \varphi|J|)(h\cdot J)^i\right) \left(J^j    -(\tfrac{1}{2} + \varphi|J|)(h\cdot J)^j\right)\\
	&=|J|^2-2 \varphi |J|  \langle h\cdot J,  J\rangle	+ \left(- \tfrac{3}{4}  - \varphi |J| + \varphi^2|J|^2 \right) |h\cdot J|^2  \\
	&\quad +   \left(\tfrac{1}{4} + \varphi|J|  +\varphi^2 |J|^2  \right) \langle h\cdot (h\cdot J) , h\cdot J \rangle \\
	& \le |J|^2 -2 \varphi |J|  \langle h\cdot J,  J\rangle\\
	&\quad + \Big[\left(- \tfrac{3}{4}  - \varphi |J| + \varphi^2|J|^2 \right)+ \left(\tfrac{1}{4} + |\varphi J|  +\varphi^2 |J|^2  \right)|h|\Big] |h\cdot J|^2\\
	& \le |J|^2 -2 \varphi |J|  \langle h\cdot J,  J\rangle+ \left(- \tfrac{1}{2}  +2 |\varphi J| + 2|\varphi J|^2 \right) |h\cdot J|^2,
\end{align*}
where we used $|h|<1$ in the last line. By our assumption that $|\varphi J| <(\sqrt{2}-1)/2$, it follows that 
$- \tfrac{1}{2}  +2 |\varphi J| + 2|\varphi J|^2<0$, and hence
\[ |\bar{J}|^2_{\bar{\gamma}} \le |J|^2 -2 \varphi |J| \langle h\cdot J, J\rangle
\le ( |J| -\varphi \langle h\cdot J, J\rangle)^2.\] 
 Combining the square root of the above inequality with our formula for $\bar{\mu}$, we get the desired inequality, 
 \begin{align*}
  \Sc(\bar\gamma,\bar\tau) &= 2(\bar\mu-|\bar{J}|_{\bar\gamma})\\
& \ge  2\left(\mu+\frac{u}{2}  -\varphi \langle h\cdot J, J\rangle\right) -  2\left(|J| -\varphi \langle h\cdot J, J\rangle\right)\\
& = 
  \Sc(\gamma, \tau)+u.
  \end{align*}
  
\end{proof}

While the general $(\varphi, Z)$-modified constraint operator $\Phi^{(\varphi, Z)}_{(g, \pi)}$ does not share the same good property as in the proposition above, we can compute the error by which it fails.

\begin{lemma}\label{lemma:sigma_error}
Let $(\gamma, \tau)$ and $(\bar\gamma, \bar\tau)$ be initial data on a manifold $U$ such that $|\bar\gamma-\gamma|_\gamma<1$. Let $\varphi$ and $Z$ be a function and a vector field on $U$ such that 
$|\varphi|<1$ and $|Z|_\gamma<1$ on $U$.
Assume that
\begin{equation}\label{equation:prescribeZ}
\Phi^{(\varphi, Z)}_{(\gamma, \tau)} (\bar\gamma, \bar\tau  ) = \Phi^{(\varphi,Z)}_{(\gamma, \tau)} (\gamma, \tau ) + (u, 0)
\end{equation}
for some function $u$. Then 
\[ 
\Sc(\bar{\gamma}, \bar{\tau})  \ge \Sc(\gamma,\tau)+ u - 6|\bar\gamma-\gamma|_\gamma^{\frac{1}{2}}|\varphi|^{\frac{1}{2}} |Z|_\gamma (|J|_\gamma^{\frac{1}{2}} + 1),
\]
where $J$ is the current density of $(\gamma, \tau)$.
\end{lemma}

\begin{proof}
 Again, let
 $(\bar{\mu}, \bar{J})$ and $(\mu, J)$ denote the energy and current densities of $(\bar{\gamma}, \bar{\tau})$ and $(\gamma, \tau)$, respectively, and let $(h, w):=(\bar{\gamma}-\gamma, \bar{\tau}-\tau)$.
  In what follows, we will compute all lengths, inner products, and ``dot contractions'' using the metric $\gamma$, unless otherwise specified.

Again, if we re-write out hypothesis~\eqref{equation:prescribeZ} using the definition of the $(\varphi, Z)$-modified constraint operator~\eqref{equation:modified2} and write as much as we can in terms of  $h$, we can see that 
\[
	\Phi(\bar{\gamma}, \bar{\tau}) = \Phi(\gamma, \tau) + (0,-\tfrac{1}{2} h \cdot J) + 
	(-2\varphi \langle h\cdot Z, Z\rangle , -\varphi |Z| h\cdot Z) + (u, 0), 
\]
or in other words,
\begin{align*}
	\bar{\mu} &= \mu + \frac{u}{2} - \varphi \langle h\cdot Z, Z\rangle\\ 
	\bar{J} &=J- \tfrac{1}{2} h\cdot J -  \varphi |Z| h\cdot Z. 
\end{align*}
After tedious but straightforward computation, we obtain
\begin{align*}
	|\bar{J}|_{\bar{\gamma}}^2 &= |J|^2 - \tfrac{3}{4} |h\cdot J|^2+\tfrac{1}{4} \langle h\cdot (h\cdot J),  h\cdot J\rangle\\
	&\quad
	- 2\varphi |Z| \langle h\cdot J, Z\rangle  
	  - \varphi |Z| \langle h\cdot J, h\cdot Z\rangle + \varphi |Z| \langle h\cdot (h\cdot J) ,h\cdot Z\rangle\\
	  &\quad
	 + \varphi^2 |Z|^2 |h\cdot Z|^2   + \varphi^2 |Z|^2 \langle h\cdot (h\cdot Z), h\cdot Z\rangle.
	\end{align*}
Since $|h|<1$, we can absorb the third term on the right side into the second term on the right side, just as we did in the proof of Lemma~\ref{lemma:sigma_preserve}. For the rest of the terms, we use Cauchy-Schwarz  together with the fact that $|h|$, $|\varphi|$ and $|Z|$ are all less than $1$ to obtain
\[
|\bar{J}|_{\bar{\gamma}}^2	
	\le |J|^2 + |h| |\varphi| |Z|^2 (4|J| + 2),
\]
so by the triangle inequality applied twice, 
\[
	|\bar{J}|_{\bar{\gamma}} \le |J| + 2|h|^{\frac{1}{2}} |\varphi|^{\frac{1}{2}}  |Z| (|J|^{\frac{1}{2}} + 1).
\]
Combining this with our formula for $\bar\mu$, we see that 
\begin{align*}
	\Sc(\bar{\gamma}, \bar{\tau}) &= 2(\bar\mu-|\bar{J}|_{\bar\gamma})\\
	& \ge [2\mu + u - 2|h| |\varphi| |Z|^2] - \left[2|J| +4|h|^{\frac{1}{2}} |\varphi|^{\frac{1}{2}}  |Z| (|J|^{\frac{1}{2}} + 1)\right]\\
	&\ge \Sc(\gamma,\tau)+ u - 6|h|^{\frac{1}{2}}|\varphi|^{\frac{1}{2}} |Z| (|J|^{\frac{1}{2}} + 1).
\end{align*}

\end{proof}

\section{The null-vector equation for kernel elements}\label{section:deformation}



The following theorem is the main result underlying the proofs of Theorem~\ref{theorem:improvable} and~\ref{theorem:improvability-error}. This is where we take advantage of the freedom to choose $\varphi$ in the large family of operators $\Phi^{(\varphi, Z)}_{(g,\pi)}$. In the following theorem and elsewhere, the notation $C^3(U)$ denotes functions in  $C^3_{\mathrm{loc}}(U)$ that have bounded $C_3$ norm.

\begin{theorem}\label{theorem:generic}
Let  $(U, g, \pi)$ be an initial data set such that $(g,\pi)\in C^5_{\mathrm{loc}}(U)\times C^4_{\mathrm{loc}}(U)$. Given any vector field $Z\in C^3_{\mathrm{loc}}(U)$, there is a dense subset $\mathcal{D}_Z$ of $C^{3}(U)$ functions such that for $\varphi\in \mathcal{D}_Z$ and for any lapse-shift pair $(f, X)$ on $U$ solving  
\[
\left(D\Phi_{(g, \pi)}^{(\varphi, Z)}\right)^*(f, X)=0,
\]
it follows that $(f, X)$ further satisfies  both equations 
\begin{align}\label{equation:Zpair}
\begin{split}
D\overline{\Phi}_{(g, \pi)}^*(f, X)&=0 \\
2fZ+|Z|_g X&=0.
\end{split}
\end{align}
\end{theorem}
We refer to the second equation as the \emph{$Z$-null-vector equation} for the following reason:
If $(U, g, \pi)$ sits inside a spacetime with  future unit normal $\mathbf{n}$ and we define a vector field $\mathbf{Y}:=2f\mathbf{n} +X$, then wherever $Z\neq0$, the equation
$2fZ + |Z|_g X=0$ is equivalent to saying that along~$U$, we have
\begin{align*}\label{equation:null-vector}
\mathbf{Y}= \frac{2f}{|Z|_g} (|Z|_g \mathbf{n} - Z),
\end{align*}
which is null. (And of course, the $Z$-null-vector equation is vacuous where $Z=0$.)

The rest of this section is concerned with the proof of this Theorem~\ref{theorem:generic}.
To see why it implies  Theorems~\ref{theorem:improvable} and~\ref{theorem:improvability-error}, skip ahead to the next section.

In this section, we use the abbreviated notation 
\[ L_\varphi:=\left(D\Phi_{(g, \pi)}^{(\varphi, Z)}\right)^*.\]
Although the operator $L_\varphi$ depends on the initial data $(g, \pi)$ and the vector field $Z$ as well as $\varphi$,  in this section we will fix $g$, $\pi$, and $Z$ and focus on the dependence on $\varphi$.

Let us take a moment to motivate our proof of Theorem~\ref{theorem:generic}. Assume that $(f,X)$ is a solution to $L_\varphi(f, X)=0$.  Then $(f, X)$ solves a Hessian-type equation (Lemma~\ref{lemma:Hessian}) so that the second derivatives, and thus higher derivatives, of $(f, X)$ can all be expressed in terms of $(f, df, X, \nabla X)$. By taking divergence of equation \eqref{equation:Hamiltonian} in $L_\varphi(f, X)=0$ and further differentiating, we can construct various linear relations among $(f, df, X, \nabla X)$  so that the vector $(f, df, X, \nabla X)$ evaluated at $p$ lies in the kernel of a square matrix~$Q$.  In general, it is hard to understand the determinant of  $Q$, but it is possible to track its dependence on $\varphi$. One may wish to show that the determinant is a nontrivial polynomial of $\varphi$ and its first few derivatives. However, it is generally not possible because that would imply $L_\varphi$ is injective for generic choice of $\varphi$, which is not true for certain  $Z$.  So instead, we will prove that for generic choice of $\varphi$, either the matrix $Q$ is a nontrivial polynomial of $\varphi$ or the columns of $Q$ have certain linear dependence, which will imply that  the kernel element of $L_\varphi$ must satisfy the $Z$-null-vector equation $2fZ+|Z|_g X=0$. In either case, \eqref{equation:Zpair} of Theorem~\ref{theorem:generic} holds.

In preparation for the proof of Theorem~\ref{theorem:generic}, let us discuss the quantities that will appear in the proof. Since \eqref{equation:Zpair} holds trivially wherever $Z=0$, let us consider the region where $Z\ne0$, and define
\begin{align}\label{equation:W}
 W_i := f\hat{Z}_i +\tfrac{1}{2}X_i ,
\end{align}
where $\hat{Z}=Z/|Z|_g$. Our goal is to show that if $L_\varphi(f, X)=0$, then $W$ must vanish for generic~$\varphi$, and then the other equation $D\overline{\Phi}_{(g, \pi)}^*(f, X)=0$ follows from \eqref{equation:adjoint-Z-mod}.  As discussed above, the proof relies on constructing linear relations among $(f, df, X, \nabla X)$. Observe that by equation~\eqref{equation:momentum},  $\nabla X$ is determined by its antisymmetric part, which we denote by
\begin{align} \label{equation:T}
T_{ij}:=\tfrac{1}{2}(X_{i;j}-X_{j;i}).
\end{align}
Also by replacing $X$ with $W$, any linear system of equations in the quantities $(f, df, X, \nabla X)$ is equivalent to a linear system of equations in the quantities
\[ P:=(W, df, T, f),\]
where we have chosen this ordering for convenience of indexing. (In particular, we order the components of $T$  in some way.) Note that at any point, $P$ is an $N$-dimensional vector, where $N = n + n + \frac{n(n-1)}{2} + 1$, and we will think of $P$ as a column vector. 

Our goal is to construct $N$ linear relations on $P$ so that we obtain a $N\times N$ matrix $Q$ such that $QP=0$ and such that $Q$ is ``as nonsingular as possible.'' 
Fixing a point $p$, if we can arrange that $\det Q$ is a nontrivial polynomial in $\varphi$ and its derivatives, then $L_\varphi$ is generically injective, but as mentioned above, this will not always be the case. In the alternative, we just need to show that for generic $\varphi$, any $P$ in the kernel of $Q$ has vanishing $W$ components. In order to see what must be done, consider the following elementary lemma. 

\begin{lemma}\label{lemma:linear-algebra}
Let $Q$ be an $N\times N$ matrix. Express $Q$ in the following block form:
\[
	Q = \left[\begin{array}{c|c} \widehat{Q}& C\\ 
	\hline R& Q_{NN}
	\end{array}\right]
\]
where $\widehat{Q}$ is the $(N,N)$-minor submatrix of $Q$ of size $(N-1)\times (N-1)$, $Q_{NN}$ is the $(N, N)$-entry of $Q$, $C$ is a $ (N-1)\times 1$ matrix, and $R$ is a $1\times (N-1)$ matrix. Suppose $\widehat{Q}$ is non-singular. Then
\begin{enumerate}
\item  $\Det Q =( \Det\widehat{Q}) (Q_{NN} - R \widehat{Q}^{-1}C)$. \label{item:determinant}
\item  Write $\widehat{Q}^{-1}C = \begin{bmatrix} H_1 \\ \vdots \\ H_{N-1}\end{bmatrix}$. For each $i$ from $1$ to $N-1$, if $H_i=0$, then any  $P=\begin{bmatrix} P_1 \\ \vdots \\ P_N\end{bmatrix}$ solving $QP=0$ must have $P_i=0$.  \label{item:vanish}
\end{enumerate}
\end{lemma}
\begin{proof}
The determinant equality is standard. We prove the second item. The first $(N-1)$ equations of  $QP=0$ say that
\[
	\widehat{Q} \begin{bmatrix} P_1 \\ \vdots \\ P_{N-1} \end{bmatrix} + P_N C=0.
\]
Since $\widehat{Q}$ is nonsingular, this implies that 
\[
	\begin{bmatrix} P_1 \\ \vdots \\ P_{N-1} \end{bmatrix} = -P_N\widehat{Q}^{-1} C = -P_N \begin{bmatrix} H_1 \\ \vdots \\ H_{N-1} \end{bmatrix}.
\]
Thus, $H_i=0$ implies $P_i=0$.

\end{proof}

From this lemma, we can get a good sense of what properties we want our matrix $Q$ (to be constructed) to have.
We would like to show that $\det \widehat{Q}$ is a nontrivial polynomial of $\varphi$ and its derivatives, and that either
\begin{itemize}
 \item $H_1=\cdots=H_n=0$ for all $\varphi$---since the first $n$ components of $P$ are the ones that correspond to $W$; or else
 \item $(Q_{NN} - R \widehat{Q}^{-1}C)$ is a nontrivial rational function of $\varphi$---since that would imply that $\det Q$ is generically nonzero.
 \end{itemize}
 This is the basic description of how the proof works at a single point $p$, and then we use some additional arguments (described in the proofs of Proposition~\ref{proposition:vanishing} and Theorem~\ref{theorem:generic}) to globalize the result. In more detail, the following lemma statement summarizes the specific properties we will need the matrix $Q$ to satisfy, and we will construct $Q$ explicitly within the proof of the lemma.

\begin{lemma}\label{lemma:matrixQ}
Let $(U, g, \pi)$ be an initial data set such that $(g,\pi)\in C^5_{\mathrm{loc}}(U)\times C^4_{\mathrm{loc}}(U)$, and let $Z$ be a vector field in $C^3_{\mathrm{loc}}(U)$ such that $Z\ne0$ on $U$. Suppose $U$ is covered by a single geodesic normal coordinate chart at $p\in U$ so that  at $p$,  $g_{ij} = \delta_{ij}$ and $Z_i=  |Z|_g\delta_i^1$.  Suppose that $\varphi$ is a locally $C^3$ function, and that $(f,X)$ is a lapse-shift pair such that
\[ L_\varphi(f,X)=0\]
on $U$.  In the following, we will use subscripts on functions to denote ordinary differentiation in the coordinate chart, that is, $\varphi_i:= \frac{\partial \varphi}{\partial x^i}$ and similar for $\varphi_{ijk}$, $f_i$, etc., while we will continue to use semicolons to denote covariant derivatives.

Then there exists a $N\times N$ matrix-valued function $Q$ on $U$, which we can express in block form, 
\[
	Q = \left[\begin{array}{c|c} \widehat{Q}& C\\ 
	\hline R& Q_{NN}
	\end{array}\right]
\]
where $\widehat{Q}$ is the $(N,N)$-minor submatrix of $Q$ of size $(N-1)\times (N-1)$, $Q_{NN}$ is the $(N, N)$-entry of $Q$, $C$ is a $ (N-1)\times 1$ matrix, and $R$ is a $1\times (N-1)$ matrix, such that  $Q$ has the following properties:
 \begin{enumerate}
\item $QP=0$, where $P=(W, df, T, f)$ is as described above (thought of as a column vector).
\item The entries of $Q$ are polynomials of $(1, \varphi, \partial\varphi, \partial^2\varphi, \partial^3\varphi)$, whose coefficients depend on up to 5 derivatives of $g$, 4 derivatives of $\pi$, and 3 derivatives of $Z$.

 \item If we decompose 
 \begin{align}\label{equation:widehat-Q0}
 \widehat{Q}= \varphi_1 \widehat{Q}_1+ \widehat{Q}_0,
 \end{align}
 where the entries of $\widehat{Q}_0$ have no dependence on $\varphi_1$, then after evaluating at the point $p$, $\widehat{Q}_1$ has the $[n] + [n] + \left[\frac{n(n-1)}{2}\right]$ block form:
 \begin{equation} \label{equation:widehat-Q}
	\widehat{Q}_1 =  \left[\begin{array}{c|c|c} D_1& 0 & 0\\ 
	\hline * & D_2 & * \\
	\hline * & 0 & D_3
	\end{array}\right]
\end{equation}
where the square matrices $D_1, D_2, D_3$ of size $n, n, \frac{n(n-1)}{2}$, respectively, are diagonal and nonsingular, the $0$'s represent zero matrices, and the asterisks represent arbitrary matrices. In particular, $\widehat{Q}_1$ is nonsingular at $p$.
 \item Third derivatives of $\varphi$ \emph{only} show up in the $R$ block of $Q$, and if we decompose $R=R_1+R_0$, where $R_1$ is linear in $\partial^3 \varphi$ while $R_0$ has no dependence on $\partial^3 \varphi$, then after evaluating at $p$, we have
 \begin{equation*}
 R_1= ( 2\varphi_{111}, \varphi_{211}, \varphi_{311},\ldots, \varphi_{n11}, 0,\ldots, 0 ).
 \end{equation*}
 \end{enumerate}
\end{lemma}

\begin{proof}
In the following, all lengths, products, covariant derivatives, and raising and lowering of indices are computed with respect to $g$.
First we would like to express $(f, df, X, \nabla X)$ in terms of $(W, df, T, f)$: 
\begin{align}
X_i &= 2 (W_i - f \hat{Z}_i)\label{equation:X}\\
X_{i;j} &=  T_{ij}+\left(\tfrac{2}{n-1} (\tr_g \pi) g^{ij} - 2\pi^{ij} \right)f \label{equation:DX}.
\end{align}
The first identity is just the definition of $W$ in \eqref{equation:W}, while the second identity follows from the definition of $T$ in \eqref{equation:T} and~\eqref{equation:momentum}. Next, we also want to be able to express the derivatives of $(W, df, T, f)$ in terms of $(W, df, T, f)$, but keep in mind that our main concern is keeping track of dependence on $\varphi$.
 Differentiating the definition of $W$ and substituting in~\eqref{equation:DX}, we obtain
\begin{align}
	W_{i;j} &= \hat{Z}_i f_j + \tfrac{1}{2} T_{ij} +\left( \hat{Z}_{i;j} + \tfrac{1}{n-1} (\tr_g \pi) g_{ij} - \pi_{ij}\right)f
	\label{equation:DW}\\
	&=  \hat{Z}_i f_j + \tfrac{1}{2} T_{ij} + A_{ij}f,\label{equation:DWalt}
\end{align}
where $A_{ij}$ is a quantity that has no dependence on $\varphi$ or its derivatives, nor any dependence on $(W, df, T, f)$. Meanwhile, using the definition of $W$ and~\eqref{equation:DX}, equation \eqref{equation:Hessian} can be rewritten as
\[
	f_{;ij} = \varphi |Z| \left[ -2(Z\odot W)_{ij} + \tfrac{2}{n-1}  \langle Z, W\rangle g_{ij} \right] +A^{(1)}_{ij}(W, T, f), 
\]
where $A_{ij}^{(1)}(W, T, f)$ is a linear function of $(W, T, f)$ that has no dependence on $\varphi$ or its derivatives. Differentiating the definition of $T$ and using~\eqref{equation:second-derivative} and the definition of $W$, we have 
\[
	T_{ij;k} = {A}^{(2)}_{ijk}(W, df, f), 
\]
where  ${A}^{(2)}_{ijk}(W, df, f)$ is a linear function of $(W, df, f)$ that has no dependence on $\varphi$ or its derivatives. 
\begin{remark}
It will be convenient to keep in mind the fact that when we express the covariant derivatives of $W$, $df$, $T$, and $f$ in terms of   $(W, df, T, f)$, the only $\varphi$ term that shows up is in the coefficient of $W$ in the expression for $f_{;ij}$. 
\end{remark}
We can now begin to construct the desired matrix $Q$. Since $L_\varphi(f, X)=0$, \eqref{equation:Hamiltonian} holds. Using~\eqref{equation:X}  and~\eqref{equation:DX}, equation \eqref{equation:Hamiltonian} can be rewritten in the form
\begin{align*}
	(\Delta f) g_{ij} -f_{;ij} &=
2\varphi |Z| (Z\odot W)_{ij} + A^{(3)}_{ij}(W, T, f)\nonumber\\
&=\varphi |Z| (Z_i\delta_j^\ell+ Z_j \delta_i^\ell)W_\ell + A^{(3)}_{ij}(W, T, f),
\end{align*}
where $A_{ij}^{(3)}(W, T, f)$ is a linear function of $(W, T, f)$ that has no dependence on $\varphi$ or its derivatives. 
By taking the divergence of the equation above (that is, taking $\nabla^i$ of both sides), we obtain
\begin{align*}
\nabla^i \left[(\Delta f) g_{ij} -f_{;ij} \right]
 &=  \varphi_i  |Z|(Z^i \delta_j^\ell+ Z_j g^{i\ell})W_\ell +\varphi |Z| (Z^i \delta_j^\ell+ Z_j g^{i\ell} )W_{\ell; i}  \\
&\quad +\varphi \nabla_i \left[ |Z|  (Z^i \delta_j^\ell +Z_j g^{i\ell} )\right]W_\ell+  g^{i\ell} \nabla_\ell\left[ A^{(3)}_{ij}(W, T, f)\right]\\
-R_{jk}g^{ik} f_i
 &=  \varphi_i  |Z|(Z^i \delta_j^\ell+ Z_j g^{i\ell})W_\ell\\
 &\quad 
 + \varphi |Z| (Z^i \delta_j^\ell+ Z_j g^{i\ell} )(\hat{Z}_\ell f_i + \tfrac{1}{2}T_{\ell i} +A_{\ell i}f)
   \\
&\quad
+ \varphi \nabla_i \left[ |Z|  (Z^i \delta_j^\ell +Z_j g^{i\ell} )\right]W_\ell 
+  g^{i\ell} \nabla_\ell\left[ A^{(3)}_{ij}(W, T, f)\right],
\end{align*}
where we used the Ricci identity on the left side to obtain the Ricci term and 
we used~\eqref{equation:DWalt} to eliminate the $W_{\ell; i}$ term on the right side.
We can use the Remark to understand the dependence of $\nabla_\ell A^{(3)}_{ij}$ term 
on $\varphi$ and combine it with the Ricci term, the $A_{\ell i}$ term, and the second to last term in the equation above in order to obtain (after dividing everything by $|Z|^2$) the equation
\begin{align}\label{equation:divergence2}
\begin{split}
0 &=  \varphi_i  (\hat{Z}^i \delta_j^{\ell}+ \hat{Z}_j g^{i\ell})W_\ell
+ \varphi (\hat{Z}^i \delta_j^{\ell}+ \hat{Z}_j g^{i\ell})(\hat{Z}_\ell f_i + \tfrac{1}{2}T_{\ell i})\\
&\quad 
+ A_j^{(4)}(df, T) +  A_j^{(5)}[\varphi](W, f), 
\end{split}
\end{align}
where $A_j^{(4)}(df, T)$ is a linear function of $(df, T)$ that has no dependence on $\varphi$ or its derivatives and  $A_j^{(5)}[\varphi](W, f)$ is a linear function of $(W, f)$ that depends on $\varphi$ but not any of its derivatives. (Note that our proof requires us to track certain dependencies on $\varphi$ but not all of them.)

\vspace{6pt}
\noindent \textbf{Construction of rows $1$ to $n$.} For each $j$ from $1$ to $n$, we define the $j$-th row of $Q$ to be the coefficients obtained from the linear relation~\eqref{equation:divergence2}. Recall the definition of $\widehat{Q}_1$ in \eqref{equation:widehat-Q0}. Evaluating at $p$, we can compute the first $n$ rows of $\widehat{Q}_1$ by focusing on the $\varphi_i$ term in~\eqref{equation:divergence2} with $i=1$. Since $\varphi_1$ can only appear in coefficients of $W$, we see that for $j=1$ to $n$, we have $ \left(\widehat{Q}_1\right)_j^\ell =0$ for $\ell>n$ (corresponding to the coefficients of $df$, $T$, and  $f$).  And for $\ell=1$ to $n$ (corresponding to the $W_\ell$ coefficients), 
\begin{align*}
 \left(\widehat{Q}_1\right)_j^\ell =
 (\hat{Z}^1 \delta_j^{\ell}+ \hat{Z}_j g^{\ell 1}) 
= \delta_j^{\ell}+ \delta_j^1 \delta^\ell_1,
\end{align*}
showing that the matrix $D_1$ is diagonal with $\det D_1=2$. This verifies the claims made about the top line of blocks in~\eqref{equation:widehat-Q}.

\vspace{6pt}
\noindent \textbf{Construction of rows $n+1$ to $2n$.}
We take the covariant derivative $\nabla_k$ of both sides of~\eqref{equation:divergence2}. We claim that doing this will give us
\begin{align*}
0 &=  \varphi_{;ik}  (\hat{Z}^i \delta_j^{\ell}+ \hat{Z}_j g^{i\ell})W_\ell
+ \varphi_{i}  (\hat{Z}^i \delta_j^{\ell}+ \hat{Z}_j g^{i\ell})W_{\ell;k}\\
&\quad + \varphi_{k} (\hat{Z}^i \delta_j^{\ell}+ \hat{Z}_j g^{i\ell})(\hat{Z}_\ell f_i + \tfrac{1}{2}T_{\ell  i})\\
&\quad+ A_{jk}^{(6)}[\varphi](df, T) +  A_{jk}^{(7)}[\varphi, \partial\varphi](W, f), 
\end{align*}
where $A_{jk}^{(6)}[\varphi](df, T)$ is a linear function of $(df, T)$ that depends on $\varphi$ but not any of its derivatives and  $A_{jk}^{(7)}[\varphi, \partial\varphi](W, f)$ is a linear function of $(W, f)$ that depends on $\varphi$ and its first derivatives, but no higher derivatives. This can be seen by applying the product rule to the first two terms of~\eqref{equation:divergence2} and noting that each resulting term not appearing explicitly above can be taken to be part of the $A^{(6)}$ or $A^{(7)}$ terms. Meanwhile, our Remark implies that $\nabla_k$ of the $A^{(4)}$ and $A^{(5)}$ terms can be taken to be part of the $A^{(6)}$ and $A^{(7)}$ terms. Next we use~\eqref{equation:DWalt} to eliminate the $W_{\ell;k}$ and collect like terms to obtain
\begin{align}\label{equation:jk}
 \tag*{(\theequation)$_{\jksub}$}
\begin{split}
0 &=  \varphi_{;ik}  (\hat{Z}^i \delta_j^{\ell}+ \hat{Z}_j g^{i\ell})W_\ell
+2 \hat{Z}^i \hat{Z}_j (\varphi_{i} \delta_k^m + \varphi_{k} \delta_i^m ) f_m \\
&\quad +\tfrac{1}{2} (\hat{Z}^i \delta_j^{\ell}+ \hat{Z}_j g^{i\ell})( \varphi_{i} \delta_k^m + \varphi_k \delta_i^m)  T_{\ell m} \\
&\quad+ A_{jk}^{(6)}[\varphi](df, T) +  A_{jk}^{(8)}[\varphi, \partial\varphi](W, f), 
\end{split}
\end{align}
where $A_{jk}^{(8)}[\varphi, \partial\varphi](W, f)$ is a linear function of $(W, f)$ that depends on $\varphi$ and its first derivatives, but no higher derivatives.  We will use the above equations to define rows  
$n+1$ to $N-1$ of $Q$. 
 Fixing $j=1$ in \ref{equation:jk}, for each $k=1$ to $n$, we define the $(n+k)$-th row of $Q$  to be the coefficients obtained from the linear relation {\renewcommand\jksub{1k}\ref{equation:jk}}. 
Evaluating at $p$, we can compute $ \left(\widehat{Q}_1\right)_{n+k}^{n+m}$ for $k, m=1$ to $n$ (corresponding to the $f_m$ coefficients) by focusing on the dependence on $\varphi_1$.
 \begin{align*}
\left(\widehat{Q}_1\right)_{n+k}^{n+m} &=
2  \hat{Z}^1 \hat{Z}_1 \delta_k^m + 2\hat{Z}^i \hat{Z}_1  \delta_k^1\delta_i^m \\
&= 2( \delta_k^m +  \delta_k^1\delta^m_1)
 \end{align*}
showing that the matrix $D_2$ is diagonal with $\det D_2=2^{n+1}$, and thus verifying the claim made about the middle line of blocks in~\eqref{equation:widehat-Q}. (Note that $\varphi_1$ will not appear in the $\varphi_{;ik}$ term since we are using normal coordinates at $p$.)

\vspace{6pt}
\noindent \textbf{Construction of rows $2n+1$ to $N-1$.}
To specify the ordering of the components of $T_{jk}$ within the vector $P$, we let $\iota$ be a bijection from $\mathcal{A}:=\{ (j, k)\,|\, n\ge j > k \ge 1\}$ to $\{2n+1, 2n+2,\ldots, N-1\}$ so that  $P_{\iota(j, k)}=T_{jk}$. We define rows $2n+1$ to $N-1$ of $Q$ by defining the $\iota(j,k)$ row of $Q$ using the coefficients of the linear relation  \ref{equation:jk}, for all $(j, k)\in\mathcal{A}$. 

Evaluating at $p$, we will compute $ \left(\widehat{Q}_1\right)^{\iota(\ell, m)}_{\iota(j, k)}$ for $(j, k), (\ell, m)\in\mathcal{A}$, corresponding to the $T_{\ell m}$ coefficients in~\ref{equation:jk}. 
Since each $(j,k) \in\mathcal{A}$ has $j>1$, it follows that $\hat{Z}_j=0$, so we can ignore all of the $\hat{Z}_j$ terms appearing in~\ref{equation:jk}. 
By focusing on the dependence on $\varphi_1$ and taking into account antisymmetry of $T$, we obtain
 \begin{align*}
\left(\widehat{Q}_1\right)^{\iota(\ell, m)}_{\iota(j, k)}&=\tfrac{1}{2}
\left( \hat{Z}^1 \delta_j^\ell \delta_k^m +
  \hat{Z}^i \delta_j^\ell \delta_k^1\delta_i^m\right) -\tfrac{1}{2}
\left(  \hat{Z}^1 \delta_j^m  \delta_k^\ell
 +\hat{Z}^i \delta_j^m  \delta_k^1\delta_i^\ell\right)\\
 &= \tfrac{1}{2} (  \delta_j^\ell \delta_k^m  +   \delta_j^\ell \delta_k^1 \delta^m_1
  - \delta_j^m\delta_k^\ell
 - \delta_j^m \delta_k^1\delta^\ell_1   )\\
 &= \tfrac{1}{2} (  \delta_j^\ell \delta_k^m + \delta_j^\ell \delta_k^1 \delta^m_1 ),
 \end{align*} 
where the last two terms of the second line vanish because $(j, k), (\ell, m)\in\mathcal{A}$. This tells us that the matrix $D_3$ is diagonal with nonzero diagonal entries. The fact that $j>1$ implies $\hat{Z}_j=0$ also explains why  $\left(\widehat{Q}_1\right)^{n+m}_{\iota(j, k)}=0$ for all $(j,k)\in\mathcal{A}$ and $m=1$ to $n$ (corresponding to the $f_m$ coefficients). We have now established all of our claims about $\widehat{Q}_1$.

\vspace{6pt}
\noindent\textbf{Construction of row $N$.} Using the same reasoning as above,  it is easy to see that if we take the covariant derivative $\nabla_q$ of 
\ref{equation:jk}, we obtain
\begin{align*}
\begin{split}
0 &=  \varphi_{;ikq}  (\hat{Z}^i \delta_j^{\ell}+ \hat{Z}_j g^{i\ell})W_\ell + A_{jkq}^{(9)}[\varphi, \partial\varphi, \partial^2\varphi](W, df, T, f),
\end{split}
\end{align*}
where  $A_{jkq}^{(9)}[\varphi, \partial\varphi, \partial^2\varphi](W, df, T, f)$ is a linear function of $(W, df, T, f)$ depending on at most two derivatives of $\varphi$. We define the $N$-th row of $Q$ using the coefficients of the equation above when $j=k=q=1$. Recalling the definition of $R_0$ and $R_1$ in Item (4), 
it is clear that only the first $n$ components of $R_1$ can be nonzero, and more precisely, evaluating $R_1$ at $p$, for $\ell=1$ to $n$, we have
\begin{align*}
 (R_1)^\ell &= \varphi_{i11}( \hat{Z}^i \delta_1^{\ell} + \hat{Z}_1 \delta^{i\ell})\\
&= \varphi_{111} \delta_1^{\ell}  +  \varphi_{i 11} \delta^{i\ell},
\end{align*}
verifying the claim made about the form of $R_1$. (Note that any discrepancy between  $\varphi_{i11}$ and  $\varphi_{;i11}$ lies in $R_0$.)

Observe that Item (1) follows from our construction of $Q$, Item (2) follows from tracking the number of times we differentiated in our construction of $Q$, and we explicitly checked Items (3) and~(4) above, and therefore the proof is complete.

\end{proof}



For a scalar-valued function $\varphi$, let $\mathbb{J}^3_q \varphi$ denote the $3$-jet\footnote{Note that this definition implicitly depends on choice of coordinate chart, but we will see that this causes no problems in our proof.}
 of  $\varphi$ at the point~$q$: 
\[
 \mathbb{J}^3_q \varphi:= (\varphi(q), \partial \varphi(q),  \partial^2 \varphi(q), \partial^3 \varphi(q)) \in \mathbb{R}^{1+n+\binom{n}{2} + \binom{n}{3} }.
 \]

\begin{proposition}\label{proposition:vanishing}
Let $(U, g, \pi)$ be an initial data set such that $(g,\pi)\in C^5_{\mathrm{loc}}(U)\times C^4_{\mathrm{loc}}(U)$, and let $Z$ be a vector field  in $C^3_{\mathrm{loc}}(U)$ such that $Z\ne0$ on $U$. Given $p\in U$, there exists a coordinate chart $U_p\subset U$, a point $q\in U_p$, and a zero set $s_q\subset \mathbb{R}^{1+n+\binom{n}{2}+\binom{n}{3}}$ of a nontrivial polynomial such that  for any $C^3$ function $\varphi$ on $U$ with $\mathbb{J}_q^3\varphi \notin s_q$,  if $(f, X)$ solves $L_\varphi(f, X)=0$ in $U$, then $2fZ+ |Z|_g X=0$ in  $U_p$. 
\end{proposition}

\begin{proof}
Choose $U_p$ to be a normal coordinate chart at $p$ so that $g_{ij} = \delta_{ij}$ and $Z_i = |Z|_g \delta^1_i$ at $p$ as in Lemma~\ref{lemma:matrixQ}, accepting that we may need to shrink $U_p$ later.
Let $Q$, $Q_{NN}$, $\widehat{Q}$, $C$, and $R$ denote the corresponding matrices constructed in Lemma~\ref{lemma:matrixQ}. By Lemma~\ref{lemma:matrixQ}, $QP=0$, and each entry of $Q$ is a polynomial of $(1, \varphi, \partial \varphi, \partial^2\varphi, \partial^3 \varphi)$ whose coefficients depend on up to 5 derivatives of $g$, 4 derivatives of $\pi$, and 3 derivatives of $Z$.

Using the variables $w, y, z, \xi$ to denote $\varphi, \partial \varphi, \partial^2\varphi, \partial^3 \varphi$, respectively, we can write
\[
 \det \widehat{Q} = F (\varphi, \partial \varphi, \partial^2 \varphi),
\] 
where $F(w, y, z)$ is an $(N-1)$-degree polynomial in $w$, $y$, and $z$, whose coefficients are functions of $x\in U_p$ (which depend on $g$, $\pi$, and $Z$).  We know that $F$ is not the zero polynomial at the point $p$ because the coefficient of $\varphi_1^{N-1}$ in  $\det \widehat{Q}$ is precisely $\det \widehat{Q}_1$, which we saw is nonzero in Item (3) of Lemma~\ref{lemma:matrixQ}. By shrinking $U_p$ as needed, continuity guarantees that we may assume that at every $x \in U_p$, $F$ is not the zero polynomial.  Therefore we may express
\[
 \widehat{Q}^{-1} C =	\begin{bmatrix} H_1(\varphi, \partial \varphi, \partial^2 \varphi)  \\ \vdots \\ H_{N-1}(\varphi, \partial \varphi, \partial^2 \varphi)   \end{bmatrix},
\]
where  $H_1(w, y, z), \dots, H_{N-1} (w, y, z)$ are rational functions of $w$, $y$, and $z$, whose coefficients are functions of $x\in U_p$.  In fact, we can write 
\[ H_i (w, y, z) = \frac{\eta_i(w, y, z)}{F(w, y, z)},\]
 where each $\eta_i$ is a polynomial.
Similarly, we can write 
\begin{align*} 
	Q_{NN} - R \widehat{Q}^{-1}C= \frac{G(\varphi, \partial \varphi, \partial^2 \varphi, \partial^3 \varphi)}{F(\varphi, \partial \varphi, \partial^2 \varphi)},
\end{align*}
where $G(w, y, z, \xi)$ is a polynomial in $w$, $y$, $z$, and $\xi$, whose coefficients are functions of $x\in U_p$.

  We will define $q$ and $s_q$ according to the following cases.\\ 
\noindent{\bf Case 1:} Suppose that the polynomials $\eta_1, \dots, \eta_n$ are identically zero in some small neighborhood of $p$. 

In this case, we take $U_p$ to be this neighborhood, we choose $q=p$, and we define
\[
s_p:= \left\{ \left.(w, y, z, \xi) \in \mathbb{R}^{1+n+\binom{n}{2} + \binom{n}{3}}\,\right|\, F(w, y, z)= 0 \mbox{ at } p\right\}
\]
to be the zero set of the nontrivial polynomial $F$ at $p$.
 As long as $\mathbb{J}_p^3\varphi \notin s_p$, we have $\det \widehat{Q}=F \ne0$ at $p$. Then since $QP=0$ and $H_i = \frac{\eta_i}{F}=0$ at each point of $U_p$, for $i=1$ to $n$, Item~\eqref{item:vanish} of Lemma~\ref{lemma:linear-algebra} implies that $P_1 = \dots= P_n=0$ in $U_p$. This just says that $W=0$ in $U_p$, which is the the same as saying that $2fZ+|Z|_g X=0$ in $U_p$. 

\noindent{\bf Case 2:} The alternative to Case 1 is that there exists a sequence of points $q_k \to p$ so that at least one of the polynomials $\eta_1, \dots, \eta_n$ is not the zero polynomial at $q_k$. 

By our definitions of $G$ and $\eta$,
\begin{equation*}
 G= F Q_{NN} - (R^1 \eta_1 +\dots+ R^{N-1} \eta_{N-1}),
 \end{equation*}
   where we write $R = \begin{bmatrix} R^1 & R^2 & \cdots & R^{N-1}\end{bmatrix}$. We would like to show that $G(x, y, z, \xi)$ is not a zero polynomial at some points near $p$. The subtlety in the following argument is because all $\eta_i$ may converge to the zero polynomial at $p$, so we have to analyze $G$ at the sequence $q_k$, instead of the limit point $p$. To do this,  we take a subsequence of $q_k$ (still denoted by $q_k$) such that one of these polynomials, say $\eta_{j_0}$ for some fixed $j_0$, satisfies both of the following properties, at all $q_k$: (1)  $\eta_{j_0}$   is not the zero polynomial, and (2)  the largest coefficient (in absolute value) of $\eta_{j_0}$ at $q_k$ is greater than or equal to the absolute value of every coefficient of every $\eta_i$ for $i=1$ to $n$.

 We \emph{claim} that $G(w, y, z, \xi)$ is not the zero polynomial at $q_k$ for $k$ large.  Given $\epsilon\in (0, \frac{1}{4})$, by Item~(4) of Lemma~\ref{lemma:matrixQ} and continuity, we know that for $k$ large enough, at $q_k$, the coefficient of $\varphi_{j_0 11}$ in $R^{j_0}$ is bounded below by~$\tfrac{1}{2}$, while the coefficient of $\varphi_{j_0 11}$ in $R^{i}$ for other $i$ is bounded between by~$\pm \epsilon$. Moreover, from the construction of $R$, for all $i>n$,  $R^{i}$ has no dependence on $\varphi_{j_0 11}$.  Putting these facts together and using the fact that $\eta_{j_0}$ at $q_k$ has the largest possible coefficient among the $\eta_i$'s, we can see that the polynomial $R^1 \eta_1 +\dots+ R^{N-1} \eta_{N-1}$ at $q_k$ must have a nonzero coefficient of one of the monomials involving $\varphi_{j_0 11}$. Finally, since $FQ_{NN}$ has no dependence on $\varphi_{j_0 11}$ at any point by construction, it follows that 
  $G(w, y, z, \xi)$ is a nontrivial polynomial function at $q_k$ for sufficiently large $k$. 

Choose $q$ to be one of these $q_k$'s for large enough $k$, and define
\begin{align*}
	s_{q} &= \left\{ \left.(w, y, z, \xi) \in \mathbb{R}^{1+n+\binom{n}{2} + \binom{n}{3}}\,\right|\, F(w, y, z) G(w, y, z, \xi)=0 \mbox{ at } q\right\}
\end{align*}
to be the zero set of the nontrivial polynomial $FG$ at $q$.  If $\mathbb{J}_q^3\varphi \notin s_q$, then we see that both $\det \widehat{Q}\ne0$ and $Q_{NN} - R \widehat{Q}^{-1}C\ne0$ at $q$ by tracking back through the definitions. So by Item \eqref{item:determinant} of Lemma~\ref{lemma:linear-algebra}, $\det Q\ne0$ at $q$.
Since we have $QP=0$, it follows that $P=0$ at $q$. But this means that we have a solution 
$(f, X)$ to $L_\varphi(f, X)=0$ such that $(W, df, T, f)$ vanishes at $q$. As discussed earlier, this implies that $(f, df, X, \nabla X)$ vanishes at $q$.  Hence $(f, X)$ vanishes identically on all of $U_p$ since $(f, X)$ satisfies a Hessian-type equation  (Lemma~\ref{lemma:Hessian}) and is uniquely determined by its 1-jet $(f, df, X, \nabla X)$ at a point. 

\end{proof}

We now prove the main result in this section, Theorem~\ref{theorem:generic}.

\begin{proof}[Proof of Theorem~\ref{theorem:generic}]
We just need to describe the set $\mathcal{D}_Z$ and prove that if $\varphi\in \mathcal{D}_Z$ and \hbox{$L_\varphi(f, X)=0$}, then the $Z$-null-vector equation $2fZ+|Z|_g X=0$ holds everywhere in $U$. By equation~\eqref{equation:adjoint-Z-mod}, this will imply that $D\overline{\Phi}^*_{(g, \pi)}(f, X)=0$. Since the $Z$-null-vector equation holds trivially where $Z=0$, let $V\subset U$ be the open subset where $Z\ne0$.

For each $p\in V$, we apply Proposition~\ref{proposition:vanishing} to obtain a coordinate chart $V_p$, a point $q\in V_p$, and a zero set $s_q\subset \mathbb{R}^{1+n+\binom{n}{2}+ \binom{n}{3}}$ of a nontrivial polynomial as described in the statement of the Proposition. These $V_p$'s cover $V$, so by second countability of manifolds, there is a sequence of points $p_i$ such that the $V_{p_i}$'s cover $V$.
For each~$i$, define $\mathcal{V}_i \subset C^3(U)$  by 
\[
	\mathcal{V}_i := \{ \left. \varphi \in C^3(U)\,\right| \,\mathbb{J}^3_{q_i} \varphi \notin s_{q_i}\}.
\] 
Since $s_{q_i}$ is the zero set of a nontrivial polynomial, it is clear that $\mathcal{V}_i$ is open and dense.  Since $C^3(U)$ is a complete metric space, by the Baire Category Theorem, 
\[ \mathcal{D}_Z := \bigcap_{i=1}^\infty \mathcal{V}_i\] 
is dense in $C^{3}(U)$. By Proposition~\ref{proposition:vanishing},  for  $\varphi\in\mathcal{D}_Z$, any $(f, X)$ solving $L_\varphi(f, X) = 0$  must have $2fZ+|Z|_g X=0$ in each $V_{p_i}$ and hence everywhere in $V$, completing the proof.

\end{proof}

\section{Improvability of the dominant energy scalar}\label{section:improvability}

 We will  use the special case $Z=J$ of Theorem~\ref{theorem:generic} to prove our improvability result, Theorem~\ref{theorem:improvable}, and more general choices of $Z$ to prove our almost improvability result, Theorem~\ref{theorem:improvability-error}. 

\begin{proof}[Proof of Theorem \ref{theorem:improvable}]
We apply Theorem~\ref{theorem:generic} in the case $U=\Int\Omega$ and $Z=J$, where $J$ is the current density of the initial data set $(g, \pi)$. Choose $\varphi$ in the dense set  $\mathcal{D}_J\subset C^3(\Int\Omega)$ as in the statement of Theorem~\ref{theorem:generic}, such that $|\varphi J|_g < (\sqrt{2}-1)/2$. Either $\left(D\Phi^{(\varphi, J)}_{(g, \pi)}\right)^*$ is injective on $\Int\Omega$ or it is not. 
If it is not injective, then there exists a nontrivial lapse-shift pair $(f, X)$ in the kernel, and by Theorem~\ref{theorem:generic}, it follows that the system~\eqref{equation:pair} holds.

On the other hand, if $\left(D\Phi^{(\varphi, J)}_{(g, \pi)}\right)^*$ is injective, we claim that the dominant energy scalar is improvable in $\Int\Omega$.  Assuming injectivity, we may apply Theorem~\ref{theorem:deformationZ} (letting $\Upsilon=0$) to find a neighborhood $\mathcal{U}$ of $(g, \pi)$ in  $C^{4,\alpha}({\Omega}) \times C^{3,\alpha}({\Omega})$ and a neighborhood $\mathcal{W}$ of $(\varphi, J)$ in $C^{2,\alpha}(\Omega)$
such that for any $V\subset\subset \Int\Omega$, there exist constants $\epsilon, C>0$ such that for $(\gamma,\tau)\in \mathcal{U}$, $(\varphi', Z')\in\mathcal{W}$, and $u\in C^{0,\alpha}_c(V)$ with $\| u \|_{C^{0,\alpha}({\Omega})}<\epsilon$, there exists $(h, w)\in C^{2,\alpha}_{c}(\Int\Omega)$ with  $\Omega$ and $\| (h, w) \|_{C^{2,\alpha}(\Omega)} \le C\| u \|_{C^{0,\alpha}(\Omega)}$ such that 
\begin{align} \label{equation:surjective}
\Phi^{(\varphi', Z')}_{(\gamma, \tau)} (\gamma+h, \tau+w  ) = \Phi^{(\varphi',Z')}_{(\gamma, \tau)} (\gamma, \tau ) + (u, 0).
\end{align}
By shrinking $\mathcal{U}$ if necessary, we can guarantee that $(\varphi, J(\gamma, \tau))\in \mathcal{W}$ for all $(\gamma, \tau)\in \mathcal{U}$ and that  $|\varphi J(\gamma, \tau)|_\gamma < (\sqrt{2}-1)/2$. In particular we can solve \eqref{equation:surjective} for the operator $\Phi^{(\varphi, J(\gamma, \tau))}_{(\gamma, \tau)}$. For $\epsilon$ small enough, $|h|_\gamma<1$, and we can invoke Lemma~\ref{lemma:sigma_preserve} to see 
that 
\[ \Sc(\gamma+h, \tau+w) \ge \Sc(\gamma, \tau)  + u.\] 
In summary, we have shown that the dominant energy scalar is improvable in $\Int\Omega$.

\end{proof}


\begin{proof}[Proof of Theorem~\ref{theorem:improvability-error}]
Choose $B$ and $\delta>0$ as in the hypotheses of Theorem~\ref{theorem:improvability-error}. First we will argue that we can select a vector field  $Z$ supported in $B$ with $|Z|_g<1$, such that for all $\varphi\in\mathcal{D}_Z$, $\left(D\Phi^{(\varphi, Z)}_{(g, \pi)}\right)^*$ is injective, where $\mathcal{D}_Z \subset C^3(\Int\Omega)$ is the dense set described in Theorem~\ref{theorem:generic}.  Consider the kernel of $D\overline\Phi_{(g,\pi)}^*$ in $\Int\Omega$. Since this kernel is finite dimensional (since a kernel element is determined by its 1-jet at a point), it is easy to select $Z$ supported in $B$ such that $2fZ+|Z|_g X$ is  \emph{not identically zero} for \emph{all} nontrivial $(f, X)\in \ker D\overline\Phi_{(g,\pi)}^*$. To see how, for each such $X$, all we need is for $Z$ to be linearly independent from $X$ at some point of $B$, and this is easy to arrange since
 $Z$ is free to change however we like from point to point,  giving us infinitely many degrees of freedom in constructing $Z$. (The exception is when $X$ vanishes in $B$, but then $f$ does not vanish, and it is still easy to choose $Z$.) Furthermore, $Z$ can be selected with $|Z|_g<1$ since scaling it down by a constant factor does not affect the direction of $Z$.
 
 We claim that with this choice of $Z$, $\left(D\Phi^{(\varphi, Z)}_{(g, \pi)}\right)^*$ is injective for all  $\varphi\in\mathcal{D}_Z$. If it is not, there exists a nontrivial $(f, X)\in \ker \left(D\Phi^{(\varphi, Z)}_{(g, \pi)}\right)^*$. Then by Theorem~\ref{theorem:generic},  $(f, X)\in \ker D\overline\Phi_{(g,\pi)}^*$ and  $2fZ+|Z|_g X$ is identically zero. But this contradicts our construction of $Z$. 

Select $\varphi\in \mathcal{D}_Z$ small enough so that  $|\varphi|<1$. By injectivity of $\left(D\Phi^{(\varphi, Z)}_{(g, \pi)}\right)^*$, we can apply
Theorem~\ref{theorem:deformationZ} (with $\Upsilon=0$ and $(\varphi', Z')=(\varphi, Z)$). That is, there exists a neighborhood $\mathcal{U}$ of $(g, \pi)$ in $C^{4,\alpha}({\Omega})\times C^{3,\alpha}({\Omega})$ such that for any $V\subset\subset\Int\Omega$, there exist constants $\epsilon, C>0$ such that for $(\gamma,\tau)\in \mathcal{U}$ and for $u\in C_c^{0,\alpha}(V)$ with $\| u \|_{C^{0,\alpha}(\Omega)}<\epsilon$, there exists $(h, w)\in C^{2,\alpha}_c(\Int\Omega)$ with $\| (h,w)\|_{C^{2,\alpha}(\Omega)} \le C \| u \|_{C^{0,\alpha}(\Omega)}$ such that 
\begin{align*} 
	\Phi^{(\varphi, Z)}_{(\gamma,\tau)} (\gamma+h, \tau+w) = \Phi^{(\varphi, Z)}_{(\gamma,\tau)} (\gamma, \tau) + (u, 0).
\end{align*}
By shrinking $\mathcal{U}$ if necessary and choosing $\epsilon$ small enough, we will have $|Z|_\gamma<1$, $|h|_\gamma<1$, and we can apply Lemma~\ref{lemma:sigma_error} to see that
\[ 
\Sc(\bar{\gamma}, \bar{\tau})  \ge \Sc(\gamma,\tau)+ u - 6|h|_\gamma^{\frac{1}{2}}|\varphi|^{\frac{1}{2}} |Z|_\gamma (|J|_\gamma^{\frac{1}{2}} + 1).
\]
By further shrinking $\epsilon$, we can make $|h|_\gamma$ small enough so that the error term above is bounded by $\delta$, and since $Z$ is supported in $B$ it will be bounded by $ \delta {\bf 1}_B$, as desired.

\end{proof}

\section{Consequences of the $J$-null-vector equation} \label{section:null}

In this section, we study properties of an initial data set that carries a lapse-shift pair  $(f, X)$  satisfying the system:
\begin{align*} \tag{\ref{equation:pair}}
\begin{split}
D\overline{\Phi}_{(g, \pi)}^*(f, X)&=0\\
2fJ+|J|_g X&=0.
\end{split}
\end{align*} 

\subsection{Null perfect fluids and Killing developments}\label{section:null_perfect}

Recall that we defined a spacetime   $(\mathbf{N}, \mathbf{g})$ to be a \emph{null perfect fluid spacetime}\footnote{In the literature, a \emph{perfect fluid spacetime}  refers a spacetime whose the Einstein tensor $G_{\alpha\beta} = p \mathbf{g}_{\alpha\beta} + (\rho+p) v_\alpha v_\beta$ with the mass density $\rho$, pressure $p$, and a unit \emph{timelike} vector $\mathbf{v}$ that represents the velocity of the fluid. We define null perfect fluid analogously to perfect fluid but with the velocity $\mathbf{v}$ either future null or zero. We also absorb the coefficient $\rho+p$ into $\mathbf{v}$ as there is no preferred normalization for a null vector.  Note that an alternative   form $G_{\alpha\beta} = p \mathbf{g}_{\alpha\beta} - v_\alpha v_\beta$ (with the minus sign) is not ``physical'' as it cannot satisfy the spacetime DEC.}  with velocity $\mathbf{v}$ and pressure $p$ if the Einstein tensor takes the form:
\begin{align}\label{equation:perfect-fluid-Einstein}
	G_{\alpha\beta} = p \mathbf{g}_{\alpha\beta} + v_\alpha v_\beta,
\end{align}
where $p$ is a scalar function, and the velocity $\mathbf{v}$ is either future null or zero at each point of $\mathbf{N}$. We derive general properties of null perfect fluids.

  \begin{lemma}\label{lemma:null-perfect-fluid}
Let  $(\mathbf{N}, \mathbf{g})$ be a null perfect fluid spacetime with  velocity $\mathbf{v}$ and pressure $p$, and $\mathbf{g}\in C^3_{\mathrm{loc}}(\mathbf{N})$.  Then the following properties hold:
\begin{enumerate}
\item \label{item:spacetime-DEC} $(\mathbf{N}, \mathbf{g})$ satisfies the spacetime DEC if and only if $p\le0$.
\item \label{item:perfect-fluid}  Let $U$ be a spacelike hypersurface in $\mathbf{N}$ with induced initial data $(g, \pi)$. Denote by $\mathbf{n}$ the future unit normal to $U$ and by  $(\mu, J)$ the energy and current densities of $(g, \pi)$. Then along~$U$, 
\begin{align}\label{equation:v+p}
\begin{split}
\mathbf{v} &= \left\{ \begin{array}{ll} \frac{1}{\sqrt{|J|_g} }(|J|_g \mathbf{n} -J) &\mbox{where } J\neq 0\\
0 &\mbox{where } J=0 \end{array}\right.\\
p&= |J|_g -\mu.
\end{split}
\end{align}
Or equivalently, with respect to an orthonormal basis $e_0, e_1, \dots, e_n$ along $U$, where $e_0 = \mathbf{n}$,  the Einstein tensor takes the form: 
\begin{align}	\label{equation:perfect-fluid}
\begin{split}
G_{00} &=\mu\\
G_{0i}& = J_i\\
G_{ij} &=\left\{  \begin{array}{ll}(|J|_g - \mu) g_{ij}+\frac{J_i J_j}{|J|_g} & \mbox{ where } J\neq 0\\
-\mu g_{ij} & \mbox{ where } J=0.\end{array}\right.
\end{split}
\end{align}
\item If $(\mathbf{N}, \mathbf{g})$ admits a Killing vector field $\mathbf{Y}$ and $\mathbf{v} = \eta \mathbf{Y}$ for some function $\eta$, then the pressure $p$ is constant on $\mathbf{N}$. \label{item:pressure}

\end{enumerate}
 \end{lemma}
\begin{proof}
Let $e_0, e_1, \dots, e_n$ be an orthonormal frame such that $e_0$ is future unit timelike. Writing the velocity vector $\mathbf{v}=v^0 e_0 +\cdots+v^n e_n$ with respect to this frame, note that its dual covector has $v_0 = -v^0$ while $v_i=v^i$ for $i=1$ to $n$.

Proof of \eqref{item:spacetime-DEC}: Since $\mathbf{g}(\mathbf{u}, \mathbf{w}) \le 0$ for any pair of future causal vectors $\mathbf{u}$ and $\mathbf{w}$, it is easy to see that the spacetime DEC holds if $p\le 0$. Conversely, suppose the spacetime DEC holds. At the points where $\mathbf{v}=0$, the fact that $G(e_0, e_0)\ge 0$ implies $p\le 0$. At points where $\mathbf{v}$ is not zero, define the future null $\bar{\mathbf{v}}:= v^0 e_0 -(v^1 e_1+ \cdots +v^n e_n)$. The spacetime DEC tells us that $G(\bar{\mathbf{v}}, \mathbf{v})\ge 0$, which then implies that $p\le 0$.

Proof of \eqref{item:perfect-fluid}: Let  $e_0 = \mathbf{n}$ be the future unit normal to $U$. 
Combining the definitions of $\mu$ and $J$ (or more precisely,~\eqref{equation:Einstein-0i}) with equation~\eqref{equation:perfect-fluid-Einstein}, we obtain
\begin{align*}
\mu &= G_{00} =-p+v_0^2  \\
J_i & = G_{0i}=v_0  v_i.
\end{align*}
Since $\mathbf{v}$ is future null or zero, we can conclude that 
$|J|_g =v_0^2$, and  $ v_0 J_i = |J|_g v_i$. These equations easily imply the desired expressions for  $\mathbf{v}$ and $p$.  The desired expression for $G_{ij}$ follows by substituting these expressions for $\mathbf{v}$ and $p$ back into \eqref{equation:perfect-fluid-Einstein}.

Proof of \eqref{item:pressure}: Express the Einstein tensor as $G_{\alpha\beta} = p \mathbf{g}_{\alpha\beta} + \eta^2 Y_\alpha Y_\beta$. Taking the Lie derivative with respect to the Killing vector $\mathbf{Y}$,
\[
0= \mathcal{L}_{\mathbf{Y}} G_{\alpha\beta}
= (\mathcal{L}_{\mathbf{Y}}p) \mathbf{g}_{\alpha\beta} +  (\mathcal{L}_{\mathbf{Y}}\eta^2 ) Y_\alpha Y_\beta.
\]
At every point where $\mathbf{Y}$ is not zero, we can easily find vectors $\mathbf{u}$ and $\mathbf{w}$ such that  $\mathbf{g}(\mathbf{u}, \mathbf{w})=0$ while $\mathbf{g}(\mathbf{Y}, \mathbf{u})$ and $\mathbf{g}(\mathbf{Y}, \mathbf{w})$ are nonzero. Plugging these into the above equation shows that $\mathcal{L}_{\mathbf{Y}}\eta^2=0$ everywhere. Next, we
use the contracted second Bianchi identity to obtain
\begin{align*}
 0 = \bm{\nabla}^\beta G_{\alpha\beta}& =   \bm{\nabla}_\alpha p +  (\bm{\nabla}^{\beta} \eta^2) Y_\alpha Y_\beta+ 
\eta^2 (\bm{\nabla}^\beta Y_\alpha )Y_\beta
+
\eta^2 Y_\alpha (\Div_{\mathbf{g}} \mathbf{Y})\\
& =   \bm{\nabla}_\alpha p,
\end{align*}
where we used the facts that $\mathcal{L}_{\mathbf{Y}}\eta^2=0$, $\mathbf{Y}$ is Killing, and $\mathbf{Y}$ is null or zero wherever $\eta$ is not zero. Hence $p$ is constant.

\end{proof}
\begin{definition}[Beig-Chru\'sciel \cite{Beig-Chrusciel:1996}]
Let $(U, g)$ be an Riemannian manifold equipped with a lapse-shift pair $(f, X)$ such that $f$ is nonvanishing on $U$.  The \emph{Killing development of $(U,g, f, X)$} is a triple $(\mathbf{N}, \mathbf{g}, \mathbf{Y})$, where $(\mathbf{N}, \mathbf{g})$ is a Lorentzian spacetime equipped with a vector field $\mathbf{Y}$, defined as follows: 
\begin{enumerate}
\item 
$\mathbf{N}:=\mathbb{R}\times U$.
\item 	$\mathbf{g} := -4f^2 du^2 + g_{ij} (dx^i +X^i du) (dx^j + X^j du)$, where $u$ is the coordinate on the $\mathbb{R}$ factor, $(x^1, \dots, x^n)$ is a local coordinate chart  of $U$, and $f, X, g_{ij}$ are all extended to $\mathbf{N}$ to be independent of the coordinate $u$.
\item $\mathbf{Y}:=\frac{\partial}{\partial u}$.
\end{enumerate}
\end{definition}
 The triple $(\mathbf{N}, \mathbf{g}, \mathbf{Y})$ is called the Killing development of $(U, g, f, X)$  because 
 $\mathbf{Y}$ is a global Killing field (Lemma~\ref{lemma:Killing-independent-u}), which is related to the data $(f,X)$ via the equation
 \[ \mathbf{Y} = 2f \mathbf{n} +X,\]
 which holds along the constant $u$ slices, where $\mathbf{n}$ is a unit normal to those slices. (We declare that this  $\mathbf{n}$ defines the time-orientation on $\mathbf{N}$). Lemma~\ref{lemma:Einstein_tangential}  implies that  $(U, g, \pi)$ sits inside the Killing development of $(U, g, f, X)$ as the $\{u=c\}$ slice for every constant $c$,  if and only if the ``$k$'' corresponding to $\pi$ satisfies 
 \begin{equation}\label{eqn:sits-in-KD}
  k_{ij} = -\frac{1}{4f} (L_X g)_{ij}. 
  \end{equation}

The following proposition is a more precise version of the first half of  Theorem~\ref{theorem:DEC}.
 \begin{proposition}\label{proposition:fluid}
Let $(U, g, \pi)$ be an initial data set equipped with a lapse-shift pair $(f, X)$ such that $f$ is nonvanishing on $U$, and assume that $(f,X)$ solves the system~\eqref{equation:pair}. Then the following holds:
\begin{enumerate}
\item \label{item:constancy-proposition}
 The dominant energy scalar $\sigma(g, \pi)$ is constant in $U$.
\item  \label{item:spacetime-proposition} Let $(\mathbf{N}, \mathbf{g}, \mathbf{Y})$ be the Killing development of $(U,g, f, X)$. Then $(U, g, \pi)$ sits inside  $(\mathbf{N}, \mathbf{g})$, which is a null perfect fluid spacetime with velocity $\mathbf{v} = \frac{\sqrt{|J|_g}}{ 2f}\mathbf{Y}$ and pressure $p = -\tfrac{1}{2}\sigma(g, \pi)$. In particular, the Einstein tensor of $\mathbf{g}$ satisfies~\eqref{equation:perfect-fluid} along every constant $u$ slice of $\mathbf{N}$.
\item If $(U, g, \pi)$ satisfies the dominant energy condition, then $(\mathbf{N}, \mathbf{g})$ satisfies the spacetime  dominant energy condition. \label{item:DEC-proposition}
\end{enumerate}
\end{proposition}
Note that conclusion \eqref{item:constancy-proposition} may fail if we remove the nonvanishing assumption on $f$. As mentioned in Section~\ref{section:introduction}, if $f$ vanishes on all of $U$, then  the system \eqref{equation:pair} is equivalent to that $L_X g=0$ and $L_X\pi=0$. In particular, these examples need not  have constant $\sigma(g,\pi)$ (as can be seen by considering explicit time-symmetric examples).


\begin{proof}
We focus on proving Item \eqref{item:spacetime-proposition}, and then Items~\eqref{item:constancy-proposition} and~\eqref{item:DEC-proposition} follow immediately from Lemma~\ref{lemma:null-perfect-fluid}.

First observe that equation~\eqref{equation:momentum} in the system~\eqref{equation:pair}, together with~\eqref{equation:k}, tells us that the ``$k$'' corresponding to our given $\pi$ satisfies~\eqref{eqn:sits-in-KD}, and hence $(U, g, \pi)$ sits inside $(\mathbf{N}, \mathbf{g})$.

Next, we wish to show that the Einstein tensor of $\mathbf{g}$ takes the form of a null perfect fluid. Specifically, we will show that it satisfies~\eqref{equation:perfect-fluid} along~$U$ (and similarly, along any constant $u$ slice). Choose an orthonormal basis $e_0, e_1,\ldots, e_n$ with $e_0=\mathbf{n}$. The equations $G_{00}=\mu$ and $G_{0i}=J_i$ in~\eqref{equation:perfect-fluid} are essentially the definitions of $\mu$ and $J$ (or more precisely, see~\eqref{equation:Einstein-0i}). We have the following general formula for $G_{ij}$ in a Killing development~\eqref{equation:Einstein_pi}: 
\begin{align}\tag{\ref{equation:Einstein_pi}}
\begin{split}
G_{ij}&=  \left[R_{ij} -\tfrac{1}{2}R_g g_{ij}\right]+
 \left[- \tfrac{3}{n-1}(\tr_g \pi)\pi_{ij} + 2\pi_{i\ell}\pi^\ell_j
  \right]\\
  &\quad  + \left[   \tfrac{1}{2(n-1)}(\tr_g \pi)^2 - \tfrac{1}{2}|\pi|_g^2 \right] g_{ij} \\
&\quad+f^{-1}\left[-\tfrac{1}{2}(L_X \pi)_{ij}  -  f_{;ij} +(\Delta_g f)g_{ij}\right]. 
\end{split}
\end{align}
We would like to compare this to~\eqref{equation:Hamiltonian} in the system~\eqref{equation:pair}, in which $\phi$ and $Z$ are both zero. Dividing \eqref{equation:Hamiltonian} by~$f$ gives
\begin{align}\label{equation:Hamiltonian-zero}
\begin{split}
0 &=  -R_{ij}  +  \left[ \tfrac{3}{n-1}(\tr_g \pi)\pi_{ij} - 2\pi_{i\ell}\pi^\ell_j
  \right]
 + \left[   -\tfrac{1}{n-1}(\tr_g \pi)^2 + |\pi|_g^2 \right] g_{ij} \\
&\quad +f^{-1}\left[ \tfrac{1}{2}(L_X \pi)_{ij}  +  f_{;ij} - (\Delta_g f)g_{ij}
 -\tfrac{1}{2}\langle X, J\rangle_g g_{ij} -\tfrac{1}{2}(X\odot J)_{ij}\right].
 \end{split}
\end{align}
Adding together the two previous equations, we get
\begin{align}
\begin{split} \label{equation:tangential-lapse-shift1}
G_{ij}  &= -\tfrac{1}{2}\left[ R_g +\tfrac{1}{n-1}(\tr_g\pi)^2 -|\pi|_g^2\right]g_{ij} \\
&\quad 
 -\tfrac{1}{2f}\langle X, J\rangle_g g_{ij} -\tfrac{1}{2f}(X\odot J)_{ij} 
 \end{split}\\
&= \left(-\mu  -\tfrac{1}{2f}\langle X, J\rangle_g\right) g_{ij} -\tfrac{1}{2f}(X\odot J)_{ij}.\label{equation:tangential-lapse-shift2}
\end{align}
At points where $J=0$, the expression for $G_{ij}$ in \eqref{equation:perfect-fluid} follows immediately, and when $J\neq0$, we can see it by substituting 
in the $J$-null-vector equation, $\frac{-X}{2f}=\frac{J}{|J|_g}$. From our work in Lemma~\ref{lemma:null-perfect-fluid}, this implies that $(\mathbf{N}, \mathbf{g})$ is a null perfect fluid spacetime with velocity $\mathbf{v}$ and pressure $p$ satisfying~\eqref{equation:v+p}. Specifically, $p=|J|_g-\mu = -\tfrac{1}{2}\sigma(g,\pi)$, and the $J$-null vector equation implies that $\mathbf{v} = \frac{\sqrt{|J|_g}}{2f} \mathbf{Y}$, completing the proof of Item~ \eqref{item:spacetime-proposition}.

\end{proof}

Next, we prove the second half of Theorem~\ref{theorem:DEC}, which is a sort of converse to Item~\eqref{item:spacetime-proposition} of Proposition~\ref{proposition:fluid}.
\begin{proposition}
Let  $(\mathbf{N}, \mathbf{g})$ be a null perfect fluid spacetime with  velocity $\mathbf{v}$ and pressure $p$, and $\mathbf{g}\in C^3_{\mathrm{loc}}(\mathbf{N})$.  Assume there is a global $C^2_{\mathrm{loc}}$ Killing vector field $\mathbf{Y}$ such that $\mathbf{v} = \eta \mathbf{Y}$ for some scalar function $\eta$. If $U$ is a spacelike hypersurface of $\mathbf{N}$ with induced initial data $(g, \pi)$ and future unit normal $\mathbf{n}$, and if we decompose $\mathbf{Y}=2f \mathbf{n} + X$ along~$U$, then the lapse-shift pair $(f, X)$ satisfies the system~\eqref{equation:pair}. 

Moreover, if $\mathbf{Y}$ is transverse to $U$, then $(\mathbf{N}, \mathbf{g}, \mathbf{Y})$ must agree with the Killing development of $(U, g,  f, X)$ in a neighborhood of $U$.
\end{proposition}

\begin{proof}
We decompose $U= V_1 \cup V_2$ so that $f\neq 0$ on $V_1$ and $f\equiv 0$ on~$V_2$. In particular, $\mathbf{Y}$ is transverse to $V_1$, and $\mathbf{Y}$ is tangent to $\Int V_2$.

On $V_1$, together with \eqref{equation:v+p}, the assumption that $\mathbf{v}=\eta \mathbf{Y}$  implies that $(f, X)$ satisfies the $J$-null-vector equation $2fJ +|J|_g X=0$ and  $\eta =\frac{\sqrt{|J|_g}}{2f}$. Once the $J$-null-vector equation is established,  the same computations in the proof of Proposition~\ref{proposition:fluid}, when viewed ``backwards,'' will imply that $D\overline{\Phi}^*_{(g, \pi)}(f, X)=0$. More specifically, the $J$-null-vector equation and \eqref{equation:perfect-fluid} imply \eqref{equation:tangential-lapse-shift2} and \eqref{equation:tangential-lapse-shift1}. Using  \eqref{equation:Einstein_pi}, we see that \eqref{equation:Hamiltonian-zero} holds, which is the same as   \eqref{equation:Hamiltonian} (in the case where $\phi$ and $Z$ are both zero), and \eqref{equation:Killing} is the same as \eqref{equation:momentum}. Taken together, this gives us~\eqref{equation:pair}, as desired.

The second paragraph of the Proposition then follows easily from the discussion in Appendix~\ref{section:spacetime}, together with Proposition~\ref{proposition:fluid}.

On $\Int V_2$, since $\mathbf{Y}= X$ is spacelike, $\mathbf{v}$ must vanish, and so does~$J$. Thus, the $J$-null-vector equation trivially holds  on $\Int V_2$.  Since $\mathbf{Y}$ is  Killing,  it follows that $L_X g=0$ and $L_X \pi=0$ on $\Int V_2$. Observe that these two equations are equivalent to $D\overline{\Phi}_{(g,\pi)}^*(0,X)=0$. Thus, we conclude that both equations~\eqref{equation:pair} hold on $V_1$ and $\Int V_2$, and then by continuity, they hold everywhere in $U$.
\end{proof}

\subsection{Removing the nonvanishing assumption on $f$}
The nonvanishing assumption of $f$ is essential in the proof of Proposition~\ref{proposition:fluid}. Thankfully, there are some situations in which we can show that $f$ is nonvanishing, and this allows us to expand the applicability of  Proposition~\ref{proposition:fluid}.  (See Theorem~\ref{theorem:f_not_zero} below.)

\begin{lemma}
Let $(U, g, \pi)$ be an initial data set, and let $(f, X)$ be a lapse-shift pair  satisfying the following equations on $U$: 
\begin{align}
	\tfrac{1}{2} (X_{i;j} + X_{j;i} ) &= \left( \tfrac{2}{n-1} (\tr_g \pi)g_{ij} - 2\pi_{ij} \right)f \tag{\ref{equation:momentum}}\\
	2fJ+|J|_g X&=0\notag.
\end{align}
Suppose that $J$ is nonvanishing on $U$ and we set $\hat{J} := \frac{J}{|J|_g}$. Then, $f$ satisfies the following equations everywhere in $U$:
\begin{align}
	\langle \nabla f,  \hat{J}\rangle_g &=- \left( \tfrac{1}{n-1} (\tr_g \pi )- \pi_{ij} \hat{J}^i\hat{J}^j\right) f \label{equation:gradf}\\
	(\nabla f)^j  &=-  \left(\hat{J}^j_{;i} \hat{J}^i + \tfrac{1}{n-1} (\tr_g \pi) \hat{J}^j -2 \pi^{ij} \hat{J}_i + \pi_{k\ell } \hat{J}^k \hat{J}^\ell \hat{J}^j\right) f.\label{equation:gradf-2}
\end{align}
\end{lemma}
\begin{proof}
Define $T_{ij} := \tfrac{1}{2} (X_{i;j} - X_{j;i})$ and $W:=f\hat{J}+\tfrac{1}{2}X$ as in the proof of Lemma~\ref{lemma:matrixQ}, and choose $Z=J$.
 Our assumption~\eqref{equation:momentum} is the same as~\eqref{equation:DX}, and therefore equation~\eqref{equation:DW} is valid. On the other hand, the $J$-null-vector equation says that $W=0$, and hence \eqref{equation:DW} says
\begin{align}\label{equation:DW2}
	0&=\hat{J}_i f_j + \tfrac{1}{2} T_{ij} + \left(\hat{J}_{i;j} +\tfrac{1}{n-1} (\tr_g \pi) g_{ij} - \pi_{ij} \right)f.
\end{align} 
Contracting \eqref{equation:DW2} with $\hat{J}^i$ and then with $\hat{J}^j$ gives 
\begin{align}
	0 &= f_j + \tfrac{1}{2} T_{ij} \hat{J}^i + \left( \tfrac{1}{n-1} (\tr_g \pi ) \hat{J}_j - \pi_{ij} \hat{J}^i\right) f\label{equation:J_i}\\
	0&= f_j \hat{J}^j + \left( \tfrac{1}{n-1} (\tr_g \pi )- \pi_{ij} \hat{J}^i\hat{J}^j\right) f,\notag
\end{align}
where we use  $|\hat{J}|_g=1$ and the anti-symmetry of $T$. Thus equation~\eqref{equation:gradf} holds. Contracting \eqref{equation:DW2} with $\hat{J}^j$ and then swapping the $i$ and $j$ indices gives 
\begin{align}
	0&=\hat{J}_j  (f_i \hat{J}^i )+ \tfrac{1}{2} T_{ji} \hat{J}^i+ \left( \hat{J}_{j;i} \hat{J}^i + \tfrac{1}{n-1} (\tr_g \pi) \hat{J}_j - \pi_{ji} \hat{J}^i \right) f\notag\\
	&=\tfrac{1}{2} T_{ji} \hat{J}^i + (\hat{J}_{j;i} \hat{J}^i - \pi_{ji} \hat{J}^i  + \pi_{k\ell } \hat{J}^k \hat{J}^\ell \hat{J}_j)f,\label{equation:J_j}
\end{align}
where we use \eqref{equation:gradf} to substitute $f_j \hat{J}^j $ in the last equation. Adding \eqref{equation:J_i} and \eqref{equation:J_j} to cancel out the terms of $T$ by anti-symmetry gives \eqref{equation:gradf-2}:
\[
	0 = f_j +  \left(\hat{J}_{j;i} \hat{J}^i + \tfrac{1}{n-1} (\tr_g \pi) \hat{J}_j -2 \pi_{ij} \hat{J}^i + \pi_{k\ell } \hat{J}^k \hat{J}^\ell \hat{J}_j\right) f.
\]

\end{proof}

\begin{corollary}\label{corollary:one-dimensional}
Let $(U, g, \pi)$ be an initial data set such that $J$ is not identically zero on~$U$.
Define $\mathcal{N}$ to be the space of all lapse-shift pairs $(f, X)\in C^2_{\mathrm{loc}}\times C^1_{\mathrm{loc}}$ solving the system~\eqref{equation:pair}. Then the vector space $\mathcal{N}$ is at most one-dimensional. Furthermore, if $(f, X)\in \mathcal{N}$ is nontrivial, then $f$ must be nonzero everywhere that $J$ is nonzero.
\end{corollary}
\begin{proof}
Because any solution $(f, X)$ to $D\overline{\Phi}_{(g, \pi)}^*(f, X)=0$ satisfies the Hessian-type equations \eqref{equation:Hessian} and \eqref{equation:second-derivative}, $(f, X)$ is determined by its 1-jet $(f, \nabla f, X, \nabla X)$ at an arbitrary point $p\in M$. (Cf. \cite[Proposition 2.1]{Corvino-Huang:2020}.) Choose any $p\in U$ so that $J\ne 0$ at $p$. By equation~\eqref{equation:gradf-2}, $\nabla f(p)$ is determined by $f(p)$. Together with the $J$-null vector equation $2f J + |J|_g X=0$, the 1-jet of $X$ at $p$ is also determined by $f(p)$. Thus, $\mathcal{N}$ is at most one-dimensional, and if $f(p)=0$, then $(f, X)$ is identically zero in~$U$.  

\end{proof}

\begin{proposition}\label{proposition:f_not_zero}
Let $(U, g, \pi)$ be an initial data set, and suppose $(f, X)$ is a nontrivial lapse-shift pair on $U$ solving the system \eqref{equation:pair}. 
If $V$ is the set of points where $J\ne0$, then $f$ is nonvanishing on~$\overline{V}$.
\end{proposition}

\begin{proof} 
In this proof, we compute all lengths, products, traces, and covariant derivatives using $g$.
Let $p\in \overline{V}$, and suppose that $f(p)=0$. We will show that the entire 1-jet of $(f, X)$ vanishes at $p$, and thus $(f, X)$ must vanish everywhere in $U$, giving the desired contradiction. By the $J$-null-vector equation, we know that $4f^2=|X|^2$ in $V$, so by continuity, $X(p)=0$. Taking the Laplacian gives
\[ 0 = \Delta(4f^2-|X|^2) = 8f\Delta f - 2\langle X, \Delta X\rangle+8 |\nabla f|^2 -2|\nabla X|^2\quad \mbox{ in } V.\]
By continuity, this equation still holds at $p$, and thus $|\nabla X|=4|\nabla f|$ at $p$.  It now suffices to prove that  $\nabla f(p)=0$.

To do this we will show that $|\nabla f|^2 \le C|f|$  in $V\cap B$ for some constant $C$ and some ball $B$ around~$p$.  Note $g$, $\pi$, and $\hat{J}$ must be bounded on $V\cap B$. The only potentially unbounded quantity appearing in~\eqref{equation:gradf-2} is $\hat{J}^j_i$.  Multiplying~\eqref{equation:gradf-2} with $f_j$, we obtain, in $V\cap B$,
\[ |\nabla f|^2 = (-\hat{J}^{j}_{;i}f_j \hat{J}^i  + \text{bounded terms})  f.\]
We claim the first term in the coefficient of $f$ above is also bounded:
\begin{align*}
\hat{J}^{j}_{;i}f_j  \hat{J}^i   &=\left[( \hat{J}^{j} f_j)_{i} - \hat{J}^{j}f_{;ji}\right] \hat{J}^i     \\
& =\left[ -\tfrac{1}{n-1}[(\tr \pi) f]_i + (\pi_{jk} f)_{;i} \hat{J}^j\hat{J}^k + 2\pi_{jk} \hat{J}^j_{;i} \hat{J}^k f - \hat{J}^{j}f_{;ji} \right] \hat{J}^i  \\
&= 2\pi_{jk} \left(\hat{J}^j_{;i}\hat{J}^i f\right) \hat{J}^k   + \text{bounded terms},
\end{align*}
where we use \eqref{equation:gradf} and symmetry of $\pi$ in the second equality. Looking at equation~\eqref{equation:gradf-2}, we see that $\hat{J}^{j}_{;i}\hat{J}^i f$ must be bounded since every other term is bounded. This completes the proof.

\end{proof}

In the following theorem, we remove  the nonvanishing assumption of $f$ in Proposition~\ref{proposition:fluid}  in the important special case that $\sigma(g, \pi)$ is identically zero. 

\begin{theorem}\label{theorem:f_not_zero}
Let $(U, g, \pi)$ be an initial data set. Assume there exists a nontrivial lapse-shift pair $(f,X)$ on $U$ solving the system~\eqref{equation:pair}, and assume that $\sigma(g, \pi)\equiv 0$ on $U$. Then  $(U, g, \pi)$ sits inside a null dust spacetime $(\mathbf{N},\mathbf{g})$ satisfying the spacetime dominant energy condition and admitting a global Killing vector field $\mathbf{Y}$.
 Moreover, $\mathbf{g}$ is vacuum on the domain of dependence of the set where $(g, \pi)$ is vacuum, and $\mathbf{Y}$  is null wherever $\mathbf{g}$ is not vacuum. 
\end{theorem}
\begin{proof}
Let  $V_1\subset U$ be the open set where $f$ is nonzero, and let  $V_2\subset U$ be  the interior of the set where $J=0$. By Proposition~\ref{proposition:f_not_zero}, we can write $U = V_1\cup V_2$. Note that the assumption $\sigma(g, \pi) \equiv 0$ implies that $(V_2, g, \pi)$ is vacuum, that is,  $\mu=|J|_g =0$. 

Let $(\mathbb{R}\times V_1, \mathbf{g}_1, \mathbf{Y}_1 )$ be the Killing development of $(V_1, g,  f, X)$.
 By Proposition~\ref{proposition:fluid}, $(V_1, g, \pi)$ sits inside $(\mathbb{R}\times V_1, \mathbf{g}_1)$, which is a null perfect fluid spacetime with velocity $\mathbf{v}=\frac{\sqrt{|J|_g}}{2f}\mathbf{Y}_1$ and $p=-\tfrac{1}{2}\sigma(g, \pi)$. Moreover, since $\sigma(g, \pi)\equiv0$ on $V_1$, it also follows that $\mathbf{g}_1$ satisfies the DEC, $p\equiv0$, and $\mathbf{Y}_1$ is null wherever $J\neq 0$. 
 From the expression of the Einstein tensor \eqref{equation:perfect-fluid},  we know that the domain of dependence of the subset $V_1\cap V_2$ in $(\mathbb{R}\times V_1, \mathbf{g}_1)$ is vacuum. Recall that  $\mathbf{Y}_1= 2f\mathbf{n}+X$ along $V_1$, where $\mathbf{n}$ is the future unit normal to $V_1$.

On the other hand, Theorem~\ref{theorem:Moncrief} says the vacuum initial data set $(V_2, g, \pi)$ sits inside a vacuum spacetime $(\mathbf{N}_2, \mathbf{g}_2)$ admitting a unique global Killing vector field $\mathbf{Y}_2$ which is equal to $2f\mathbf{n}+X$ along the hypersurface $V_2$, where $\mathbf{n}$ is the future unit normal to $V_2$.

By uniqueness of the vacuum development,   the domain of dependence on $V_1\cap V_2$  for the spacetime metric $\mathbf{g}_1$ must be isometric to the corresponding domain of dependence in $\mathbf{N}_2$ for $\mathbf{g}_2$ where both are defined. The two developments have compatible overlap, giving rise to a single spacetime $(\mathbf{N}, \mathbf{g})$ in which $(U, g, \pi)$ sits. By construction, this $(\mathbf{N}, \mathbf{g})$ is a null dust spacetime satisfying the spacetime DEC and is vacuum on the domain of dependence of  $V_2$. 

To patch the two Killing vectors $\mathbf{Y}_1$ and $\mathbf{Y}_2$ on $V_1\cap V_2$, we use the fact that a Killing vector in a vacuum spacetime is uniquely determined by the lapse-shift pair $(f, X)$ solving $D\Phi^*(f, X)=0$ on an initial data set (see \cite[p.496]{Moncrief:1975} and \cite[Lemma 2.2]{Fischer-Marsden-Moncrief:1980}). On $V_1 \cap V_2$,  $\mathbf{Y}_1$ and $\mathbf{Y}_2$ both equal $2f\mathbf{n}+X$ with $(f,X)$ solving $D\Phi^*(f, X) = D\overline{\Phi}^*(f, X)=0$ as $J=0$ there.  We conclude that $\mathbf{Y}_1$ and $\mathbf{Y}_2$ must coincide on the domain of dependence of $V_1\cap V_2$, and therefore they gives rise to a single global Killing vector field in $(\mathbf{N}, \mathbf{g})$. This completes the proof.

\end{proof}

\subsection{Vanishing of $J$ in an asymptotically flat end}
In the previous subsections, we have derived several strong consequences of an initial data set that has a nontrivial lapse-shift pair $(f, X)$ solving the system~\eqref{equation:pair}. If the lapse-shift pair solves the system over an asymptotically flat end, we are able to use the asymptotics of $(f, X)$  to say more about the initial data set. 

We first note the following ``harmonic'' asymptotics of $(f, X)$ in an asymptotically flat end. The result is well-known for the usual adjoint linearized  constraint, see \cite[Proposition 2.1]{Beig-Chrusciel:1996}. The basic idea is to use the Hessian type equations \eqref{equation:Hessian} and \eqref{equation:second-derivative} of $(f, X)$ to show that $(f, X)$ has at most linear growth along a geodesic ray. See \cite[Appendix C]{Beig-Chrusciel:1996} (Cf. \cite[Proposition B.4]{Huang-Martin-Miao:2018}). Then the desired harmonic expansions of $(f, X)$ follows from the fact that $\Delta_g f$ and each $\Delta_g X_i$ are in $C^{2,\alpha}_1$ together with a  bootstrapping argument. The relation \eqref{equation:ADM-asymptotics} below follows from \cite[Proposition 3.1]{Beig-Chrusciel:1996} (see, also \cite[Corollary A.9]{Huang-Lee:2020}).

\begin{lemma}\label{lemma:asymptotics}
Let $(M, g, \pi)$ be an $n$-dimensional asymptotically flat initial data set of type $(q, \alpha)$  with the ADM energy-momentum $(E, P)$. Let $s = \mathrm{max}\{0, 1-q\}$. Suppose that $(f, X)$ is a lapse-shift pair on $\Int M$ such that $D\overline\Phi_{(g, \pi)}^*(f, X) =0$. Then the following holds:
\begin{enumerate}
\item For $i, j=1,\ldots n$, there exist constants $c_i, d_{ij}\in \mathbb{R}$ with $d_{ji}=-d_{ij}$, such that
\begin{align}\label{equation:asymptotics-linear}
\begin{split}
f&= \sum_{i=1}^n c_i x^i + O_{2,\alpha}(|x|^{s})\\
X_i &= \sum_{j=1}^n d_{ij}x^j + O_{2,\alpha}(|x|^{s }).
\end{split}
\end{align}
\item For $i, j=1,\ldots n$, if the $c_i$ and $d_{ij}$ above are all zero, then there exist constants $a, b_{i}\in \mathbb{R}$  such that
\begin{align}\label{equation:asymptotics-constant}
\begin{split}
f&= a + O_{2,\alpha}(|x|^{ -q})\\
X_i &= b_i + O_{2,\alpha}(|x|^{ -q}).
\end{split}
\end{align}
We also have, for each $i$, 
\begin{align} \label{equation:ADM-asymptotics}
	b_i E = -2aP_i.
\end{align} 
\item For $i, j=1,\ldots n$, if the $c_i$, $d_{ij}$, $a$, $b_i$ above are all zero, then $(f, X)$ vanishes identically.
\end{enumerate}
\end{lemma}  


\begin{lemma}\label{lemma:unbounded}
Let $(M, g, \pi)$ be an asymptotically flat initial data set of type $(q, \alpha)$ with the ADM energy-momentum $(E, P)$. Assume there exists a nontrivial lapse-shift pair $(f, X)$ on $\Int M$ solving the system~\eqref{equation:pair}. If $|J|_g>0$ in an unbounded subset $V\subset M$,  then $|E|=|P|$.
\end{lemma}
\begin{proof}
By Lemma~\ref{lemma:asymptotics}, $(f, X)$ satisfies the asymptotics \eqref{equation:asymptotics-linear}. We claim  $c_i=0, d_{ij}=0$ for all~$i,j$. To get a contradiction, suppose  $c_i, d_{ij}$ are not all zero. Since $V$ is unbounded, the equation $2fJ+|J|_g X=0$ implies $|X|_g=2|f|$ in $V$, and thus $c_i\neq 0$ for some $i$. By rotating the coordinates and rescaling $(f, X)$, we may assume
 $f =x^1 + O_2(|x|^{s})$. Recall $s=\max \{ 0, 1-q\}$, and note $s-1\in [-q, 0)$. Then $\tfrac{\partial f}{\partial x^1} = 1+O_1(|x|^{s-1})$, and equation \eqref{equation:gradf} and asymptotical flatness of $(g, \pi)$ imply that $\tfrac{\partial f}{\partial x^1} \hat{J}_1 = O_1(|x|^{-q})$. Together with the  asymptotics of $\tfrac{\partial f}{\partial x^1}$, we see that $\hat{J}_1 = O_1(|x|^{s-1})$ in $V$, and then equation \eqref{equation:gradf-2} implies that 
\[
	\tfrac{\partial f}{\partial x^1} = -\sum_k \left(\hat{J}_k \tfrac{\partial \hat{J}_{1}}{\partial x^k}  \right) f + O(|x|^{-q}) = O(|x|^{s-1})
\] 
in $V$, which contradicts the fact that $\tfrac{\partial f}{\partial x^1}$ is asymptotic to~$1$. 

By Lemma~\ref{lemma:asymptotics}, $(f, X)$ satisfies the asymptotics \eqref{equation:asymptotics-constant} with  $2|a| = |b|>0$, and hence $|E|=|P|$ by \eqref{equation:ADM-asymptotics}.

\end{proof}

Next, we combine Theorem~\ref{theorem:improvable} with Lemma~\ref{lemma:unbounded} to prove the following theorem. 

\begin{theorem}\label{theorem:AF}
Let $(M, g, \pi)$ be an asymptotically flat initial data set, possibly with boundary, such that 
$(g,\pi) \in C^5_{\mathrm{loc}}(\Int M)\times C^4_{\mathrm{loc}}(\Int M)$. Suppose there exists a sequence of compact sets $\Omega_j$ with smooth boundary such that the sequence exhausts $\Int M$, and for all $j$, the dominant energy scalar is NOT improvable in $\Int\Omega_j$. 

Then there exists a nontrivial lapse-shift pair $(f, X)$ on $\Int M$ solving the system \eqref{equation:pair}. If we further assume  that the ADM energy-momentum $(E, P)$ satisfies $|E|\neq |P|$, then the current density $J$ vanishes outside some compact subset.
\end{theorem}


\begin{proof}
Suppose there exists a sequence of compact sets $\Omega_j$ exhausting $\Int M$ such that for all $j$, the dominant energy scalar is \emph{not} improvable in $\Int\Omega_j$.  By Theorem~\ref{theorem:improvable},  there exists a nontrivial lapse-shift pair $(f^{(j)}, X^{(j)})$ on $\Int\Omega_j$ solving the system \eqref{equation:pair}. Since the $1$-jet of $(f^{(j)}, X^{(j)})$ at any point is nonzero, we rescale it so that $|f^{(j)}| +|\nabla f^{(j)}| + |X^{(j)}| +|\nabla X^{(j)}|=1$ at a fixed point $p$ in all $\Omega_j$. Standard elliptic theory shows that, as $j\to\infty$, a subsequence of $(f^{(j)}, X^{(j)})$ converges to a nontrivial lapse-shift pair $(f, X)$ defined on $\Int M$ solving the system \eqref{equation:pair}. This proves the first part of Theorem~\ref{theorem:AF}.

If we assume $|E| \neq |P|$ and that there is a nontrivial lapse-shift pair $(f, X)$  on $\Int M$ solving the system \eqref{equation:pair}, by Lemma~\ref{lemma:unbounded}, $J$ must vanish outside a compact subset. 

\end{proof}

\section{Bartnik's stationary conjecture}\label{section:Bartnik}

In Section~\ref{section:decrease-energy}, we construct a family of deformations by conformal change  to decrease the ADM energy without changing the ADM linear momentum, while the deformation slightly breaks the DEC in a fixed compact set. Then we use the improvability and almost-improvability results to reinstate the DEC (Theorem~\ref{theorem:Bartnik2}).  In Section~\ref{section:admissibility}, we define an admissible extension for the Bartnik mass and combine Theorem~\ref{theorem:Bartnik2} with Theorem~\ref{theorem:AF} to prove our main result on the Bartnik stationary conjecture, Theorem~\ref{theorem:Bartnik}. In Section~\ref{section:energy}, we include some results of independent interest about the Bartnik \emph{energy}. 

\subsection{Deformations that decrease the ADM energy}\label{section:decrease-energy}

The following notations will frequently appear in this section: $(M, g, \pi)$ denotes an asymptotically flat initial data set of type $(q,\alpha)$, $(E, P)$ is the ADM energy-momentum,   $(\mu, J)$ is the energy and current density, and $\Omega_r$ denotes the compact part of $M$ enclosed by the sphere~$|x|=r$.

\begin{proposition}\label{proposition:conformal}
Let $(M, g, \pi)$ be an asymptotically flat initial data set satisfying the dominant energy condition. For each $r_0\gg 1$, there is a one-parameter family of asymptotically flat initial data sets $(g_t, \pi_t)$ for $t\in (-\epsilon, \epsilon)$ with $(g_0, \pi_0) = (g, \pi)$ such that 
\begin{enumerate}
 \item $(g_t, \pi_t)= (g,\pi)$ in $\Omega_{r_0}$ for all $t$. \label{item:B}
\item \label{item:E} $E_t  = E- t$ and $P_t = P$, where $(E_t, P_t)$ is the ADM energy-momentum of $(g_t, \pi_t)$. 
\item \label{item:convergence}
Let $(\mu_t, J_t)$ denote  the energy and current densities of $(g_t, \pi_t)$.  As $t\to 0$, 
\begin{align*}
\| (g_t, \pi_t) - (g,\pi)\|_{C^{2,\alpha}_{-q}(M)\times C^{1,\alpha}_{-1-q}(M)} &\to 0\\
\|(\mu_t, J_t) - (\mu, J)\|_{L^1(M)}&\to 0,
\end{align*}
and on each compact subset $(\mu_t, J_t)\to (\mu, J)$ in $C^{0,\alpha}$.

\item \label{item:annulus} For sufficiently large $r_1>r_0$, $\Sc(g_t, \pi_t) > \Sc(g, \pi)$  in $M\smallsetminus \Omega_{r_1}$.
\end{enumerate}
\end{proposition}
\begin{proof}
We will  select a conformal factor $u_t\to 1$ such that we have exact knowledge of $\Delta_g u_t$ outside $\Omega_{2r_0}$ and exact knowledge of how $u_t$ affects the energy, and then we cut it off arbitrarily to make it~$1$ on $\Omega_{r_0}$. To do this, choose $\delta>0$ smaller than both $q$ and $1$. 
We claim that there exists a function $v\in C^{2,\alpha}(M)$ such that on $M\smallsetminus \Omega_{r_0}$,
\begin{align*}
	\Delta_g v &= -|x|^{-n-\delta} \\
		v &=-\tfrac{1}{2} |x|^{2-n} + O_{2,\alpha}(|x|^{2-n-\delta}).
\end{align*}	
To do this, extend $-|x|^{-n-\delta}$ on $M\smallsetminus  \Omega_{r_0}$ to some function $\rho$ on all of $M$
 such that 
  $\int_M \rho \, d\mu_g = \tfrac{n-2}{2}\omega_{n-1}$. Since $\Delta_g: C^{2,\alpha}_{-q}(M)\to C^{0,\alpha}_{-2-q}(M)$ is an isomorphism, we can solve the Poisson equation
 $\Delta_g v=\rho$ for some $v\in C^{2,\alpha}_{-q}(M)$. (If $M$ has a boundary, we can either fill it in arbitrarily or solve for $v$ with Neumann conditions.) By harmonic expansion, we know that $v=a|x|^{2-n} + O_{2,\alpha}(|x|^{2-n-\delta})$, for some constant $a$. (See~\cite[Remark A.39]{Lee:book}, for example.)  Integrating the Poisson equation, we see that
\[
	a = \tfrac{-1}{(n-2)\omega_{n-1}}  \int_M \Delta_g v\, d\mu_g = \tfrac{-1}{(n-2)\omega_{n-1}} \int_M \rho\, d\mu_g =-\tfrac{1}{2}.
\]

Let $\chi$ be a smooth nonnegative cut-off function such that $\chi \equiv 0$ on $\Omega_{r_0}$ and $\chi\equiv1$ outside $\Omega_{2r_0}$. Define $u_t = 1+t\chi v$. For small $\epsilon>0$, we have $u_t>0$ for all $t\in(-\epsilon, \epsilon)$. We define the one-parameter  family of  initial data sets:
\begin{align*}
	(g_t)_{ij} = u_t^{\frac{4}{n-2}} g_{ij} \quad \mbox{and} \quad (\pi_t)^{ij} = u_t^{-\frac{6}{n-2}} \pi^{ij}.
\end{align*}
Note that because $\chi\equiv 0$ in $\Omega_{r_0}$, $(g_t, \pi_t) = (g, \pi)$ in $\Omega_{r_0}$ which implies Item~\eqref{item:B}. 

By a standard computation, Item~\eqref{item:E} follows:
\[
	E_t = E +2ta =E-t \quad \mbox{and}\quad P_t = P.
\]

To verify  Item \eqref{item:convergence}, we see that   $\| (g_t, \pi_t) - (g, \pi)\|_{C^{2,\alpha}_{-q}(M)\times C^{1,\alpha}_{-1-q}(M)} \to 0$ as $t\to 0$ because $\| u_t - 1\|_{C^{2,\alpha}_{-q}(M)} \to 0$. This directly implies that on each compact subset $(\mu_t, J_t)\to (\mu, J)$ in $C^{0,\alpha}$. By direct computation (Cf. \cite[Equation (37)]{Eichmair-Huang-Lee-Schoen:2016}, in which  $\pi$ is a (0,2)-tensor), $\mu_t$ and $J_t$ satisfy  
\begin{align}\label{equation:mu_t}
\begin{split}
	2 \mu_t &= u_t^{-\frac{4}{n-2}} \left(-\tfrac{4(n-1)}{n-2} u_t^{-1} \Delta_g u_t + R_g - |\pi|_g^2 + \tfrac{1}{n-1} (\tr_g \pi)^2\right)\\
	&= u_t^{-\frac{4}{n-2}} \left(2\mu + 2t\phi\right)\\
	J_t^i&= u_t^{-\frac{6}{n-2}} \left((\mathrm{div}_{g} \pi )^i+ \tfrac{2(n-1)}{n-2} u_t^{-1}(u_t)_{,j} \pi^{ij} - \tfrac{2}{n-2} g^{ij}u_t^{-1}(u_t)_{,j} \tr_g \pi\right)\\
	&=u_t^{-\frac{6}{n-2}} \left(J^i+ t\Upsilon^i\right)
\end{split}
\end{align}
where  $\phi, \Upsilon$ denote 
\begin{align*}
	\phi &=- \tfrac{2(n-1)}{n-2} u_t^{-1} \Delta_g (\chi v)\\
	\Upsilon^i &= \tfrac{2(n-1)}{n-2} u_t^{-1}(\chi v)_{,j} \pi^{ij} - \tfrac{2}{n-2} g^{ij}u_t^{-1}(\chi v)_{,j} \tr_g \pi.
\end{align*}
From this it is simple to check that  $(\mu_t, J_t)$ converges to $(\mu, J)$ in $L^1$ as $t\to 0$.

Last, we prove Item \eqref{item:annulus}. By \eqref{equation:mu_t},  the dominant energy scalar satisfies
\begin{align*}
	\Sc(g_t, \pi_t)  \ge u_t^{-\frac{4}{n-2}}\left( \Sc(g,\pi) + 2t(\phi- |\Upsilon|_g)\right)\quad \mbox{ in } M.
\end{align*}
In $M\smallsetminus \Omega_{2r_0}$ we have $\chi\equiv1$ and $ \Delta_g v=-|x|^{-n-\delta}$, so   
\begin{align*}
	\phi & = \tfrac{2(n-1)}{n-2} u_t^{-1} |x|^{-n-\delta}\\
	\Upsilon^i &=\tfrac{2(n-1)}{n-2} u_t^{-1} v_{,j} \pi^{ij} - \tfrac{2}{n-2} g^{ij}u_t^{-1}v_{,j} \tr_g \pi=O(|x|^{-n-q}).
\end{align*}
Since $\delta<q$, we see that for 
 $r_1$ sufficiently large, we have $\phi- |\Upsilon|_g>0$ in $M\smallsetminus \Omega_{r_1}$. Also, for $r_1$ sufficiently large, we have $v<0$ on  $M\smallsetminus \Omega_{r_1}$, and consequently, 
  we can guarantee that  $u_t=1+t\chi v <1$ there.
 Therefore
\[
\Sc(g_t, \pi_t) > u_t^{-\frac{4}{n-2}}\Sc(g, \pi) > \Sc(g, \pi) \quad \mbox{ in } M\smallsetminus \Omega_{r_1}.
\] 


\end{proof}

\begin{theorem}\label{theorem:Bartnik2}
Let $(M, g, \pi)$ be an asymptotically flat initial data set with non-empty smooth boundary $\partial M$. Suppose that the dominant energy condition holds and that $(g,\pi)$ is $C^5_{\mathrm{loc}}\times C^4_{\mathrm{loc}}$ on $\Int M$. Then  one of the two following statements must be true:
\begin{enumeratei}
\item  \label{item:non-improvable} $\sigma(g,\pi)\equiv 0$ on $\Int M$, and there exists a nontrivial solution $(f, X)$ on all of $\Int M$ solving the system \eqref{equation:pair}.  
\item \label{item:deformation}
There exists a bounded  neighborhood of $V_0$ of $\partial M$ such that for any $\epsilon>0$, there exists a perturbation $(\bar{g}, \bar{\pi})$ of $(g,\pi)$  such that 
\begin{itemize}
\item $(\bar{g}, \bar{\pi})$ satisfies the dominant energy condition.
\item $(\bar{g}, \bar{\pi})=(g,\pi)$ in $V_0$.
\item $\overline{E}< E$ and $\overline{P}=P$. 
\item  $\| (\bar{g}, \bar{\pi}) - (g,\pi)\|_{C^{2,\alpha}_{-q}(M)\times C^{1,\alpha}_{-1-q}(M)}  <\epsilon$ and $\|(\bar{\mu}, \bar{J}) - (\mu, J)\|_{L^1(M)}<\epsilon$,
\end{itemize}
where $(\overline{E}, \overline{P})$ is the ADM energy-momentum  and $(\bar{\mu}, \bar{J})$ is the energy and current density of $(\bar{g}, \bar{\pi})$.
\end{enumeratei}
\end{theorem}
\begin{remark} If $(\mu, J)\in C^{0,\alpha}_{-n-\delta}$ for some $\delta\in(0,1)$, then we can arrange for $(\bar\mu, \bar{J})$ to be $\epsilon$-close in this space as well.
\end{remark}

\begin{proof}
 We will consider two cases. In Case 1, we assume $\Sc(g, \pi)>0$ at some point in $\Int M$. In Case 2, we assume that there exists a bounded neighborhood $V_0$ of $\partial M$ with smooth boundary such that the dominant energy scalar is improvable in~$\Omega_r \smallsetminus V_0$ for sufficiently large $r$. We will show that in either case, Item \eqref{item:deformation} holds.  We will then show that  the negations of both cases imply Item~\eqref{item:non-improvable}.

\vspace{6pt}
\noindent{\bf Case 1}: We assume $\Sc(g,\pi)>0$ at some point in $\Int M$.  

Let $V_0$ be an open neighborhood of $\partial M$ with smooth boundary such that $\Sc(g,\pi)>0$ somewhere outside~$V_0$. 
 There exists an open ball $B$ in $M \smallsetminus V_0$ and a $\delta>0$ such that  $\Sc(g,\pi)>2\delta$ on $B$. Choose $r_0>1$ large enough so that $B\cup V_0\subset \Omega_{r_0}$.   Let $(g_t, \pi_t)$ be the one-parameter family of initial data sets constructed in Proposition~\ref{proposition:conformal}. Thanks to the properties of $(g_t,\pi_t)$, in order to establish Item~\eqref{item:deformation}, it suffices to deform $(g_t, \pi_t)$ within
 $\Omega := \Omega_{r_1}\smallsetminus V_0$ to obtain the DEC, where $r_1$ is from  Proposition~\ref{proposition:conformal}.

To do this, we apply Theorem~\ref{theorem:improvability-error} to  $(\Omega, g, \pi)$ with $\delta$ and $B$ as chosen above. Recall that  Theorem~\ref{theorem:improvability-error} says that there exists a neighborhood $\mathcal{U}$ of $(g, \pi)$ in $C^{4,\alpha}({\Omega})\times C^{3,\alpha}({\Omega})$ such that for any $V\subset\subset\Int\Omega$, there exist
 constants $\epsilon, C>0$ such that for all 
  $(\gamma, \tau)\in \mathcal{U}$ and   $u\in C^{0,\alpha}_c(V)$ with $\| u \|_{C^{0,\alpha}(\Omega)}< \epsilon$, there exists $(h, w)\in C^{2,\alpha}_c(\Int\Omega)$ with $\|(h,w)\|_{C^{2,\alpha}(\Omega)}\le C\| u\|_{C^{0,\alpha}(\Omega)}$ such that the initial data set  $(\gamma+h, \tau+w)$ satisfies
\[
	\Sc(\gamma+h, \tau+w) \ge \Sc(\gamma,\tau)+u -\delta {\bf 1}_B.
\]

For $t$ sufficiently small, we set $(\gamma,\tau) = (g_t, \pi_t)$ and $u=\Sc(g_t, \pi_t)^-$.  (Recall $v= v^+ -v^-$ is the decomposition of a function $v$ into its positive and negative parts.) Then by the choice of $B$ and Proposition~\ref{proposition:conformal}, for $t$ sufficiently small, 
\begin{align*}
	&\Sc(g_t, \pi_t) > \delta \mbox{ in } B\\
	&\Sc(g_t, \pi_t)^- \mbox{ is compactly supported in } \Int\Omega\\
	&\|\Sc(g_t, \pi_t)^- \|_{C^{0,\alpha}({\Omega})}\to 0 \mbox{ as } t\to \infty.
\end{align*}
There exists $(h_t, w_t)\in C^{2,\alpha}_c(\Int\Omega)$ such that 
\[
\| (h_t, w_t)\|_{C^{2,\alpha}(\Omega)}\le C\|\Sc(g_t, \pi_t)^- \|_{C^{0,\alpha}({\Omega})}\to 0
\]
 and 
\begin{align*}
	\Sc(g_t+ h_t, \pi_t+w_t) &\ge \Sc(g_t, \pi_t) +\Sc(g_t, \pi_t)^- - \delta {\bf 1}_B\\
&= \Sc(g_t, \pi_t)^+ - \delta {\bf 1}_B > 0.
\end{align*}
It follows that for sufficiently small $t>0$, the initial data $(\bar{g}, \bar{\pi}):= (g_t + h_t, \pi_t + w_t)$  satisfies  Item~\eqref{item:deformation}.
 
 \vspace{6pt}
 \noindent {\bf Case 2:} Assume that there exists a bounded neighborhood $V_0$ of $\partial M$ such that $V_0$ has smooth boundary and that the dominant energy scalar is improvable in $\Omega_{r}\smallsetminus V_0$ for all sufficiently large $r$.

Choose $r_0>1$ such that $V_0 \subset \Omega_{r_0}$, and let $(g_t, \pi_t)$ be the one-parameter family of initial data sets constructed in Proposition~\ref{proposition:conformal}, 
and choose $r_1>r_0$ large enough so that Proposition~\ref{proposition:conformal} holds and the dominant energy scalar is improvable in $\Omega:=\Omega_{r_1}\smallsetminus V_0$. Once again, it suffices to deform $(g_t, \pi_t)$ within  $\Omega$ to obtain the DEC.  
 
 By definition of improvability on $\Omega$, 
 there exists a neighborhood $\mathcal{U}$   of $(g, \pi)$ in $C^{4,\alpha}({\Omega})\times C^{3,\alpha}({\Omega})$ 
 such that for any $V\subset\subset\Int\Omega$, there exist
   constants $\epsilon, C>0$ such that for all 
  $(\gamma, \tau)\in \mathcal{U}$ and   $u\in C^{0,\alpha}_c(V)$ with $\| u \|_{C^{0,\alpha}(\Omega)}< \epsilon$, there exists $(h, w)\in C^{2,\alpha}_c(\Int\Omega)$ with $\|(h,w)\|_{C^{2,\alpha}(\Omega)}\le C\| u\|_{C^{0,\alpha}(\Omega)}$ such that the initial data set  $(\gamma+h, \tau+w)$ satisfies
\[
	\Sc(\gamma+h, \tau+w) \ge \Sc(\gamma,\tau)+u.
\]
Once again, for small $t>0$, we can choose $(\gamma, \tau)=(g_t,\pi_t)$ and $u=\Sc(g_t, \pi_t)^-$, and the argument is the same as in Case 1, except now there is no $- \delta {\bf 1}_B$ term to worry about.
\vspace{6pt}

We have shown that the two cases above imply Item \eqref{item:deformation} to hold. Now, suppose  we have the negations of both Case 1 and Case 2. The negation of Case 1 states that that $\Sc\equiv 0$ in all of $\Int M$.  Let $V_k$ be an open neighborhood of $\partial M$ with smooth boundary and lying within a distance $\tfrac{1}{k}$ from $\partial M$. By the negation of Case 2, there exists a sequence $r_k\to\infty$ such that for all $k$, 
the dominant energy scalar is not improvable in $\Omega_{r_k}\smallsetminus V_k$. Since $\Omega_{r_k}\smallsetminus V_k$ exhausts $\Int M$, we can invoke Theorem~\ref{theorem:AF} to obtain the desired lapse-shift pair $(f, X)$ satisfying the system \eqref{equation:pair}.

\end{proof}

\subsection{Bartnik mass and admissible extensions}\label{section:admissibility}

Before we present the definition of Bartnik mass, we will discuss admissibility of an asymptotically flat extension. As already addressed in \cite{Bartnik:1997}, it is necessary to impose some no-horizon or non-degeneracy condition in defining an admissible extension, but there is no clear consensus for what exactly it should be. (Even in the time-symmetric case, there are potentially inequivalent choices, which can lead to different desirable properties of Bartnik mass, see~\cite{Jauregui:2019}.) We define a no-horizon condition in the following.

For a closed, two-sided, immersed hypersurface $\Sigma$ in an initial data set $(M, g, \pi)$, we define the \emph{null expansion} $\theta^+ := H_\Sigma+ \mathrm{tr}_\Sigma k$ with respect to the unit normal $\nu$, where  $H_\Sigma$ is the tangential divergence of $\nu$ and $\mathrm{tr}_\Sigma k$ is the trace of $k$ over the tangent space of $\Sigma$. (Recall $k$ is related to $\pi$ by \eqref{equation:k}.) We say an immersed hypersurface $\Sigma$ is a \emph{marginally outer trapped hypersurface} (or \emph{MOTS} for short) if $\theta^+$ vanishes  everywhere on $\Sigma$.

Separately, in asymptotically flat $(M, g, \pi)$, we say that an  immersed closed hypersurface $\Sigma$ in $M$ is an \emph{outer embedded boundary} if $\Sigma = \partial U$, where $U$ is a connected open set that contains the infinity (that is, the asymptotically flat end). The definition implies  that $\Sigma$ cannot have transversal self-intersection and that the only way that $\Sigma$ can fail to be embedded is when locally two sheets of $\Sigma$ touch tangentially from inside, where ``inside'' refers to the complement of $U$. 

\begin{definition}
We say that an asymptotically flat initial data set $(M, g, \pi)$, possibly with boundary, satisfies $\mathscr{N}_1$ if it contains no smooth MOTS that is an outer embedded boundary, except possibly $\partial M$ itself.  
\end{definition}

Note the condition $\mathscr{N}_1$ says that $(M, g, \pi)$ contains neither embedded MOTS homologous to $\partial M$ (except possibly $\partial M$ itself) nor MOTS that is immersed and touches itself tangentially from inside. We remark that in the time-symmetric case, a MOTS is a minimal surface, and a minimal surface that is an outer embedded boundary is necessarily embedded because of a comparison principle. However, the same reasoning does not hold true for general MOTS.

In the next result, we show that the condition $\mathscr{N}_1$ is ``open'' among certain small deformations of $(g, \pi)$. This is  the only part of the overall proof of Theorem~\ref{theorem:Bartnik} where we have to assume that $3\le n \le 7$.  

\begin{proposition}\label{proposition:openness}
Let $3\le n \le 7$ and let $(M, g, \pi)$ be an $n$-dimensional asymptotically flat initial data set  with nonempty smooth boundary $\partial M$. Suppose $(M, g, \pi)$  satisfies the condition~$\mathscr{N}_1$. For each bounded neighborhood $V_0$ of $\partial M$, there exists $\epsilon>0$ such that if $(\bar{g}, \bar{\pi})$ is an initial data set on $M$ with $(\bar{g}, \bar{\pi})=(g, \pi)$ everywhere in $V_0$ and $\|(\bar{g}, \bar{\pi})- (g, \pi) \|_{C^{2}_{-q}(M)\times C^{1}_{-1-q}(M)} <\epsilon$, then $(M, \bar{g}, \bar{\pi})$ also satisfies the condition~$\mathscr{N}_1$. 
\end{proposition}
\begin{proof}
In the following argument,  we use the results of Andersson and Metzger~\cite{Andersson-Metzger:2009, Andersson-Metzger:2010}, Eichmair~\cite{Eichmair:2009, Eichmair:2010}, and the survey paper of Andersson, Eichmair, and Metzger~\cite{Andersson-Eichmair-Metzger:2011}.  

Suppose, on the contrary, that no such $\epsilon$ exists. Then there exists a sequence $(g_i, \pi_i)$ of initial data sets with $(g_i, \pi_i)=(g, \pi)$ everywhere in $V_0$ and $\|(g_i, \pi_i)- (g, \pi) \|_{C^{2}_{-q}(M)\times C^{1}_{-1-q}(M)} \to 0$ so that $\mathscr{N}_1$ fails for each $(g_i, \pi_i)$. Namely, each $(M, g_i, \pi_i)$ contains a closed MOTS $\Sigma_i$ that is an outer embedded boundary. Ideally, we would like to extract a subsequence of $\Sigma_i$ converging to some $\Sigma$ in $(M, g, \pi)$, but this cannot be done directly. First we will produce a sequence of \emph{stable embedded} MOTS $\Sigma_i'$ in $(M, g_i, \pi_i)$.

There is a sufficiently large $r$ so that the coordinate spheres of radius greater than or equal to $r$ all have positive $\theta^+$  (with respect to the unit normal pointing to infinity) in $(M, g_i, \pi_i)$ for all $i$. By a MOTS comparison principle~(see, for example, \cite[Proposition 4]{Eichmair-Huang-Lee-Schoen:2016}), it follows that $S_r$ must enclose $\Sigma_i$. 
Let $U_i$  be the part of $\Int \Omega_r$ lying outside of $\Sigma_i$. That is, $\partial U_i$ is  the disjoint union of $\Sigma_i$ and $S_r$, and we have $\theta^+=0$ on $\Sigma_i$ and $\theta^+>0$ on $S_r$, both with respect to the unit normal pointing toward infinity in $(g_i, \pi_i)$. 

In order to obtain a stable embedded MOTS in $U_i$, we follow the idea of \cite[Theorem 5.1]{Andersson-Metzger:2009} to slightly modify the initial data near $\Sigma_i$.  Let $\phi \ge 0$ be a smooth scalar function on $M$ so that $\phi\equiv 0$ in a neighborhood of $S_r$ and $\phi>0$ on all $\Sigma_i$. Let $\delta_i\to 0$ be a sequence of positive numbers, and let $\pi_{i}' = \pi_i +\delta_{i} \phi g_i$. It is easy to see that with respect to $(g_i, \pi_i')$, the null expansion $\theta^+$ becomes strictly negative everywhere on $\Sigma_i$ while $S_r$ has the same positive $\theta^+$. By \cite[Theorem 3.3]{Andersson-Eichmair-Metzger:2011}, there exists a closed \emph{embedded} MOTS $\Sigma_i'$ in $(U_i, g_i, \pi_i')$ that is  also \emph{stable} in the sense of MOTS,\footnote{We refer to \cite[Section 3.6]{Andersson-Eichmair-Metzger:2011} for stability of MOTS, and we note $\Sigma_i$ also satisfies a stability inequality in the sense of \cite[(5)]{Eichmair:2010} that is sufficient to apply Schoen-Simon's regularity theory in \cite{Schoen-Simon:1981}.} and is $C$-almost minimizing\footnote{Here, $C$-almost minimizing is in the sense of~\cite{Eichmair:2009}.}
  in $U_i$, where $C$ can be chosen independent of~$i$ (since $C$ depends only on $\sup_M |\pi_i'|_{g_i}$). Also, $\Sigma_i'$ is homologous to $S_r$, and we define $U_i'\subset U_i$ to be the open set bounded between $\Sigma_i'$ and~$S_r$.

The $C$-almost minimizing property of $\Sigma_i'$ in $U_i$ implies that $\Sigma_i'$ has a uniform area bound, and the equation $\theta^+=0$  implies $\Sigma_i'$ has uniformly bounded mean curvature. Together with the stability of~$\Sigma_i'$, we can apply the compactness result of \cite{Schoen-Simon:1981} (see also \cite[Theorem 1.3]{Andersson-Metzger:2010} and \cite[Theorem A.2]{Eichmair:2010}) to obtain a subsequential limit $\Sigma$, which is a closed immersed  MOTS in $(M, g, \pi)$. By passing to a further subsequence if necessary, $U_i'$ converges to some $U$ as currents, where $\partial U= S_r - \Sigma$ as currents.  

We show that the $C$-almost minimizing property of the sequence $\Sigma_i'$ in $U_i'$ guarantees that $\Sigma$ is a smooth boundary component of $U$, and thus $\Sigma$ is an outer embedded boundary. Suppose, to get a contradiction, $\Sigma$ has no collar neighborhood in $U$, which implies that two sheets of $\Sigma$ must touch at some point from inside of $U$. We can argue as in \cite[p. 23]{Andersson-Eichmair-Metzger:2011} that in a neighborhood of that point, for sufficiently large $i$, two sheets of $\Sigma_i'$ would be close enough in $U_i'$ so that the area of $\Sigma'$ can be reduced by adding a catenoid neck to violate the $C$-almost minimizing property in $U_i'$. (Note that the above argument does not give ``inner'' embeddedness of $\Sigma$  because  two sheets of $\Sigma$ can possibly touch from the complement of $U$ and $U_i$, where the $C$-almost minimizing property of $\Sigma_i'$ is not known to hold.)

Finally, we check that $\Sigma$ is not equal to $\partial M$.  Since $(M, g, \pi)$ satisfies $\mathscr{N}_1$ and $(g_i, \pi_i)$ is identical to $(g, \pi)$ on $V_0$, each $\Sigma_i$ must intersect the complement of $V_0$, and consequently, so must $\Sigma_i'$.  Hence $\Sigma$ also intersects the complement of $V_0$. In summary, $\Sigma$
is a MOTS that is an outer embedded boundary and is not $\partial M$. This contradicts the assumption that $(M, g, \pi)$ satisfies~$\mathscr{N}_1$. 

Note that if $\partial M$ is not a MOTS, then we do not need the argument in the previous paragraph to see that the MOTS limit $\Sigma$ cannot equal $\partial M$. So in that case, the assumption that $(\bar{g}, \bar{\pi})=(g, \pi)$ in~$V_0$ in Proposition~\ref{proposition:openness} is not needed.

\end{proof}

\begin{definition}\label{definition:admissible}
Let $(\Omega_0, g_0, \pi_0)$ be a $n$-dimensional compact initial data set with nonempty smooth boundary.
We say that $(M, g, \pi)$ is an \emph{admissible extension of $(\Omega_0, g_0, \pi_0)$} if the following holds:
\begin{enumeratei}
\item $(M, g, \pi)$ is an $n$-dimensional asymptotically flat initial data set (as defined in Appendix~\ref{section:asymp_flat}) with boundary $\partial M$, satisfying the dominant energy condition. 
\item  There exists an identification of the boundaries $\partial M$ and  $\partial \Omega_0$ via diffeomorphism, and under this identification, the following equalities hold along $\partial M \cong \partial \Omega_0$:
\begin{align*}
g_0|_{\partial \Omega_0} &= g|_{\partial M} \\
H_{\partial \Omega_0} &= H_{\partial M} \\
\pi_0\cdot \nu_0 &= \pi\cdot\nu.
\end{align*}
Here, $\nu_0$ and $\nu$ denote the unit normals with respect to $g_0$ and $g$, respectively (both of which point into $M$ and out of $\Omega_0$). The first equation is between the induced metrics, the second equation is between the mean curvatures (computed with respect to $g_0$ and $g$ and the normals $\nu_0$ and $\nu$, respectively), and the third equation is between the $g_0$-contraction of $\pi_0$ with $\nu_0$ and the $g$-contraction of $\pi$ with $\nu$.\footnote{The third identity is equivalent to asking for a matching condition on the tangential trace of $k$ and on the one-form $k(\nu, \cdot)$ restricted on the tangent space of the boundary as  in \cite[Definition 2]{Bartnik:1997}. }

\item $(M, g, \pi)$ satisfies the no-horizon condition $\mathscr{N}_1$.

\end{enumeratei}
\end{definition}

\begin{remark}\label{remark:admissible}
We include two other conditions that can replace $\mathscr{N}_1$ in Definition~\ref{definition:admissible}. 

The condition $\mathscr{N}_2$ requires an asymptotically flat initial data set $(M, g, \pi)$ to satisfy $\mathscr{N}_1$, and  $\partial M$ itself is not a MOTS. The last sentence of the proof of Proposition~\ref{proposition:openness}
implies that $\mathscr{N}_2$ is an open condition with respect to deformations of initial data in $C^{2}_{-q}(M)\times C^{1}_{-1-q}(M)$ that fix the induced data on the boundary. 

The condition $\mathscr{N}_3$ says that $\partial M$ is strictly outward-minimizing in $(M, g, \pi)$, in the sense that it has volume strictly less than any hypersurface enclosing it. As discussed in \cite[Section 6]{Jauregui:2019}, the condition $\mathscr{N}_3$ is an open condition with respect to deformation of metrics in $C^{1}_{-q}$ that fix the induced metric on the boundary. 

All results in this paper will hold if we replace the condition $\mathscr{N}_1$ in Definition~\ref{definition:admissible} by $\mathscr{N}_2$, or $\mathscr{N}_3$, or any combination of these conditions.
\end{remark}

Let $(M, g, \pi)$ be an asymptotically flat initial data set with the ADM energy-momentum $(E, P)$. If  $E\ge |P|$, we define the ADM mass to be $m_{\mathrm{ADM} }:= \sqrt{E^2 - |P|^2}$. If $E<|P|$, then we define $m_{\mathrm{ADM}}=-\infty$, purely for the sake of convenience. 
\begin{definition}
Let $(\Omega_0, g_0, \pi_0)$ be a compact initial data set with nonempty smooth boundary, satisfying the dominant energy condition. Define
$\mathcal{B}$ to be the set of all admissible extensions $(M, g, \pi)$ of $(\Omega_0, g_0, \pi_0)$. If this set is nonempty, then we define the 
\emph{Bartnik mass} of $(\Omega_0, g_0, \pi_0)$ to be
\[ 
m_B(\Omega_0, g_0, \pi_0):=\inf_{(M, g, \pi)\in\mathcal{B}}  m_{\mathrm{ADM}}(g, \pi).
\]
We say that $(M, g, \pi)\in \mathcal{B}$ is a \emph{Bartnik mass minimizer} for $(\Omega_0, g_0, \pi_0)$ if it achieves this infimum. 
\end{definition}
\begin{remark}\label{remark:pmt_corners}
As long as an appropriate spacetime positive mass theorem ``with corners'' holds, each $(M, g, \pi)$ in the definition above has $E\ge|P|$, and consequently, $m_B(\Omega_0, g_0, \pi_0)\ge0$. Although we cannot find such a result stated in the literature, we observe that such a theorem holds for \emph{spin} manifolds, by essentially the same reasoning as in the proof of the time-symmetric case in~\cite[Theorem 3.1]{Shi-Tam:2002}.
The only difference in the spacetime case is that the appropriate Dirac operator leads to an extra boundary term on p.~104, which precisely corresponds to the condition $\pi_0\cdot\nu_0=\pi\cdot\nu$ in Definition~\ref{definition:admissible}. Consequently, if one adjusts the definition of admissibility to demand that $\Omega_0\cup M$ glued along their common boundaries is spin, then the Bartnik mass is always nonnegative.

\end{remark}

\begin{proof}[Proof of Theorem~\ref{theorem:Bartnik}]
We need to establish Items (1) through(4) of Theorem~\ref{theorem:Bartnik}.
By Proposition~\ref{proposition:openness} and the definition of admissibility, it is clear that the deformed initial data set $(\bar{g}, \bar{\pi})$ in Item \eqref{item:deformation} of Theorem~\ref{theorem:Bartnik2} is an admissible extension, but it would contradict the ADM mass minimizing property of $(M, g, \pi)$. Therefore Item~\eqref{item:non-improvable} of Theorem~\ref{theorem:Bartnik2} must hold. That is, $\sigma(g,\pi)=0$ everywhere in $\Int M$ (which is Item~\eqref{item:sigma_vanish}),  and there exists a nontrivial solution $(f, X)$ on all of $\Int M$ solving the system~\eqref{equation:pair}. We apply  Theorem~\ref{theorem:f_not_zero} to conclude Items~\eqref{item:Bartnik-spacetime} and~\eqref{item:spacetime_vac}. Item~\eqref{item:vanishing-J} follows from  Lemma~\ref{lemma:unbounded}.

\end{proof}
\begin{remark}
In the time-symmetric case, given $(\Omega_0, g_0, 0)$, it is an open question whether a time-symmetric Bartnik mass minimizer $(M, g, 0)$ must be unique. For an initial data set $(\Omega_0, g_0, \pi_0)$ it is clear that a Bartnik mass minimizer $(M, g, \pi)$ is \emph{not} unique because one can take different asymptotically flat spacelike hypersurfaces in the same spacetime development to obtain different minimizing extensions. However, it is still interesting to ask whether the spacetime development of a Bartnik mass minimizer is unique. 
\end{remark}

\subsection{Bartnik energy}\label{section:energy}

We define the quasi-local energy of a compact initial data set $(\Omega_0, g_0, \pi_0)$ with nonempty smooth boundary in a similar fashion as the Bartnik mass.

\begin{definition}
Let $(\Omega_0, g_0, \pi_0)$ be a compact initial data set with nonempty smooth boundary, satisfying the dominant energy condition. Define $\mathcal{B}$ to be the set of all admissible extensions $(M, g, \pi)$ of $(\Omega_0, g_0, \pi_0)$ as in Definition~\ref{definition:admissible}. If this set is nonempty, then we define the \emph{Bartnik energy} to be
\[ 
E_B(\Omega_0, g_0, \pi_0):=\inf_{(M, g, \pi)\in\mathcal{B}}  E( g, \pi).
\]
We say that $(M, g, \pi)\in \mathcal{B}$ is a \emph{Bartnik energy minimizer} for $(\Omega_0, g_0, \pi_0)$ if it achieves this infimum.
\end{definition}

In the next result, we observe that Bartnik mass and Bartnik energy should be equal. It may be physically interpreted  that a natural definition of the Bartnik quasi-local linear momentum by $|P_B|:=\sqrt{E_B^2-m_B^2}$ is always zero, and thus the quasi-local quantities defined by minimizing the corresponding quantities of asymptotically flat extensions do not capture  ``dynamics'' information. Unfortunately, our proof requires $C^\infty$ smoothness.

\begin{theorem}
For this theorem only, we re-define admissibility to require admissible extensions $(g,\pi)$ to be $C^\infty_{\mathrm{loc}}$ on $\Int M$, which then slightly changes the formal definitions of $m_B$ and $E_B$.

Let $(\Omega_0, g_0, \pi_0)$ be a three-dimensional compact initial data set with nonempty smooth boundary, satisfying the dominant energy condition. If an admissible extension exists and $m_B(\Omega_0, g_0, \pi_0)>0$,
 then 
\[ m_B(\Omega_0, g_0, \pi_0) = E_B(\Omega_0, g_0, \pi_0).\]
Moreover, $(M, g, \pi)$ is a Bartnik energy minimizer for $(\Omega_0, g_0, \pi_0)$ if and only if it is a Bartnik mass minimizer with ADM momentum equal to zero. 
\end{theorem}
\begin{remark}
The dimension assumption is used to perform gluing to an exact Kerr initial data set in the exterior because the spacelike slices in 4-dimensional Kerr spacetimes form an ``admissible family'' for gluing construction. See \cite[Appendix F]{Chrusciel-Delay:2003} and  Definition 4.5 and Example 4.6 in \cite{Corvino-Huang:2020}. We expect the analogous result to hold in higher dimensions. 
\end{remark}

To prepare for the proof, we review basic facts about an $(n+1)$-dimensional asymptotically flat spacetime. Given a coordinate chart $(t, x^1, \ldots, x^n)$ in the asymptotically flat region of spacetime,  where $(x^1, \ldots, x^n)$ is a spatial asymptotically flat coordinate chart, we define a \emph{boosted slice of angle $\beta$} to be a spacelike hypersurface defined $t = \beta x^1$ for some number $ \beta \in (-1, 1 )$.  Those boosted slices define a family of $(n-1)$-dimensional submanifolds $\Sigma_{\beta, r}$ by intersecting  the hyperplane $\{ t = \beta x^1\}$ with the cylinder of radius $\{ |x| = r\}$. With respect to the Minkowski metric, those $\Sigma_{\beta, r}$ have  null expansion $\theta^+>0$, as well as positive mean curvature, because they are the isometric images of ellipsoids in the $t=0$ slice. Therefore, with respect to a Lorentzian metric that is asymptotically flat,  those submanifolds $\Sigma_{\beta, r}$ have positive $\theta^+$ and positive mean curvature for all $\beta$, for all sufficiently large $r$. We will restrict our attention to  Kerr spacetimes and consider boosted slices with respect to a Boyer-Lindquist coordinate chart $(t, r,\theta, \phi)$ by letting $(x^1, x^2, x^3)$ be the Cartesian coordinates corresponding to $(r, \theta, \phi)$. 
 
\begin{proof}
We clearly have $m_B(\Omega_0, g_0, \pi_0)\le E_B(\Omega_0, g_0, \pi_0)$ by definition, so we need only prove the reverse inequality.

For any $\epsilon>0$, our hypotheses imply that there exists an admissible extension $(M, g, \pi)$ of $(\Omega_0, g_0, \pi_0)$ such that $0<  m_{\mathrm{ADM}}(g, \pi) < m_B(\Omega_0, g_0, \pi_0)+\epsilon$. We will perform a two-step process to construct another admissible extension $(M, \tilde{g}, \tilde{\pi})$ such that 
\[ E(\tilde{g}, \tilde{\pi}) < m_{\mathrm{ADM}}(g, \pi) +\epsilon,\] 
which implies the desired inequality $E_B(\Omega_0, g_0, \pi_0)\le m_B(\Omega_0, g_0, \pi_0)$.

The first step is to apply a gluing theorem of Corvino and the first author~\cite[Theorem~1.4]{Corvino-Huang:2020} to construct initial data $(\bar{g}, \bar{\pi})$ on $M$ with the following properties:
 For $R>0$ sufficiently large, 
\begin{itemize}
\item $(\bar{g}, \bar{\pi})= (g, \pi)$ on $\Omega_R$, where $\Omega_R$ is the compact subset enclosed by $|x|=R$,
\item $(\bar{g}, \bar{\pi})$ is equal to the initial data on a boosted slice of the Kerr spacetime outside $\Omega_{2R}$, 
\item $(\bar{g}, \bar{\pi})$ satisfies the DEC everywhere,
\item   $0<  m_{\mathrm{ADM}}(\bar{g}, \bar{\pi}) < m_{\mathrm{ADM}}(g, \pi) +\epsilon$,
\item   $(\bar{g}, \bar{\pi})\in C^\infty_{\mathrm{loc}}(\Int M)$\footnote{This is the step of the argument that requires $C^\infty_{\mathrm{loc}}$ smoothness. In general, applying this gluing theorem causes a loss of derivatives.}
 and
$\| (\bar{g}, \bar{\pi})- (g, \pi)\|_{C^{2,\alpha}_{-q'}(M)\times C^{1,\alpha}_{-1-q'}(M)}<\epsilon$ for some $q' \in (\tfrac{1}{2}, q)$,
\end{itemize}
 We note that the $\epsilon$-closeness in $C^{2,\alpha}_{-q'}(M)\times C^{1,\alpha}_{-1-q'}(M)$ statement is implicitly contained in the proof of \cite[Theorem 4.9]{Corvino-Huang:2020}: It is clear that the Kerr initial data is $\epsilon$-close to $(g, \pi)$ outside $\Omega_{2R}$ for $R$ large by asymptotical flatness. It suffices to verify such an estimate  in the transition region $A_R:=\Omega_{2R} \smallsetminus \Omega_R$ of the deformation, denoted by $(h, w)$. The estimate of the rescaled $(h^R, w^R)$ in the unit annulus is obtained in \cite[Proposition 4.4]{Corvino-Huang:2020}, which implies $\|(h, w)\|_{C^{2,\alpha}_{-q}(A_R)\times C_{-1-q}^{1,\alpha}(A_R)}\le C$.  Hence, the estimate of $(h, w)$ in the weighted norm with a slower fall-off rate $q'$ can be made arbitrary small for large $R$. Together with Proposition~\ref{proposition:openness} and Remark~\ref{remark:admissible}, we see that $(M, \bar{g}, \bar{\pi})$ is admissible, as long as $\epsilon$ is small enough. 
 
 For the second step, we look at the portion of $(M, \bar{g}, \bar{\pi})$ outside $\Omega_{2R}$ that is exactly a boosted slice in the Kerr spacetime, which we may express as $t = \beta x^1$ for some number $\beta\in (-1, 1)$. We now ``bend'' this slice to the $t=0$ slice by letting $t = \eta(r) x^1$, where $\eta(r)$ is a smooth scalar function depending only on $r$ such that $\eta(2R) = \beta$ and $\eta(r) = 0$ for $r\ge 3R$. We define a new $(\tilde{g}, \tilde{\pi})$ as follows:
\begin{itemize}
\item $(\tilde{g}, \tilde{\pi}) = (\bar{g}, \bar{\pi})$ on  $\Omega_{2R}$.
\item Outside $\Omega_{2R}$, $(\tilde{g}, \tilde{\pi})$ equals the induced data on $\{ t = \eta(r) x^1\}$ in the Kerr spacetime.
\end{itemize}
 One feature of this bending process is that it ensures that for $R$ sufficiently large, the portion of $(M, \tilde{g}, \tilde{\pi})$ outside of $\Omega_{2R}$ is foliated by hypersurfaces of positive null expansion and positive mean curvature. 
Since the boosted slice $t = \beta x^1$ has the same ADM mass as the $t=0$ slice, and the latter has zero linear momentum, we have \[ E(\tilde{g}, \tilde{\pi})=m_{\mathrm{ADM}}(\tilde{g}, \tilde{\pi})=m_{\mathrm{ADM}}(\bar{g}, \bar{\pi})< m_{\mathrm{ADM}}(g, \pi) +\epsilon.\]

The only thing left to verify is admissibility of $(M, \tilde{g}, \tilde{\pi})$. The first two properties in Definition~\ref{definition:admissible} are clear, which only leaves the no-horizon condition. 
 The part of $(\tilde{g}, \tilde{\pi})$ that is different from $(\bar{g}, \bar{\pi})$ can be foliated hypersurfaces of positive null expansion, so the comparison principle for $\theta^+$ implies that condition $\mathscr{N}_1$ on $(\bar{g}, \bar{\pi})$ implies condition $\mathscr{N}_1$ on $(\tilde{g}, \tilde{\pi})$ as well. (The same is true for the conditions $\mathscr{N}_2$ and $\mathscr{N}_3$.)
 
Finally, the last sentence of the theorem is an immediate consequence of the equality $m_B(\Omega_0, g_0, \pi_0) = E_B(\Omega_0, g_0, \pi_0)$.

\end{proof}
\appendix
\section{Asymptotically flat manifolds}\label{section:asymp_flat}

Let $M$ be a connected, $n$-dimensional manifold, possibly with boundary. For $q\in\left(\tfrac{n-2}{2}, n-2\right)$ and $\alpha\in(0,1)$, 
we say that an initial data set $(M, g, \pi)$ is \emph{asymptotically flat} of type $(q, \alpha)$ if  there exists  a compact subset $K \subset M$ and a diffeomorphism $M\smallsetminus K  \cong \mathbb{R}^n \smallsetminus B$ such that 
\[
	(g - g_{\mathbb{E}}, \pi) \in C^{2,\alpha}_{-q} (M)\times C^{1,\alpha}_{-1-q} (M)
\]
and
\[
	(\mu, J) \in L^1 (M),
\]
where $g_{\mathbb{E}}$ is a Riemannian background metric on $M$ that is equal to the Euclidean metric in the coordinate chart $M\smallsetminus K  \cong \mathbb{R}^n \smallsetminus B$. Note that with this definition, $(M, g)$ is necessarily complete.  The function spaces above, $C^{k,\alpha}_{-q}$, refer to \emph{weighted H\"{o}lder spaces} (as defined in~\cite{Huang-Lee:2020}, for example). Note that our convention is  such that $f\in C^{k,\alpha}_{-q}(M)$ if and only if $f\in C^{k,\alpha}_{\mathrm{loc}}(M)$ and there is a positive constant $C$ such that, for any multi-indices $I$ with $|I|\le k$,
\begin{align*}
	|(\partial^I f)(x)|\le C|x|^{-|I|-q} \quad \mbox{ and } \quad [f]_{k,\alpha; B_1(x)} \le C|x|^{-k-\alpha-q}
\end{align*}
on $M\smallsetminus K$.  We use the notation $O_{k,\alpha}(|x|^{-q})$ to denote an arbitrary function that lies in $C^{k,\alpha}_{-q}$, and we also simply write $O(|x|^{-q})$ for $O_0(|x|^{-q})$.

 Let $\Omega_k$ be a sequence of compact subsets with smooth boundary that exhausts $\mathbb{R}^n$, and define $\Sigma_k := \partial \Omega_k$.
The \emph{ADM energy} $E$ and the \emph{ADM linear momentum} $P=(P_1, \dots, P_n)$ of an asymptotically flat initial data set $(M, g, \pi)$ are defined as 
\begin{align*}
	E&:= \tfrac{1}{2(n-1)\omega_{n-1}} \lim_{k\to \infty} \int_{\Sigma_k}\sum_{i,j=1}^n (g_{ij,i}-g_{ii,j})\nu^j \, d\mu\\
	P_i &:= \tfrac{1}{(n-1)\omega_{n-1}} \lim_{k\to \infty} \int_{\Sigma_k} \sum_{i,j=1}^n \pi_{ij} \nu^j \, d\mu
\end{align*}	
where the integrals are computed on $\Sigma_k$ in $M\smallsetminus K \cong \mathbb{R}^n \smallsetminus B$, $\nu^j$ is the outward unit normal to~$\Sigma_k$, $d\mu$ is the measure on  $\Sigma_k$ induced by the Euclidean metric,  $\omega_{n-1}$ is the volume of the standard $(n-1)$-dimensional unit sphere, and the commas denote partial differentiation in the coordinate directions. The condition $q>\frac{n-2}{2}$ and integrability of $(\mu, J)$ imply that the limits in the definition of ADM energy-momentum exist and are independent of the choice of exhaustion $\Sigma_k$.

\section{Spacetime with a Killing vector field}\label{section:spacetime}

For a spacetime admitting a Killing vector field, the Einstein tensor along a spacelike hypersurface can be expressed in terms of the lapse-shift pair of the Killing vector. While these formulas have appeared in the literature, because of different sign and normalization conventions, we include a self-contained discussion of the curvature formulas that will be used elsewhere in this paper. 

Let $(\mathbf{N}, \mathbf{g})$ be an $(n+1)$-dimensional, time-oriented spacetime equipped with a spacelike hypersurface $U$, and let $G:= \mathrm{Ric}_{\mathbf{g}} - \tfrac{1}{2} R_{\mathbf{g}} \mathbf{g}$ denote the Einstein tensor of $\mathbf{g}$. Consider any local frame $e_0, e_1\ldots, e_n$ of~$\mathbf{N}$ such that $e_0=\mathbf{n}$ is the future unit normal of a spacelike hypersurface $U$ while $e_1,\ldots, e_n$ are tangent to $U$.

 As a simple consequence of the Gauss and Codazzi equations:
\begin{align*}
G_{00}  &= \tfrac{1}{2} (R_g - |k|_g^2 + (\tr_g k)^2) \\
G_{0i} & = (\Div_g k)_i - \nabla_i (\tr_g k), 
\end{align*}
where $g$ is the induced metric on the hypersurface, and we define the second fundamental form $k$ of a spacelike hypersurface to be the tangential component of $\bm{\nabla} \mathbf{n}$. This is the same as saying
\begin{align}
\begin{split}\label{equation:Einstein-0i}
G_{00}  &= \mu \\
G_{0i} & = J_i, 
\end{split}
\end{align}
where $\mu$ and $J$ are the energy and current densities of the induced initial data $(g,\pi)$. Of course, this was
the original motivation for the definitions of $\mu$ and $J$. 

Let $\mathbf{Y}$ be a vector field on $\mathbf{N}$ which is transverse to $U$ everywhere. Given coordinates  $(x^1, \ldots, x^n)$  on $U$, we extend these functions to a neighborhood of $U$  by making them constant on the flow lines of $\mathbf{Y}$. We also define a function $u$ via integrating $\mathbf{Y}$ from $u=0$ at~$U$ so that $\frac{\partial}{\partial u}=\mathbf{Y}$. Then $(u, x^1, \ldots, x^n)$ defines coordinates in a neighborhood of $U$ in $\mathbf{N}$, and it is straightforward to see that the metric must take the form
\begin{align}\label{equation:Killing-development}
\mathbf{g} = -4f^2 du^2 + g_{ij} (dx^i + X^i du) (dx^j + X^j du),
\end{align}
where $g_{ij}$ is the induced metric on each constant $u$ slice, and the decomposition $\mathbf{Y} = 2f \mathbf{n} +X$ holds along each one of these slices, where $\mathbf{n}$ is the future unit normal of the slice.  

\begin{lemma}\label{lemma:Einstein_tangential}
Suppose the spacetime $(\mathbf{N}, \mathbf{g})$ takes the form \eqref{equation:Killing-development}.
Then the following equations hold, where the $i,j$ indices run from $1$ to $n$: 
\begin{enumerate}
\item The second fundamental form of the constant $u$-slices is expressed as:
\begin{equation}\label{equation:second_fund}
 k_{ij} = \frac{1}{4f}\left[(L_{\mathbf{Y}} g)_{ij}- (L_X g)_{ij}\right].
 \end{equation}
\item  The Einstein tensor along  the constant~$u$ slices takes the form: 
\begin{align}
\begin{split}\label{equation:Einstein-ij}
G_{ij} &= \left[R_{ij} -\tfrac{1}{2}R_g g_{ij}\right] 
+ \left[(\tr_g k)k_{ij} - 2k_{i\ell}k^\ell_j \right]
 + \left[ -\tfrac{1}{2}(\tr_g k)^2 +\tfrac{3}{2}|k|_g^2 \right]g_{ij}\\
&\quad+ f^{-1}\left[ -\tfrac{1}{2}(L_X k)_{ij} + \tfrac{1}{2}\tr_g (L_X k)g_{ij} -  f_{;ij} +(\Delta_g f) g_{ij} \right]\\
&\quad+ (2f)^{-1}\left[ (L_{\mathbf{Y}} k)_{ij} -  \tr_g (L_{\mathbf{Y}} k)g_{ij} \right].
\end{split}
\end{align}

\end{enumerate}
\end{lemma}
\begin{proof}

The first equation is just a re-statement of a standard computation of $\frac{\partial }{\partial u} g_{ij}$. 

For the second equation, the Gauss equation implies that 
\[ \mathrm{Ric}_{\mathbf{g}}(e_i, e_j) = R_{ij}  + (\tr_g k) k_{ij} - k_{i\ell} k^\ell_{ j} 
-\mathrm{Rm}_{\mathbf{g}}(\mathbf{n}, e_i, e_j, \mathbf{n}),\]
and the trace gives
\[ 	
R_{\mathbf{g}} = R_g +  (\tr_g k)^2 - |k|_g^2  -2 \mathrm{Ric}_{\mathbf{g}}(\mathbf{n},\mathbf{n}).
\]

To compute $\mathrm{Rm}_{\mathbf{g}}(\mathbf{n}, e_i, e_j, \mathbf{n})$, we must understand 
$\bm{\nabla} \mathbf{n}$ better. While the tangential part is $k$, we can show that $\bm{\nabla}_{\mathbf{n}} \mathbf{n} =f^{-1}\nabla f$. Using the fact that $\mathbf{n}=-2f \bm{\nabla} u$, we have, for  $e_i$ tangential to the constant $u$-slices, 
\begin{align*}
	\mathbf{g} (\bm{\nabla}_{\mathbf{n}} \mathbf{n}, e_i) &=  -2 \bm{\nabla}^2 u (f\mathbf{n}, e_i)\\
	&=-2 \bm{\nabla}_{e_i} \bm{\nabla}_{f\mathbf{n}}u + 2\bm{\nabla}_{\bm{\nabla}_{e_i} (f\mathbf{n})} u\\
	&= 2f^{-1} e_i(f) \bm{\nabla}_{f\mathbf{n}} u\\
	&= f^{-1} e_i(f).
\end{align*}
Using our knowledge of $\bm{\nabla} \mathbf{n}$, we can show that 
\begin{align*}
 \mathrm{Rm}_{\mathbf{g}}(\mathbf{n}, e_i, e_j, \mathbf{n}) &= -f^{-1} \mathbf{g}(\bm{\nabla}_{f\mathbf{n}} \bm{\nabla}_{e_i} \mathbf{n} -  \bm{\nabla}_{e_i} \bm{\nabla}_{f\mathbf{n}} \mathbf{n} - \bm{\nabla}_{[f\mathbf{n}, e_i]} \mathbf{n}, e_j)\\ 
  &=f^{-1} \mathbf{g} \left(\bm{\nabla}_{e_i} \mathbf{n}, \bm{\nabla}_{e_j} (f\mathbf{n})\right) + f^{-1} \mathbf{g} (\bm{\nabla}_{e_i} \bm{\nabla}_{f\mathbf{n}} \mathbf{n}, e_j) \\
  &\quad - f^{-1} \left[ \mathbf{g}(\bm{\nabla}_{f\mathbf{n}} \bm{\nabla}_{e_i} \mathbf{n} - \bm{\nabla}_{[f\mathbf{n}, e_i]} \mathbf{n}, e_j) +\mathbf{g} \left(\bm{\nabla}_{e_i} \mathbf{n}, \bm{\nabla}_{e_j} (f\mathbf{n})\right)  \right]\\
&=  k_{i\ell} k^\ell_{ j} + f^{-1}\left[f_{;ij} 
-  (L_{f\mathbf{n}} k)_{ij}\right],
\end{align*}
and the trace gives
\[
 \mathrm{Ric}_{\mathbf{g}}(\mathbf{n},\mathbf{n}) = |k|_g^2 +f^{-1}[\Delta_g f 
-  \tr_g(L_{f\mathbf{n}} k)].
\]
Plugging these in to our equations for $\mathrm{Ric}_{\mathbf{g}}$ and $R_\mathbf{g}$, we obtain
\[ \mathrm{Ric}_{\mathbf{g}}(e_i, e_j) = R_{ij}  + (\tr_g k) k_{ij} - 2k_{i\ell} k^\ell_{ j} + f^{-1}[  (L_{f\mathbf{n}} k)_{ij}  -f_{;ij}],\]
and 
\[ 	
R_{\mathbf{g}} = R_g +  (\tr_g k)^2 - 3|k|_g^2 +2f^{-1}[
  \tr_g(L_{f\mathbf{n}} k)-\Delta_g f ].
\]
Combing these two equations and using the fact that $\mathbf{Y} = 2f\mathbf{n}+X$ yields the desired result.

\end{proof}

Now let us consider what happens when $\mathbf{Y}$ is Killing. We record the following easy fact.

\begin{lemma}\label{lemma:Killing-independent-u}
Suppose the spacetime $(\mathbf{N}, \mathbf{g})$ takes the form \eqref{equation:Killing-development}. Then $\mathbf{Y} := \frac{\partial}{\partial u}$ is Killing if and only if the functions  $g_{ij}$, $f$ and $X^i$ are all independent of~$u$. In particular, if $\mathbf{Y}$ is Killing, then $L_{\mathbf{Y}} g$ and $L_{\mathbf{Y}} k$ both vanish.
\end{lemma}

\begin{corollary}
Let $(\mathbf{N}, \mathbf{g})$ be a spacetime admitting a Killing vector field $\mathbf{Y}$. Let $U$ be a spacelike hypersurface with the induced data $(g, \pi)$ and future unit normal $\mathbf{n}$. Suppose $Y$ is transverse to $U$ and  $\mathbf{Y} = 2f \mathbf{n} + X$ along $U$. Then along $U$
\begin{align}\label{equation:Killing}
	\tfrac{1}{2} (L_X g)_{ij} =\left(\tfrac{2}{n-1} (\tr_g \pi) g_{ij} - 2\pi_{ij} \right)f,
\end{align}
and the tangential components of the Einstein tensor  take the form: 
\begin{align}
\begin{split}\label{equation:Einstein_pi}
G_{ij} &=  \left[R_{ij} -\tfrac{1}{2}R_g g_{ij}\right]+
 \left[- \tfrac{3}{n-1}(\tr_g \pi)\pi_{ij} + 2\pi_{i\ell}\pi^\ell_j
  \right]\\
&\quad  + \left[   \tfrac{1}{2(n-1)}(\tr_g \pi)^2 - \tfrac{1}{2}|\pi|_g^2 \right] g_{ij} \\
&\quad+f^{-1}\left[-\tfrac{1}{2}(L_X \pi)_{ij}  -  f_{;ij} +(\Delta_g f)g_{ij}\right]. 
\end{split}
\end{align}
\end{corollary}
\begin{proof}
As mentioned above, since $\mathbf{Y}$ is Killing, $L_{\mathbf{Y}} g$ and $L_{\mathbf{Y}} k$ are both zero, and then we can see that the Corollary is a direct (but tedious) consequence of Lemma~\ref{lemma:Einstein_tangential} after expressing involving $k$ in terms of $\pi$.

In more detail, we use the equations
\begin{align*}
k_{ij} &= g_{i\ell} g_{jm} \pi^{\ell m}  -\tfrac{1}{n-1}(\tr \pi) g_{ij}\\
(L_X k)_{ij} &= (L_X \pi)_{ij} + 2 \pi^\ell_j (L_X g)_{i\ell} -\tfrac{1}{n-1}(\tr \pi)(L_X g)_{ij} \\
&\quad 
 -\tfrac{1}{n-1} \left[\tr \, (L_X \pi) + \pi^{\ell m} (L_X g)_{\ell m}\right] g_{ij},
 \end{align*}
where  $(L_X \pi)_{ij} = g_{i\ell} g_{jm} (L_X\pi)^{\ell m}$. Here and below, the covariant derivative, trace, and norm are all with respect to $g$.

 Equation~\eqref{equation:Killing} follows immediately
from~\eqref{equation:second_fund}, and substituting into~\eqref{equation:Einstein-ij} yields:
\begin{align*}
G_{ij} &=  \left[R_{ij} -\tfrac{1}{2}R g_{ij}\right]+
 \left[ \tfrac{3}{n-1}(\tr \pi)\pi_{ij} - 2\pi_{i\ell}\pi^\ell_j
  \right]
 + \left[   -\tfrac{3}{2(n-1)}(\tr \pi)^2 +\tfrac{3}{2}|\pi|^2 \right] g_{ij} \\
&\quad+ f^{-1} \left[-\tfrac{1}{2}(L_X \pi)_{ij} - \pi^\ell_j (L_X g)_{i\ell} + \tfrac{1}{2(n-1)}(\tr \pi)(L_X g)_{ij}\right] \\
&\quad+f^{-1} \left[ \tfrac{1}{2}\pi^{\ell m} (L_X g)_{\ell m} - \tfrac{1}{n-1}(\tr\pi)(\Div X)\right] g_{ij} \\
&\quad +f^{-1} \left[-  f_{;ij} +(\Delta f)g_{ij}\right].
\end{align*}
Using equation~\eqref{equation:Killing} and its trace, we can eliminate the  $L_X g$ and $\Div X$ terms to obtain \eqref{equation:Einstein_pi}.

\end{proof}

\section{Initial data in pp-waves}\label{se:pp}

We prove Lemma~\ref{lemma:pp-initial}. Recall that a pp-wave spacetime metric is defined by 
\begin{equation} \label{eqn:pp}
\mathbf{g} = 2 du\,dx^n + S\, (dx^n)^2 +\sum_{a=1}^{n-1} (dx^a)^2
\end{equation}
where $S$ is a function independent of $u$. Observe that when $S>0$, the  metric $\mathbf g$ takes the form \eqref{equation:Killing-development} with the induced metric on the $u$-slices  
\begin{equation} \label{eqn:pp-slice}
	g =  S(dx^n)^2 +\sum_{a=1}^{n-1} (dx^a)^2,
\end{equation}
with $(f, X)$ given by 
 \begin{align*}
 f = \tfrac{1}{2} S^{-\frac{1}{2}} \quad \mbox{ and } \quad X= S^{-1} \tfrac{\partial}{\partial x^n}. 
\end{align*} 
Note that the orientation of $\mathbf n$ is chosen so that  $\mathbf{Y}=\tfrac{\partial}{\partial u}$ is future-pointing, that is, $-\mathbf g(\tfrac{\partial}{\partial u}, \mathbf n) = 2f >0$.

In the computations below, we will use commas to denote partial differentiation. We first compute the Christoffel symbols of $g$. For convenience, define $\Gamma_{ijk} := g(\nabla_{\partial_i} \partial_j, \partial_k) = \frac{1}{2} (g_{ik,j}+ g_{jk,i} - g_{ij,k})$. Then, for $a, b, c=1,\dots, n-1$,
\begin{align}\label{eqn:Christoffel-raised}
\begin{split}
	\Gamma_{abc} &= \Gamma_{nab} = \Gamma_{anb} = \Gamma_{abn}= 0, \quad \Gamma_{nna} = -\tfrac{1}{2} S_{,a}, \\
	\Gamma_{nan} &= \Gamma_{ann} = \tfrac{1}{2} S_{,a}, \quad \Gamma_{nnn} = \tfrac{1}{2} S_{,n} 
\end{split}
\end{align}
and the Christoffel symbols of $g$ are given by
\begin{align}\label{eqn:Christoffel}
\begin{split}
	\Gamma_{ab}^c &=\Gamma_{na}^b  = \Gamma_{an}^b= \Gamma_{ab}^n = 0,\quad \Gamma_{nn}^a = -\tfrac{1}{2} S_{,a},\\
	\Gamma_{na}^n&=\tfrac{1}{2} S^{-1} S_{,a}, \quad \Gamma_{nn}^n= \tfrac{1}{2} S^{-1} S_{,n}.
\end{split}
\end{align}

\begin{lemma}\label{le:pi}
 If $S>0$, then the conjugate momentum of the constant $u$-slices of the metric~\eqref{eqn:pp},  as a $(2, 0)$-tensor, is given by
\begin{align*}
	\pi^{nn} &= 0\\
	\pi^{na}&=\pi^{an} = \tfrac{1}{2} S^{-\frac{3}{2}}\tfrac{\partial S}{\partial x^a}\\
	\pi^{ab} &= -\tfrac{1}{2} S^{-\frac{3}{2}}\tfrac{\partial S}{\partial x^n}\, \delta^{ab}
\end{align*}
where the $a$  and $b$ indices run from $1$ to $n-1$.
\end{lemma}
\begin{proof}

Note that the covector of $X$ has the coefficients $X_n=1$ and $X_a=0$ for $a=1,\dots, n-1$. Since $L_{\mathbf Y} g=0$, we can use~\eqref{equation:second_fund} and~\eqref{eqn:Christoffel} to compute the second fundamental form $k$ to be
\begin{align*}
	k_{ij} &= - \tfrac{1}{4f} (X_{i;j} + X_{j;i} )= -\tfrac{1}{2} S^{\frac{1}{2}} \big(X_{i,j} + X_{j,i} - 2\Gamma_{ij}^n X_n\big) = S^{\frac{1}{2}} \, \Gamma_{ij}^n
\end{align*}
which gives 
\begin{align*}
	 k_{nn}&= \tfrac{1}{2} S^{-\tfrac{1}{2}} S_{,n}\\
	 k_{na}&=k_{an}= \tfrac{1}{2} S^{-\frac{1}{2}}S_{,a}\\
	k_{ab}&=0\\
	\tr_g k &= \tfrac{1}{2} S^{-\frac{3}{2}}  S_{,n}.
\end{align*}

Raising the indices of $k$ and using the relation $\pi^{ij} = k^{ij} - (\tr_g k) g^{ij}$, we obtain the desired result. 

\end{proof}

\begin{lemma}
 If $g$ is given by the formula~\eqref{eqn:pp-slice} and $\pi$ is given by the formulas in Lemma~\ref{le:pi}, then their  energy  and current densities $(\mu, J)$ are given by  
\begin{align*}
		\mu  =-\tfrac{1}{2} S^{-1} \Delta' S\quad \mbox{ and } \quad 	 J =\tfrac{1}{2} S^{-\frac{3}{2}} (\Delta' S) \tfrac{\partial}{\partial x^n}. 
\end{align*}
Consequently,  $|\mu| = |J|_g$, and $\mu\ge0$ so long as $\Delta' S \le0$.
\end{lemma}
This can be seen as a fairly direct consequence of Lemma~\ref{le:pi} combined with fact that the Einstein tensor of $\mathbf{g}$ from~\eqref{eqn:pp} is $G_{\alpha\beta}=-\tfrac{1}{2} (\Delta' S) Y_\alpha Y_\beta$, but here we will provide a direct proof in terms of the initial data $(g,\pi)$.
\begin{proof}
The most complicated step is to compute the scalar curvature of $g$.
 Using the general formula 
\[
	\mathrm{Rm}(\partial_i,\partial_j, \partial_k, \partial_{\ell} ) = \Gamma_{jk\ell, i} - \Gamma_{ik\ell, j} - \Gamma_{jk}^m \Gamma_{i\ell m} + \Gamma_{ik}^m \Gamma_{j\ell m},
\]
  for $i,j,k,\ell, m$ ranging from $1$ to~$n$, it follows from~\eqref{eqn:Christoffel-raised} and~\eqref{eqn:Christoffel} that the only nonzero curvature components are 
\begin{align*}
	\mathrm{Rm}(\partial_n, \partial_a, \partial_b, \partial_n) &= \Gamma_{abn,n} - \Gamma_{nbn,a} - \Gamma_{ab}^m \Gamma_{nnm} + \Gamma_{nb}^m \Gamma_{anm}\\
	&=-\tfrac{1}{2} S_{,ab} + \Gamma_{nb}^n \Gamma_{ann}\\
	&=-\tfrac{1}{2} S_{,ab} + \tfrac{1}{4} S^{-1} S_{,a} S_{,b} \qquad \mbox{ for $a, b=1,\dots, n-1$}. 
\end{align*}
We then obtain
\begin{align*}
	\mathrm{Ric}(\partial_n, \partial_n )&= -\tfrac{1}{2} \Delta' S + \tfrac{1}{4} S^{-1} |\nabla' S|^2
\end{align*}
where  $\Delta'$ and $\nabla' $ represent the Euclidean Laplacian and gradient, respectively, in the $(x^1,\dots, x^{n-1})$ variables.  Therefore, the scalar curvature of $g$ is given by
\begin{align*}
	R_g &= g^{nn} \mathrm{Ric}(\partial_n, \partial_n) +\sum_{a=1}^n g^{aa} \mathrm{Ric}(\partial_a, \partial_a)\\
	&=  g^{nn} \mathrm{Ric}(\partial_n, \partial_n) + \sum_{a=1}^n g^{aa} g^{nn} \mathrm{Rm}(\partial_a,\partial_n, \partial_n,  \partial_a)\\
	&=2g^{nn} \mathrm{Ric}(\partial_n, \partial_n)\\
	&=-S^{-1} \Delta' S + \tfrac{1}{2} S^{-2} |\nabla' S|^2.
\end{align*}

Meanwhile, from the computations in Lemma~\ref{le:pi}, we have
\begin{align*}
	|k|^2 &= \tfrac{1}{4} S^{-3} (S_{,n})^2 + \tfrac{1}{2} S^{-2} |\nabla' S|^2\\
	(\tr_g k)^2 &= \tfrac{1}{4} S^{-3} (S_{,n})^2.
\end{align*}
Together with the formula for $R_g$  above, we obtain
\begin{align*}
	\mu &= \tfrac{1}{2} (R_g - |k|_g^2 + (\tr_gk)^2)=-\tfrac{1}{2} S^{-1} \Delta' S.
\end{align*}
For $J$, we compute
\begin{align*}
	J^j &= \pi^{ij}_{;i} =\pi^{ij}_{,i} + \Gamma^{i}_{\ell i} \pi^{\ell j} + \Gamma^j_{\ell i} \pi^{\ell i}.
\end{align*}
and insert~\eqref{eqn:Christoffel} to obtain
\begin{align*}
	J^a &= \pi^{ia}_{,i}  + \Gamma^{i}_{\ell i} \pi^{\ell a} + \Gamma^a_{\ell i} \pi^{\ell i}\\
	&=\pi^{na}_{,n} + \pi^{ba}_{,b} + \Gamma^n_{n n} \pi^{n a}  + \Gamma^n_{bn} \pi^{b a} + \Gamma^a_{nn} \pi^{nn}\\
	&=\tfrac{1}{2}\big( S^{-\tfrac{3}{2} }S_{,a} \big)_{,n} -\tfrac{1}{2} \big(S^{-\tfrac{3}{2}} S_{,n} \big)_{,a} + \tfrac{1}{4} S^{-\tfrac{5}{2}} S_{,a} S_{,n}-\tfrac{1}{4} S^{-\tfrac{5}{2}} S_{,a} S_{,n}\\
	&=0\\
	J^n &=\pi^{in}_{,i} + \Gamma^i_{\ell i} \pi^{\ell n} + \Gamma^n_{\ell i} \pi^{\ell i}\\
	&=\pi^{an}_{,a} + \Gamma^n_{an} \pi^{an} + 2\Gamma^n_{bn} \pi^{bn} \\
	&=\pi^{an}_{,a} +3 \Gamma^n_{an} \pi^{an}  \\
	&=\tfrac{1}{2} \big(S^{-\tfrac{3}{2}} S_{,a}\big)_{,a} + \tfrac{3}{4}  S^{-\tfrac{5}{2} } \sum_{a=1}^{n-1}(S_{,a} )^2\\
	&=\tfrac{1}{2} S^{-\tfrac{3}{2}} \Delta' S.
\end{align*}

\end{proof}


\end{document}